\documentclass[11pt,reqno]{amsart}
\usepackage{amsmath}
\usepackage{amsthm}
\usepackage{graphicx}
\usepackage{amsfonts}
\usepackage{amssymb}
\usepackage{hyperref}
\usepackage{bm}
\usepackage{color}
\usepackage[margin=1.1in]{geometry}
\usepackage{enumerate}
\usepackage{mathrsfs}

\numberwithin{equation}{section}
\newtheorem{theorem}{Theorem}[section]
\newtheorem{lemma}[theorem]{Lemma}
\newtheorem{proposition}[theorem]{Proposition}
\newtheorem{corollary}[theorem]{Corollary}

\theoremstyle{definition}

\newtheorem{definition}[theorem]{Definition}
\newtheorem{remark}[theorem]{Remark}
\newtheorem{assumption}[theorem]{Assumption}

\newtheorem{conjecture}{Conjecture}
\newtheorem{problem}[conjecture]{Open Problem}

\newcommand{\lsi}[1]{\mathrm{LSI}(#1) }

\def\E{{\mathbb E}}
\def\R{{\mathbb R}}
\def\N{{\mathbb N}}
\def\FF{{\mathbb F}}

\def\PP{{\mathbb P}}

\def\P{{\mathcal P}}

\def\J{{\mathcal J}}

\def\W{{\mathcal W}}

\def\F{{\mathcal F}}

\def\C{{\mathcal C}}
\def\fx{a}
\def\fy{b}
\def\Var{\mathrm{Var}}
\def\Cov{\mathrm{Cov}}

\def\Enabla{\E^{\!\nabla}}
\def\const{r}
\def\testf{ h}

\title[Projected Langevin dynamics for entropic optimal transport]{Projected Langevin dynamics and a gradient flow for entropic optimal transport}

\author{Giovanni Conforti}
\address{Giovanni Conforti: Centre de Math\'ematiques Appliqu\'ees, \'Ecole Polytechnique. {Email: Giovanni.conforti@polutechnique.edu}}
\author{Daniel Lacker}
\address{Daniel Lacker: Department of Industrial Engineering \& Operations Research, Columbia University. {Email: daniel.lacker@columbia.edu}}
\thanks{G.C.\ acknowledges funding from the grant SPOT (ANR-20-CE40-0014). D.L.\ is partially supported by the NSF CAREER award DMS-2045328. S.P.\ gratefully acknowledges the support from NSF grants DMS-2134012, DMS-2133244 and DMS-2052239 and a PIMS PRN (Kantorovich Initiative). }
\author{Soumik Pal}
\address{Soumik Pal: Department of Mathematics, University of Washington. {Email: soumik@uw.edu}} 
\subjclass[2020]{49Q22, 60H30}
\keywords{optimal transport, entropy regularization, gradient flow, log-Sobolev inequalities, McKean-Vlasov diffusions}

\begin{document}

\begin{abstract}
The classical (overdamped) Langevin dynamics provide a natural algorithm for sampling from its invariant measure, which uniquely minimizes an energy functional over the space of probability measures, and which concentrates around the minimizer(s) of the associated potential when the noise parameter is small. We introduce analogous diffusion dynamics that sample from an  entropy-regularized optimal transport, which uniquely minimizes the same energy functional but constrained to the set $\Pi(\mu,\nu)$ of couplings of two given marginal probability measures $\mu$ and $\nu$ on $\R^d$, and which concentrates around the optimal transport coupling(s) for small regularization parameter. 
More specifically, our process satisfies two key properties: First, the law of the solution at each time stays in $\Pi(\mu,\nu)$ if it is initialized there. Second, the long-time limit is the unique solution of an entropic optimal transport problem.
In addition, we show by means of a new log-Sobolev-type inequality that the convergence holds exponentially fast, for sufficiently large regularization parameter and for a class of marginals which strictly includes all strongly log-concave measures. By studying the induced Wasserstein geometry of the submanifold $\Pi(\mu,\nu)$, we argue that the SDE can be viewed as a Wasserstein gradient flow on this space of couplings, at least when $d=1$, and we identify a conjectural gradient flow for $d \ge 2$. The main technical difficulties stems from the appearance of conditional expectation terms which serve to constrain the dynamics to $\Pi(\mu,\nu)$.
\end{abstract}

\maketitle

\tableofcontents

\section{Introduction}

Obtaining efficiently high-quality samples from probability measures that minimize a given energy functional is one of the fundamental problems in probability and statistics. A popular approach for its solution is to construct an appropriate dynamics, often taking the the form of either a stochastic process or of a gradient flow, that samples from the desired law in the large time limit. The emergence of optimal transport and its entropic regularization as a powerful and versatile tool for applications in machine learning and beyond naturally leads to consider the problem of sampling efficiently from probability measures that minimize a given energy (entropy) functional \emph{under constraints}. The goal of this work is to initiate the study of an important instance of this general problem by constructing and analyzing a natural stochastic process which is constrained to the space of couplings of two given marginals and which converges to the solution of an entropic optimal transport problem (a.k.a.\ Schr\"odinger bridge).

To set the stage, we first recall the setting of the classical optimal transport problem, which for given probability measures $\mu$ and $\nu$ on $\R^d$ and a given nonnegative continuous cost function $c : \R^{2d} \cong \R^d \times \R^d \to \R$, is given by
\begin{align}
\inf_{\pi \in \Pi(\mu,\nu)} \int_{\R^{2d}} c(x,y)\,\pi(dx,dy). \label{intro:OT}
\end{align}
Here $\Pi(\mu,\nu)$ denotes the set of couplings, probability measures on $R^d \times \R^d$ with first marginal $\mu$ and second marginal $\nu$.
A recently popular variant is the \emph{entropic optimal transport} problem which, for a regularization parameter $\epsilon > 0$, takes the form
\begin{align}
\inf_{\pi \in \Pi(\mu,\nu)} \bigg(\int_{\R^{2d}} c(x,y)\,\pi(dx,dy) + \epsilon H(\pi\,|\,\mu \otimes \nu)\bigg), \label{intro:EOT}
\end{align}
where $H$ is the usual relative entropy (or KL-divergence).
The regularized problem \eqref{intro:EOT} always admits a unique solution $\pi$,
because relative entropy is strictly convex and lower semicontinuous, and because $\Pi(\mu,\nu)$ is convex and weakly compact.
Moreover, the solution $\pi$ is also 
the unique element of $\Pi(\mu,\nu)$ taking the form
\begin{align}
\pi(dx,dy) = \exp\big((\varphi(x)+\psi(y) - c(x,y))/\epsilon \big)\mu(dx)\nu(dy), \label{intro:potentials}
\end{align}
for some Borel functions $\varphi$ and $\psi$, called the \emph{Schr\"odinger potentials}, which are a.s.\ unique up to an additive constant.
The regularized problem \eqref{intro:EOT} is widely studied for its computational advantages \cite{cuturi2013sinkhorn,PeyreCuturi}, the well known Sinkhorn algorithm providing a fast method for approximating its solution, and statistical advantages, such as a sample complexity that does not suffer from the curse of dimensionality \cite{genevay2019sample,mena2019statistical}. It also provides a convenient mathematical tool for studying \eqref{intro:OT} as the $\epsilon \to 0$ limit, furnishing alternative proofs of the HWI inequality \cite{gentil2020entropic} and Caffarelli's contraction theorem \cite{fathi2020proof,chewi2022entropic}, for instance.  Remarkably, when $c(x,y)=|x-y|^2/2$, solving \eqref{intro:EOT} is equivalent to optimizing the large deviations rate function in Sanov's Theorem for a system of independent Brownian particles. In this interpretation, problem \eqref{intro:EOT} is a sound mathematical formulation of an old question posed by E.\ Schr\"odinger in the seminal work \cite{Schr32} about the most likely evolution of a cloud of independent Brownian particles conditionally to the observation of their configuration at two consecutive times. For this reason, \eqref{intro:EOT} is also known as the Schr\"odinger problem and its solutions as Schr\"odinger bridges. The Schr\"odinger problem has been studied for several decades independently of its relation with optimal transport, 
and we refer to the survey articles \cite{LeoSch} and \cite{chen2021stochastic} for further discussion of the history and applications.

In this paper, we introduce stochastic dynamics which converge to the unique solution $\pi$ of \eqref{intro:EOT}, in the case where $(\mu,\nu)$ take the form $\mu(dx)=e^{-U(x)}dx$ and $\nu(dy)=e^{-V(y)}dy$.
The dynamics are governed by the following stochastic differential equation (SDE):
\begin{align}
\begin{split}
dX_t &= \Big( \E[\nabla_x c(X_t,Y_t)\,|\,X_t]  -  \nabla_x c(X_t,Y_t) - \epsilon \nabla U(X_t)\Big)dt + \sqrt{2\epsilon }\,dW_t, \\
dY_t &= \Big(  \E[\nabla_y c(X_t,Y_t)\,|\,Y_t]  -  \nabla_y c(X_t,Y_t) - \epsilon \nabla V(Y_t)\Big)dt + \sqrt{2\epsilon }\,dB_t,
\end{split} \label{def:mainSDE}
\end{align}
where $W$ and $B$ are independent $d$-dimensional Brownian motions.
We term this SDE the \emph{projected Langevin dynamics} in light of the conditional expectation term, which serves, as we will see, to preserve the space of couplings $\Pi(\mu,\nu)$.
Our first main result is the following:

\begin{theorem} \label{th:main}
Suppose the functions $(c,U,V)$ are twice continuously differentiable, such that $\mu(dx)=e^{-U(x)}dx$ and $\nu(dy)=e^{-V(y)}dy$ are probability measures. Assume:
\begin{itemize}
\item Subgaussian marginals: $\int_{\R^d} e^{\delta|x|^2}(e^{-U(x)}+e^{-V(x)})dx < \infty$ for some $\delta > 0$. 
\item $c$ is nonnegative, and $\nabla^2 c$ is bounded.
\item $\nabla^2 U$ and $\nabla^2 V$ are bounded from below in  semidefinite order.
\item $|\nabla U| \in L^p(\mu)$ and $|\nabla V| \in L^p(\nu)$ for some $p > 2(d+1)$, and $U \in L^1(\mu)$ and $V \in L^1(\nu)$.
\end{itemize}
Then, for any $P_0 \in \Pi(\mu,\nu)$, the following hold:
\begin{enumerate}
\item There exists a unique in law weak solution of \eqref{def:mainSDE} starting from $(X_0,Y_0) \sim P_0$.
\item The law $P_t$ of $(X_t,Y_t)$ satisfies $P_t \in \Pi(\mu,\nu)$ for each $t \ge 0$.
\item As $t\to\infty$ we have $P_t \to \pi$ in $2$-Wasserstein distance, and $H(P_t\,|\,\pi) \to 0$.
\item As $t\to\infty$ we have 
\begin{align*}
\int_{\R^d}\big|\E[\nabla_x c(X_t,Y_t)\,|\,X_t=x] - \nabla\varphi(x)\big|^2\,\mu(dx) &\to 0, \\
\int_{\R^d}\big|\E[\nabla_y c(X_t,Y_t)\,|\,Y_t=y] - \nabla\psi(y)\big|^2\,\nu(dy) &\to 0.
\end{align*}
\end{enumerate}
\end{theorem}

The proofs of (1,2) and (3,4) are given in Sections \ref{se:wellposedness} and \ref{se:longtime}, respectively. See Definition \ref{def:weaksolution} for a precise definition of a (probabilistic) weak solution of \eqref{def:mainSDE}.

A corollary of Theorem \ref{th:main} is that $\pi$ itself is the unique invariant measure for the dynamics \eqref{def:mainSDE}. The fact that $\pi$ is invariant is a consequence of the known identities (see Section \ref{se:differentiabilitypotentials})
\begin{equation}
\nabla\varphi(x)=\E_{\pi}[\nabla_xc(X,Y)\,|\,X=x], \qquad \nabla\psi(y)=\E_{\pi}[\nabla_yc(X,Y)\,|\,Y=y], \label{intro:pi-condexp}
\end{equation}
for  sufficiently regular $(c,U,V)$, where $(X,Y) \sim \pi$ in the expectations.  Indeed, these formulas follow by differentiating the so-called \emph{Schr\"odinger equations},
\begin{equation}
\begin{split}
1 &= \int_{\R^d} e^{ (\varphi(x)+\psi(y) - c(x,y))/\epsilon} \nu(dy), \ \ \text{for all } x, \\
1 &= \int_{\R^d} e^{ (\varphi(x)+\psi(y) - c(x,y))/\epsilon} \mu(dx), \ \ \text{for all } y,
\end{split} \label{eq:Schrsyst}
\end{equation}
which are themselves consequences of $\pi$ having the form \eqref{intro:potentials} and marginals $(\mu,\nu)$.

The fact that $\Pi(\mu,\nu)$ is invariant for the dynamics \eqref{def:mainSDE}, in the sense of Theorem \ref{th:main}(2), can be deduced from the so-called \emph{mimicking theorem} from stochastic analysis \cite[Corollary 3.7]{BrunickShreve}. For a direct derivation, note by It\^o's formula and the tower property that a solution of \eqref{def:mainSDE} must satisfy
\begin{align}
\frac{d}{dt}\E[\testf(X_t)] &= \E\Big[ \nabla\testf(X_t) \cdot \Big( \E[\nabla_x c(X_t,Y_t)\,|\,X_t]  -  \nabla_x c(X_t,Y_t) - \epsilon \nabla U(X_t)\Big) + \epsilon \Delta\testf(X_t) \Big] \nonumber \\
	&= \E\Big[ - \epsilon \nabla\testf(X_t) \cdot \nabla U(X_t) + \epsilon \Delta\testf(X_t) \Big], \label{intro:marginalcalculation}
\end{align}
for nice test functions $\testf$.
This shows that the marginal law $\rho_t$ of $X_t$ is a (weak) solution of the Fokker-Planck equation $\partial_t\rho = \epsilon\mathrm{div}(\rho\nabla U) + \epsilon\Delta \rho$, for which $\rho_t \equiv \mu$ is the time-invariant solution.

The rest of this introduction will describe an analogy between \eqref{def:mainSDE} and the classical Langevin dynamics (Section \ref{se:intro:Langevin}), our results on the exponential convergence $P_t \to \pi$ (Section \ref{se:exprate}),  some special cases and extensions (Section \ref{se:extensions}), and related literature (Section \ref{se:literature}).
Then, in Section \ref{se:gradientflow} we explain the geometric perspective, in the language of Otto calculus; namely, we describe the geometry of the (non-geodesically convex!) submanifold $\Pi(\mu,\nu)$ of Wasserstein space and how the flow $(P_t)$ of Theorem \ref{th:main} can be seen as the steepest descent (or gradient flow) for the functinoal $H(\cdot\,|\,\pi)$ in this submanifold, at least when $d=1$.

\subsection{The analogy with Langevin dynamics} \label{se:intro:Langevin}
The dynamics \eqref{def:mainSDE} are perhaps best understood in analogy with the classical (overdamped) Langevin dynamics.
Consider the global minimization problem
\begin{equation}
\inf_{(x,y) \in \R^{2d}} c(x,y) = \inf_{\rho \in \P(\R^{2d})}\int_{\R^{2d}}c\,d\rho, \label{intro:globalminimization}
\end{equation}
where $\P(\R^{2d})$ denotes the set of probability measures on $\R^{2d}$. The right-hand side of \eqref{intro:globalminimization} is a trivial rewriting of the left which is useful for our analogy, as it exhibits the problem as the unconstrained version of the optimal transport problem \eqref{intro:OT}.
Consider then the (unconstrained) entropic regularization:
\begin{align}
\inf_{\rho \in \P(\R^{2d})} \bigg(\int_{\R^{2d}}c\,d\rho + \epsilon H(\rho\,|\,\mu \otimes \nu)\bigg). \label{intro:globalminimization-entropy}
\end{align}
This problem always admits the unique optimizer $\rho_*(dx,dy) \propto e^{-c(x,y)/\epsilon}\mu(dx)\nu(dy)$, even when the minimization problem \eqref{intro:globalminimization} may have multiple solutions. As $\epsilon \downarrow 0$, the measure $\rho_*$ is increasingly concentrated on the set of global minimizers of $c$. If the minimizer of $c$ is unique, then $\rho_*$ converges to it; otherwise, every limit point of $\rho_*$ is a minimizer for the right-hand side of \eqref{intro:globalminimization}. Analogously, the entropic optimal transport problem \eqref{intro:EOT} always admits the unique optimizer $\pi$, and the limit points of $\pi$ as $\epsilon\to 0$ are always optimizers for \eqref{intro:OT}.

The Langevin dynamics corresponding to \eqref{intro:globalminimization-entropy} take the form of the SDE
\begin{align}
\begin{split}
dX'_t &= - \big(\nabla_x c(X'_t,Y'_t) + \epsilon\nabla U(X'_t)\big)\,dt + \sqrt{2\epsilon }\,dW_t \\
dY'_t &= - \big(\nabla_y c(X'_t,Y'_t) + \epsilon\nabla V(Y'_t)\big)\,dt + \sqrt{2\epsilon }\,dB_t.
\end{split} \label{intro:Langevin}
\end{align}
It is well known that this SDE defines a Markov process with $\rho_*$ as its unique invariant measure. 
In this sense, the Langevin dynamics ``sample" from the measure $\rho_*$, which is concentrated on the set of optimizers of \eqref{intro:globalminimization} for $\epsilon$ small. The Langevin dynamics thus yield a sampling (or Monte Carlo) method for approximately solving the optimization problem \eqref{intro:globalminimization}. In perfect analogy, our projected Langevin dynamics \eqref{def:mainSDE} sample from approximate solutions of the optimal transport problem \eqref{intro:OT}.

To take the analogy a step further, let us examine the $\epsilon=0$ version of the equations. Note that any critical point of $c$ is stationary for the noiseless dynamics
\begin{equation}
\frac{d}{dt}(X'_t,Y'_t) = - \nabla c(X'_t,Y'_t) . \label{intro:Langevin-noiseless}
\end{equation}
That is, when $\epsilon=0$, there is no longer a unique invariant measure, and there are potentially many additional unstable equilibrium points.
The same is true for the natural $\epsilon=0$ version of our SDE \eqref{def:mainSDE}, which is
\begin{align}
\begin{split}
dX_t &= \Big( \E[\nabla_x c(X_t,Y_t)\,|\,X_t]  -  \nabla_x c(X_t,Y_t) \Big)dt , \\
dY_t &= \Big(  \E[\nabla_y c(X_t,Y_t)\,|\,Y_t]  -  \nabla_y c(X_t,Y_t) \Big)dt .
\end{split} \label{def:mainSDE-noiseless}
\end{align}
It is easy to see that \emph{any invertible Monge coupling} is stationary for \eqref{def:mainSDE-noiseless}. That is, suppose a random vector $(X_0,Y_0)$ is supported on the graph of a measurable bijection, so that $X_0$ and $Y_0$ generate the same $\sigma$-algebra. Setting $(X_t,Y_t)=(X_0,Y_0)$ for all $t > 0$ defines a solution of \eqref{def:mainSDE-noiseless}, because $\E[\nabla_x c(X_t,Y_t)\,|\,X_t]  =  \nabla_x c(X_t,Y_t)$ and $\E[\nabla_y c(X_t,Y_t)\,|\,Y_t] =  \nabla_y c(X_t,Y_t)$.

We will not rigorously address the $\epsilon\to 0$ limit of our SDE \eqref{def:mainSDE},  nor the well-posedness of the noiseless equation \eqref{def:mainSDE-noiseless}, both of which appear to be quite delicate:

\begin{problem}
When does the noiseless equation \eqref{def:mainSDE-noiseless} admit a unique (weak) solution? Are there invariant measures which are not bijective Monge couplings? Do the dynamics \eqref{def:mainSDE} converge to a solution of \eqref{def:mainSDE-noiseless} as $\epsilon \to 0$, and is there a corresponding (Freidlin-Wentzell) large deviations principle?
\end{problem}

Perhaps the most important special case of the optimal transport problem \eqref{intro:OT} arises from the quadratic cost $c(x,y)=|x-y|^2/2$, for which the projected Langevin dynamics \eqref{def:mainSDE} become
\begin{align}
\begin{split}
dX_t &= \big( Y_t - \E[ Y_t \,|\,X_t] - \epsilon \nabla U(X_t)\big)dt + \sqrt{2\epsilon }\,dW_t, \\
dY_t &= \big( X_t - \E[ X_t \,|\,Y_t]  - \epsilon \nabla V(Y_t)\big)dt + \sqrt{2\epsilon }\,dB_t.
\end{split} \label{def:mainSDE-quadratic}
\end{align}
It is instructive to understand these dynamics in connection with the famous result of Brenier \cite{brenier1991polar}:
Because $\mu$ is absolutely continuous, there is a unique optimal coupling $\pi_0$ for \eqref{intro:OT}, and it is supported on the graph of the gradient of a convex function; this gradient is called \emph{the Brenier map}.
For $\epsilon > 0$, the solution $\pi$ of \eqref{intro:EOT} is absolutely continuous with respect to $\mu \otimes \nu$ and thus also with respect to Lebesgue measure on $\R^{2d}$. As $\epsilon \to 0$, this diffuse measure $\pi$ converges \cite[Theorem 5.10]{nutz2021introduction} to the degenerate measure $\pi_0$ which is supported on the graph of a function. This behavior is reflected in the dynamics \eqref{def:mainSDE-quadratic}: For small $\epsilon$, the magnitude $|Y_t-\E[Y_t|X_t]|$ of the dominant term in the drift of $X_t$ measures how close $Y_t$ is to being a deterministic function of $X_t$. Hence, in the $\epsilon \to 0$ limit, the only way for the dynamics to stabilize is for $Y_t$ to become $X_t$-measurable (and vice versa).

\subsection{Exponential convergence, energy decay and a new logarithmic Sobolev inequality} \label{se:exprate}

In this section, we specialize to the case of quadratic cost $c(x,y)=|x-y|^2/2$, and we  address the exponential convergence properties of \eqref{def:mainSDE-quadratic}. We are unable to apply the general theory of gradient flows on metric spaces developed in \cite{AGSbook}, for the key reason that $\Pi(\mu,\nu)$ fails to be geodesically convex in the Wasserstein space $(\P_2(\R^{2d}),\W_2)$, which we explain further in Section \ref{se:geodesicconvexity}. Instead, we adopt the direct and natural strategy of analyzing the behavior of the ``energy" functional 
\begin{equation}
\J(P) := \int c(x,y)\,P(dx,dy) + \epsilon H(P\,|\,\mu \otimes \nu), \label{def:Jfunctional}
\end{equation}
along the flow $(P_t)_{t \ge 0}$. In fact, a simple computation shows that $\J(P)-\J(\pi)=\epsilon H(P\,|\,\pi)$ for $P \in \Pi(\mu,\nu)$; this is an instance of the Pythagorean theorem for entropic projections \cite{csiszar1975divergence}. We prove in Lemma \ref{le:entropydynamics} the energy dissipation identity (technically, just inequality)
\begin{equation}
\frac{d}{dt}\J(P_t) = \epsilon\frac{d}{dt}H(P_t\,|\,\pi) = - \epsilon^2 \overline{I}(P_t\,|\,\pi), \label{intro:Jdecreases}
\end{equation}
where, for $P\in\Pi(\mu,\nu)$ and $R = \log dP/d\pi$, we define
\begin{equation}
\overline{I}(P\,|\,\pi) := \E_{P}\Big[  \big|\nabla_x R - \E_P[\nabla_x R\,|\,X]\big|^2 + \big|\nabla_y R - \E_P[\nabla_y R\,|\,Y]\big|^2 \Big] \label{intro:Ibar}
\end{equation}
if the weak gradient $\nabla R$ exists and belongs to $L^2(P)$, and otherwise $\overline{I}(P\,|\,\pi):=\infty$. This functional might be regarded as a projected form of the classical (relative) Fisher information, and it plays a critical role in both the qualitative and quantitative convergence analysis.

A key technical step in the qualitative convergence proof is to show (in Proposition \ref{pr:Istability}) that $\overline{I}(P_n\,|\,\pi) \to 0$ implies $P_n \to \pi$ weakly, if $P_n \in \Pi(\mu,\nu)$. This in turn relies on a strong compactness argument, which uses a Sobolev embedding in conjunction with the remarkable fact (Lemma \ref{le:Df}) that the derivative of $x\mapsto \E_P[\nabla_x R\,|\,X=x]$ can be bounded in $L^1(\mu)$-norm by $C(1+\overline{I}(P\,|\,\pi))$, with $C$ not depending on the choice of $P\in \Pi(\mu,\nu)$. 

To prove an exponential rate of convergence, we naturally try to bound $\overline{I}(\cdot\,|\,\pi)$ from below in terms of $H(\cdot\,|\,\pi)$.  This leads naturally to consider the following functional inequality of independent interest:
\begin{equation}
H(P\,|\,\pi) \le \frac{\epsilon}{\const}\overline{I}(P\,|\,\pi), \qquad \forall P \in \Pi(\mu,\nu). \label{intro:newLSI} 
\end{equation}
We call this a \emph{projected logarithmic Sobolev inequality}.
Recall that we say that $\pi$ satisfies the \emph{logarithmic Sobolev inequality} with constant $\kappa > 0$, or $\lsi{\kappa}$, if 
\begin{equation}
H(P\,|\,\pi) \le \kappa^{-1} I(P\,|\,\pi), \qquad \forall P \in \P(\R^{2d}), \label{intro:LSI}\tag{$\lsi{\kappa}$}
\end{equation}
where $I(P\,|\,\pi) := \|\nabla \log dP/d\pi\|_{L^2(P)}^2$ is the usual relative Fisher information (defined to be $\infty$ if the gradient does not exist in $L^2(P)$).
The projected LSI differs from the usual LSI in two ways. On the one hand, it is less restrictive in the sense that it is only required to hold for $P \in \Pi(\mu,\nu)$, not all $P \in \P(\R^{2d})$. On the other hand, it is more restrictive because trivially $\overline{I} \le I$.
Hence, the projected LSI can be viewed as a lower bound on the \emph{deficit} of the LSI, i.e., the amount by which the LSI can be strengthened using additional information on $P$. 
See \cite{fathi2016quantitative,eldan2020stability} for recent general studies on the stability of the Gaussian LSI, i.e., lower bounds on the deficit, although their results are not suitable for our specific situation.

Our next result gives a sufficient condition for the projected LSI \eqref{intro:newLSI}. Roughly speaking, it says that if the conditional distributions of $\pi$ satisfy LSI with a large enough constant, then \eqref{intro:newLSI} holds. We make use of the  disintegrations of $\pi$ with respect to its marginals $(\mu,\nu)$:
\begin{equation*}
\pi(dx,dy)=\mu(dx)\pi^x(dy), \qquad \pi(dx,dy)=\nu(dy)\pi^y(dx).
\end{equation*}
This notation is somewhat abusive, but we will take care to not conflate the symbols $x$ and $y$.
The proof of the following is given in Section \ref{se:newLSIproof-suff}.

\begin{theorem}\label{thm:suff_cond_new_LSI}
   Assume that there exist constants $\kappa^{X|Y}$ and $\kappa^{Y|X}$ such that
   \begin{enumerate}[(i)] 
   \item The conditional measure $\pi^{y}$ satisfies $\lsi{\kappa^{X|Y}}$, uniformly in $y\in\R^d$.
   \item The conditional measure $\pi^{x}$ satisfies $\lsi{\kappa^{Y|X}}$, uniformly in $x\in\R^d$.
   \end{enumerate}
   If, moreover,
\begin{equation}\label{eq:exp_conf_suff_cond}
(\kappa^{X|Y}\kappa^{Y|X})^{1/2}\epsilon > 2       ,
\end{equation}
then the projected LSI \eqref{intro:newLSI} holds with 
\begin{equation}\label{intro:newLSI_r}
r =\epsilon(\kappa^{X|Y}+\kappa^{Y|X})\bigg(1-\frac{4}{\epsilon^2 \kappa^{X|Y}\kappa^{Y|X}}\bigg).
\end{equation}
As a consequence, for any initialization $P\in\Pi(\mu,\nu)$ of the SDE \eqref{def:mainSDE-quadratic} we have
\begin{equation}\label{eq:exp_decay_entropy}
H(P_t\,|\,\pi) \le e^{-  r t} H(P_0\,|\,\pi), \ \text{ for } t \ge 0.
\end{equation}
Moreover, we also have
\begin{equation}\label{eq:exp_conv_drift}
\begin{split}
\int_{\mathbb{R}^d}\,\big|\,\mathbb{E}[X_t-Y_t|X_t=x]-\nabla\varphi(x)\big|^2\, \mu(dx) \leq \frac{4}{\kappa^{Y|X}} e^{-  r t} H(P_0\,|\,\pi),\\
\int_{\mathbb{R}^d} \,\big|\,\mathbb{E}[Y_t-X_t|Y_t=y]-\nabla\psi(y)\big|^2\, \nu(dy) \leq \frac{4}{\kappa^{X|Y}} e^{-  r t} H(P_0\,|\,\pi).
\end{split}
\end{equation}
\end{theorem}

We next present two sets of assumptions on the marginals $(\mu,\nu)$ under which the sufficient condition \eqref{eq:exp_conf_suff_cond} is met. First, we consider strongly log-concave marginals, for which the results are cleanest. Secondly, we weaken the strong log-concavity assumption and ask only that $U,V$ have an asymptotically positive integrated convexity profile, defined below. 

\subsubsection{Log-concave marginals}

We adopt the convention that $\beta^{-1}=0$ when $\beta=\infty$. The proof of the following corollary, given in Section \ref{se:logconcave-proofs}, relies on the Bakry-\'Emery criterion along with recent convexity bounds on Schr\"odinger potentials due to \cite{chewi2022entropic}. 

\begin{corollary}\label{intro:corollary_log_conc}
Let $c(x,y)=|x-y|^2/2$. Assume that $U$ and $V$ are twice continuously differentiable and that there exist $\alpha_U,\alpha_V\in(0,\infty)$, $\beta_U,\beta_V\in(0,\infty]$ such that 
\begin{equation*}
\alpha_U\mathrm{I} \preceq    \nabla^2 U(x) \preceq \beta_{U}\mathrm{I}, \quad \alpha_V\mathrm{I} \preceq   \nabla^2 V(y) \preceq \beta_{V}\mathrm{I}, \qquad \text{for all } x,y \in \R^d,
\end{equation*}
in semidefinite order. Then, if 
\begin{equation}\label{eq:log_conc_suff_cond} 
\epsilon >  \frac{\alpha^{-1}_U\alpha^{-1}_V-\beta^{-1}_U\beta^{-1}_V}{(\alpha^{-1}_U+\beta^{-1}_U)^{1/2}(\alpha^{-1}_V+\beta^{-1}_V)^{1/2}} =: \epsilon_c,
\end{equation}
the projected LSI \eqref{intro:newLSI} holds with $r$ as in \eqref{intro:newLSI_r} and with
\begin{equation}
    \begin{split}
            \kappa^{X|Y}=\sqrt{4\alpha_U/(\epsilon^2\beta_V)+\alpha^2_U}+\alpha_U , \\
            \kappa^{Y|X}=\sqrt{4\alpha_V/(\epsilon^2\beta_U)+\alpha^2_V}+\alpha_V .
    \end{split} \label{def:LSI-logco-kappa}
\end{equation}
In the case when $\alpha_U=\alpha_V=\alpha$ and $\beta_U=\beta_V=\beta,$ \eqref{eq:log_conc_suff_cond} takes the simpler form
\begin{equation*}
\epsilon>\alpha^{-1}-\beta^{-1}.
\end{equation*}
Furthermore, the exponential $L^2$-convergence \eqref{eq:exp_conv_drift} of the conditional expectations $\mathbb{E}[X_t-Y_t\,|\,X_t]$ and $\mathbb{E}[Y_t-X_t\,|\,Y_t]$ to the gradients of the Schr\"odinger potentials also holds.
\end{corollary}

The assumed upper bounds on $\nabla^2U$ and $\nabla^2V$ automatically hold with $\beta_U=\beta_V=\infty$, in which case \eqref{eq:log_conc_suff_cond} and \eqref{def:LSI-logco-kappa} become
\begin{equation}
\epsilon > (\alpha_U\alpha_V)^{-1/2} , \quad \text{and} \quad \kappa^{X|Y}=2\alpha_U, \ \ \kappa^{Y|X}=2\alpha_V. \label{eq:epsilon-restrictive}
\end{equation}
But a finite value of $\beta_U$ or $\beta_V$ permits sharper results, because the constants $(\kappa^{X|Y},\kappa^{Y|X})$ become larger and the condition \eqref{eq:log_conc_suff_cond}  less restrictive.
A noteworthy special case is when the marginals $(\mu,\nu)$ are Gaussian, so that $\alpha_U=\beta_U$ and $\alpha_V=\beta_V$. Then $\epsilon_c=0$, and as $\epsilon \to 0$ the exponent $r$ tends to zero like $\epsilon(\alpha_U+\alpha_V)/2$ at first order.
It is only in the Gaussian case that we are able to show exponential convergence \emph{for all values of} $\epsilon > 0$; when either marginal is non-Gaussian, we have a critical level $\epsilon_c > 0$ below which exponential convergence is not guaranteed. We do not know if this is a fundamental obstruction or a byproduct of our proof.

\begin{problem}
When $\epsilon \le \epsilon_c$, what is the rate of convergence of $H(P_t\,|\,\pi) \to 0$?
\end{problem}

It is worth noting that, for unbounded cost functions, the only known results on the exponential convergence of Sinkhorn's algorithm also require a sufficiently large regularization parameter $\epsilon$ \cite{conforti2023convergence}.

\subsubsection{Asymptotically log-concave marginals}

This section summarizes a significant extension of Corollary \ref{intro:corollary_log_conc} beyond the log-concave setting. 
We introduce the \emph{integrated convexity profile}
   $\kappa_U : \R_+ \to \R \cup \{-\infty\}$ of a function $U \in\mathcal{C}^1(\R^d;\R)$ as follows
\begin{equation}\label{eq:kappa_def}\kappa_U(r):= \inf\biggl\{\frac{\langle\nabla U(x)-\nabla U(y),\,x-y\rangle}{|x-y|^2}:\quad |x-y|=r\biggr\}\,.
\end{equation}
Referring to $\kappa_U$ as to the integrated convexity profile of $U$ is justified by the observation that the condition $ \kappa_U(r) \geq \alpha $ is equivalent to 
\begin{equation*}
    \int_{0}^{1} \big\langle  y-x ,\nabla^2 U ((1-\lambda)x+ \lambda y) (y-x)\big\rangle\,d\lambda \geq \alpha |x-y|^2, \quad \forall \, x,y\in \R^d, |x-y|=r.
\end{equation*}
In words, $\kappa_U(r)\geq \alpha$ means that integrating the second directional derivative of $U$ along a segment of length $r$ yields the same lower bound that one would obtain for a $\alpha$-(semi) convex function. We remark that certain lower bounds on $\kappa_U$ are known to imply the exponential trend to equilibrium for the overdamped Langevin dynamics with drift $-\nabla U$, see \cite{eberle2016reflection} for example. The main result of this section is that if $U,V$ have an asymptotically positive integrated convexity profile, i.e. satisfy the assumptions \eqref{A1_conv_prof}-\eqref{A3_conv_prof} of Theorem \ref{thm:conv_prof_informal} below, then the desired inequality \eqref{intro:newLSI} holds with a positive constant, and thus the SDE \eqref{def:mainSDE-quadratic} converges exponentially in relative entropy. 
The following theorem summarizes Theorem \ref{thm:conv_prof_rigorous} stated in Section \ref{sec:proof_conc_prof}, where the reader can find full details and the explicit but complicated  constants.

\begin{theorem}[Summary of Theorem \ref{thm:conv_prof_rigorous}]\label{thm:conv_prof_informal}
Let $c(x,y)=|x-y|^2/2$. Assume that  for $W=U,V$ the following hold.
\begin{enumerate}[(i)]
    \item\label{A1_conv_prof} There exist $\alpha_W\in(0,\infty)$and $R_W,L_W\in [0,\infty)$ such that 
\begin{equation}\label{def:as_conv_prof}
\kappa_{W}(r)\geq \begin{cases}\alpha_W &\quad \mbox{for $r> R_W$}\\ \alpha_W-L_W & \quad \mbox{for $r\leq R_W$}.\end{cases} 
\end{equation}
\item\label{A3_conv_prof} There exist $\beta_W\in (0,\infty]$ such that for all $x\in\R^d$, in semidefinite order,
\begin{equation*}
    \nabla^2 W(x) \preceq \beta_{W} \mathrm{I}.
\end{equation*}
\end{enumerate}
    Then, there exist a strictly positive function $\kappa$ on $(0,\infty)^2\times[0,\infty)^2\times(0,\infty]$ such that 
    \begin{enumerate}[(1)]
    \item $\kappa=\kappa(\epsilon,\alpha,L,R,\beta)$ is increasing in $\alpha,\epsilon$ and decreasing in $L,R,\beta$. Moreover, for any fixed $(\alpha,L,R,\beta)$, we have $\liminf_{\epsilon\rightarrow\infty} \kappa(\epsilon,\alpha,L,R,\beta)>0$.
    \item  Assumptions (i) and (ii) of Theorem \ref{thm:suff_cond_new_LSI} hold with
    \begin{equation}\label{eq:kappa_cond_conv_prof}
        \kappa^{X|Y} = \kappa(\epsilon,\alpha_U,L_U,R_U,\beta_V),\quad         \kappa^{Y|X} = \kappa(\epsilon,\alpha_V,L_V,R_V,\beta_U).
    \end{equation}
    \end{enumerate}
As a corollary, the condition \eqref{eq:exp_conf_suff_cond} holds for all sufficiently large $\epsilon$, and thus so does the projected LSI \eqref{intro:newLSI}, with constant $r$ given by \eqref{intro:newLSI_r} and with $\kappa^{X|Y},\kappa^{Y|X}$ given by \eqref{eq:kappa_cond_conv_prof}. Furthermore, the exponential $L^2$-convergence \eqref{eq:exp_conv_drift} of the conditional expectations $\mathbb{E}[X_t-Y_t\,|\,X_t]$ and $\mathbb{E}[Y_t-X_t\,|\,Y_t]$ to the gradients of the Schr\"odinger potentials also holds.
\end{theorem}

If $W$ is the sum of a strongly convex function and a Lipschitz function with second derivative bounded below, then $W$ satisfies the assumption \eqref{def:as_conv_prof}. This reveals an important class of examples of marginals $(\mu,\nu)$ for which  Theorem \ref{thm:conv_prof_informal}  applies but Corollary \ref{intro:corollary_log_conc} does not.

The qualitative behavior of $\kappa$ with respect to its arguments $\alpha,R,L$ reflects the intuition that stronger lower bounds on the convexity profile of $U$ (resp. $V$) translate into better values of $\kappa^{X|Y}$(resp. $\kappa^{Y|X}$). The fact that $\kappa$ is decreasing in $\beta$ can be interpreted as the fact that the more convex is $V$, the worse is the LSI constant of the conditional distributions $\pi^y$. The dependence of $\kappa$ in $\epsilon$ can be explained noting that as $\epsilon$ grows (the ``high temperature" regime), the optimal coupling $\pi$ approaches the product measure $\mu \otimes \nu$, and LSI constants for product measures are notoriously easier to bound. 
The explicit expressions of $\kappa^{X|Y}$ and $\kappa^{Y|X}$ are given at \eqref{eq:conditional_kappa_1} and \eqref{eq:conditional_kappa_2} below and require additional definitions that we prefer to omit from this introduction.

The findings leading to Theorem \ref{thm:conv_prof_informal} and the more general Theorem \ref{thm:conv_prof_rigorous} have a broader scope than establishing the exponential trend to equilibrium for the projected Langevin dynamics \eqref{def:mainSDE}.
Along the way we prove a new result of independent interest in Theorem \ref{th:newLSI-kappa}: Any probability density $\rho$ with $-\log\rho$ having an asymptotically positive integrated convexity profile in the sense of \eqref{def:as_conv_prof}  can be realized as the push-forward of the Gaussian measure through a Lipschitz map. It is well known that this implies that $\rho$ satisfies the LSI and is in fact considerably stronger: for example, it implies eigenvalues comparison \cite[Section 1.1]{mikulincer2022lipschitz}. There has been significant progress in recent years in showing that suitable perturbations of log-concave probability measures can be obtained as Lipschitz push-forwards of the Gaussian \cite{mikulincer2021brownian,mikulincer2022lipschitz,fathi2023transportation}. As in some of these works, our proof is based on an analysis of the so-called heat flow maps introduced by Kim and Milman in \cite{kim2012generalization}. Our results also improve
on a recent result of the first author \cite{conforti2022weak}, which established lower bounds on the  LSI constant for the conditional distribution of $\pi$ under similar assumptions, but with constants exhibiting worse dependence on $\epsilon$. Notably, our $\kappa^{X|Y}$ is increasing as a function of $\epsilon$, whereas the lower bound on the LSI constant in \cite{conforti2022weak} vanishes as $\epsilon$ diverges.

\subsection{Special cases and extensions} \label{se:extensions}

\subsubsection{Quadratic cost and Gaussian marginals}

When the functions $(c,U,V)$ are quadratic, the SDE \eqref{def:mainSDE} admits a fairly explicit solution. 
We record this here both because it is an instructive tractable example and because of its potential interest in applied contexts. Indeed, the submanifold of $\P_2(\R^d)$ consisting of Gaussian measures can be identified with the space of positive definite matrices, equipped with a particular Reimannian geometry sometimes known as the \emph{Bures-Wasserstein geometry}; see \cite{bhatia2019bures} for details of this geometry and \cite{lambert2022variational} for an analysis of a gradient flow on this submanifold and its relevance to variatonal inference in statistics.

Let $\mu$ and $\nu$ be the centered Gaussian measures with covariance matrices $\Sigma_\mu$ and $\Sigma_\nu$, which are given  $d \times d$ positive definite matrices. The SDE \eqref{def:mainSDE-quadratic} corresponding to the quadratic cost function becomes
\begin{align}
\begin{split}
dX_t &= \big( Y_t - \E[Y_t\,|\,X_t]  - \epsilon \Sigma_\mu^{-1} X_t\big)dt + \sqrt{2\epsilon }\,dW_t, \\
dY_t &= \big( X_t - \E[X_t\,|\,Y_t]  - \epsilon \Sigma_\nu^{-1} Y_t\big)dt + \sqrt{2\epsilon }\,dB_t.
\end{split} \label{def:SDE-Gaussian}
\end{align}
Suppose we make an ansatz, that the law of $(X_t,Y_t)$ is a nondegenerate Gaussian measure belonging to $\Pi(\mu,\nu)$. Then, necessarily, $\E[X_tX_t^\top] = \Sigma_\mu$ and $\E[Y_tY_t^\top] = \Sigma_\nu$ for all $t \ge 0$. Let $\Sigma_t = \E[X_tY_t^\top]$.
A well known property of Gaussian measures yields the following expressions for the conditional expectations:
\begin{align*}
\E[Y_t\,|\,X_t] = \Sigma_t^\top \Sigma_\mu^{-1}X_t, \qquad \E[X_t\,|\,Y_t] = \Sigma_t \Sigma_\nu^{-1}Y_t.
\end{align*}
The SDE \eqref{def:SDE-Gaussian} then becomes
\begin{align*}
\begin{split}
dX_t &= \big( Y_t -  (\Sigma_t^\top + \epsilon I) \Sigma_\mu^{-1} X_t\big)dt + \sqrt{2\epsilon }\,dW_t, \\
dY_t &= \big( X_t - (\Sigma_t + \epsilon I) \Sigma_\nu^{-1} Y_t\big)dt + \sqrt{2\epsilon }\,dB_t.
\end{split}
\end{align*}
To determine the dynamics of $\Sigma_t$, we use It\^o's formula:
\begin{align}
\frac{d}{dt} \Sigma_t &= \E\big[ X_t \big( X_t - (\Sigma_t + \epsilon I) \Sigma_\nu^{-1} Y_t\big)^\top + \big( Y_t - (\Sigma_t^\top + \epsilon I) \Sigma_\mu^{-1} X_t\big)Y_t^\top \big] \nonumber \\
	&= \Sigma_\mu + \Sigma_\nu - (\Sigma_t^\top + \epsilon I) \Sigma_\mu^{-1} \Sigma_t - \Sigma_t \Sigma_\nu^{-1} (\Sigma_t^\top + \epsilon I) . \label{def:Riccati}
\end{align}
This matrix-Riccati equation admits a local-in-time solution by Peano's theorem, which must then provide a local-in-time solution to our SDE. In conjunction with Theorem \ref{th:main}(1), we deduce that the unique solution of the SDE \eqref{def:SDE-Gaussian} is given by the Gaussian whose covariance matrix is the unique global-in-time solution of this Riccati equation.

Note that there is an abundance of stationary solutions for the $\epsilon=0$ equation. Indeed, if $S$ is any invertible matrix satisfying $S\Sigma_\mu S^\top = \Sigma_\nu$, then $\Sigma=\Sigma_\mu S^\top$ defines a stationary solution of \eqref{def:Riccati} (corresponding to the coupling $Y=SX$).
For example, if $\Sigma_\mu=\Sigma_\nu=I$, then $S$ can be any orthogonal matrix.
This reflects the previously observed fact that any bijective Monge coupling yields a stationary solution of the noiseless version of our SDE \eqref{def:mainSDE-noiseless}.

Even more transparent is the case of dimension $d=1$, when $\mu$ and $\nu$ are centered Gaussians with variances $\sigma_\mu,\sigma_\nu > 0$. The Riccati equation \eqref{def:Riccati} for $\sigma_t:=\E[X_tY_t]$ reduces to the ODE
\begin{align*}
\frac{d}{dt} \sigma_t &= \sigma_\mu + \sigma_\nu - (\sigma_\mu^{-1}+\sigma_\nu^{-1})\sigma_t(\sigma_t+\epsilon) = (\sigma_\mu^{-1}+\sigma_\nu^{-1})[\sigma_\mu \sigma_\nu - \sigma_t(\sigma_t+\epsilon)].
\end{align*}
As a function of $\sigma_t$, the quadratic right-hand side has roots at
\begin{align*}
\sigma^{\pm} := -(\epsilon/2) \pm \sqrt{(\epsilon/2)^2+\sigma_\mu \sigma_\nu}
\end{align*}
and is positive between these two roots.
Solutions of the above ODE converge exponentially fast to $\sigma^+$ when initialized from $\sigma_0 > \sigma^-$, increasing monotonically if $\sigma^- < \sigma_0 < \sigma^+$ and decreasing monotonically if $\sigma_0 > \sigma^+$. For $\sigma_0 \le \sigma^-$ the ODE behaves differently, but this case is irrelevant: Clearly $\sigma^- < -\sqrt{\sigma_\mu \sigma_\nu}$, whereas Cauchy-Schwarz implies $|\sigma_t| \le \sqrt{\Var(X_t)\Var(Y_t)}= \sqrt{\sigma_\mu \sigma_\nu}$.
When $\epsilon=0$ the ODE acquires an additional (unstable) equilibrium point at $\sigma_0=-\sqrt{\sigma_\mu \sigma_\nu}$.

\subsubsection{The case of non-stationary marginals}

This paper limits its discussion to the setting in which the initial law $P_0$ of $(X_0,Y_0)$ belongs to $\Pi(\mu,\nu)$. We note here some observations and speculations beyond this setting. Suppose $P_0$ has marginals $(\mu_0,\nu_0)$, potentially distinct from $\mu(dx)=e^{-U(x)}dx$ and $\nu(dy)=e^{-V(y)}dy$. Our well-posedness proof in Section \ref{se:wellposedness} can be adapted under some additional smoothness restrictions on $(U,V)$. The solution $(X,Y)$ of the SDE \eqref{def:mainSDE} must satisfy $X_t \stackrel{d}{=} \overline{X}_t$ and $Y_t \stackrel{d}{=} \overline{Y}_t$ for each $t \ge 0$, where $\overline{X}$ and $\overline{Y}$ solve the SDEs
\begin{align*}
d\overline{X}_t &= -\epsilon \nabla U(\overline{X}_t)dt + \sqrt{2\epsilon}\, dW_t, \quad \overline{X}_0 \sim \mu_0, \\
d\overline{Y}_t &= -\epsilon \nabla V(\overline{Y}_t)dt + \sqrt{2\epsilon}\, dB_t, \quad \overline{Y}_0 \sim \nu_0.
\end{align*}
That is, $P_t \in \Pi(\mu_t,\nu_t)$, where $\mu_t$ and $\nu_t$ are the laws of $\overline{X}_t$ and $\overline{Y}_t$, respectively. These two (marginal) processes are ergodic, ensuring that $\mu_t \to \mu$ and $\nu_t \to \nu$. Hence, any limit point of $P_t$ as $t\to\infty$ must belong to $\Pi(\mu,\nu)$.
For this reason, we expect that $P_t$ in fact converges to the same limit $\pi$, regardless of the initialization.

\begin{problem}
Does $P_t \to \pi$, even when $P_0$ does not belong to $\Pi(\mu,\nu)$? A key difficulty is the lack of an obvious choice of  ``energy" functional to play the role of $\J$.
\end{problem}

\subsubsection{The multi-marginal case} \label{se:multimarginal}

It is straightforward to extend all of the results of this paper to the multi-marginal setting, which we sketch in this section without specifying precise assumptions. Consider probability measures $\mu_i(dx)=e^{-U_i(x)}dx$, for $i=1,\ldots,m$, and a cost function $c : (\R^d)^m \to \R$. The multi-marginal optimal transport problem (see \cite{pass2015multi} for a survey) takes the form
\begin{equation*}
\inf_{\pi \in \Pi(\mu_1,\ldots,\mu_n)}\int_{(\R^d)^m} c(x_1,\ldots,x_n) \,\pi(dx_1,\ldots,dx_n),
\end{equation*}
where $\Pi(\mu_1,\ldots,\mu_n)$ is the set of probability measures on $(\R^d)^m$ with marginals $(\mu_1,\ldots,\mu_n)$. This  includes the Wasserstein barycenter problem as notable special case, as explained in \cite{AguehCarlier}. The natural entropic regularization \cite{carlier2020differential} takes the form
\begin{equation}
\inf_{\pi \in \Pi(\mu_1,\ldots,\mu_n)}\bigg(\int_{(\R^d)^m} c(x_1,\ldots,x_n) \,\pi(dx_1,\ldots,dx_n) + \epsilon H(\pi \,|\, \mu_1\otimes \cdots \otimes \mu_n)\bigg). \label{def:EOTmultimarginal}
\end{equation}
The natural analogue of the SDE \eqref{def:mainSDE} is the following:
\begin{equation*}
dX^i_t = \big( \E[\nabla_{x_i} c(\bm{X}_t)\,|\,X^i_t]  -  \nabla_{x_i} c(\bm{X}_t) - \epsilon \nabla U_i(X^i_t)\big)dt + \sqrt{2\epsilon }\,dW^i_t, \quad i=1,\ldots,m,
\end{equation*}
where $\bm{X}_t=(X^1_,\ldots,X^m_t)$, and $(W^1,\ldots,W^n)$ are independent Brownian motions.
The obvious analogue of Theorem \ref{th:main} holds in this setting: If the initial law belongs to $\Pi(\mu_1,\ldots,\mu_m)$, then so does the law at any time $t > 0$, and as $t \to \infty$ it converges to the unique optimizer of \eqref{def:EOTmultimarginal}.

\subsection{Additional prior literature} \label{se:literature}

The SDE \eqref{def:mainSDE} can be viewed as a McKean-Vlasov SDE, in the sense that the drift can be expressed as a functional of the joint law of $(X_t,Y_t)$. There are some similar SDEs involving conditional expectations which have appeared in prior work, in the study of Lagrangian stochastic models of turbulent flows \cite{bossy2011conditional,bossy2018wellposedness} and stochastic local volatility models in mathematical finance \cite{JourdainZhou,LacShkZha,djete2022non}. 
The papers \cite{LelRouSto,JouLelRou} also study SDEs involving a conditional expectation, designed to speed up the simulation of a low-dimensional pushforward of a high-dimensional distribution.

The recent concurrent work \cite{deb2023wasserstein}, which shares an author in common with this paper, describes a Mckean-Vlasov SDE that arises out of the Sinkhorn algorithm used to solve the entropic optimal transport problem. Their SDE is constructed from the continuous-time limit of the marginals from the Sinkhorn algorithm as the regularization parameter $\epsilon \rightarrow 0+$. Since the Sinkhorn algorithm does not preserve the marginal distribution at each iteration, unlike our setting, their dynamics are not constrained to the set of couplings. However, the associated Fokker-Planck equation can be interpreted as a ``mirror'' gradient flow on the Wasserstein space.

There is a curious link with the recent work \cite{LackerZhang}, which, for $c$ of the form $c(x,y)=K(x-y)$, studied the stationary solutions of the SDE 
\begin{align}
\begin{split}
dX_t &= -\Big(\nabla_xc(X_t,Y_t) + (m-1)\E[\nabla_xc(X_t,Y_t)\,|\,X_t] + \nabla U(X_t) \Big) dt + \sqrt{2}dW_t, \\
dY_t &= -\Big(\nabla_yc(X_t,Y_t) + (m-1)\E[\nabla_yc(X_t,Y_t)\,|\,Y_t] + \nabla U(Y_t) \Big) dt + \sqrt{2}dB_t,
\end{split} \label{intro:LackerZhangSDE}
\end{align}
For an integer $m \ge 2$, the stationary solution of \eqref{intro:LackerZhangSDE} characterizes the edge-marginal of a certain infinite (automorphism-invariant) system of particles interacting over the $m$-regular tree. This still makes sense for $m=1$, if we interpret the $1$-regular tree as  the complete graph on two vertices. The  results of \cite{LackerZhang} have no meaning in the case $m=0$, but, perhaps by mere coincidence, the SDE \eqref{intro:LackerZhangSDE} becomes a special case of ours \eqref{def:mainSDE}.
Non-stationary solutions of \eqref{intro:LackerZhangSDE} were not studied in \cite{LackerZhang}, nor the case where $U$ and $V$ are distinct.

A distinct but philosophically connected topic  in discrete probability is the construction of rapidly mixing Markov chains on contingency tables. Contingency tables are matrices whose entries are nonnegative integers and which have fixed row and column sums, which can obviously be viewed as discrete couplings.  The problem of counting the number of contingency tables or that of sampling a randomly chosen contingency table has a long history in statistics, computer science and combinatorics. See the references in \cite{DiaGangoli}, especially in their Section 10 where they introduce a canonical random walk to sample from the uniform distribution on contingency tables. This was generalized to other exponential families in \cite{DiaSturm} where a beautiful connection to techniques from computational algebra such as Gr\"obner bases was exploited. 
As opposed to these stochastic processes on discrete couplings that converge to an invariant distribution, 
our projected Langevin dynamics induce deterministic flows on the set of couplings that converge to the optimizer of a functional.

\subsection{Organization of the paper}

The remaining four sections of the paper may be read independently, for the most part; the exceptions are some regularity lemmas in Section \ref{se:entropycomesdown} which are applied in Section \ref{se:longtime}, and an identity in Section \ref{se:keyidentity} which is reused in Section \ref{se:newLSIproof-suff}.
Section \ref{se:gradientflow} contains a mostly formal discussion of the Wasserstein geometry of $\Pi(\mu,\nu)$ and the interpretation of our SDE as a gradient flow, in the language of Otto calculus. 
Then, Section \ref{se:wellposedness} proves well-posedness of (a somewhat more general version of) the SDE \eqref{def:mainSDE}. Section \ref{se:longtime} studies the long-time convergence, qualitatively, and the final Section \ref{se:expconvergence} justifies the results on quantitative convergence and projected log-Sobolev inequalities announced in Section \ref{se:exprate}.

\section{Gradient flows and the geometry of $\Pi(\mu,\nu)$} \label{se:gradientflow}

In this section we explain how to formally relate the SDE \eqref{def:mainSDE} to a Wasserstein gradient flow in the submanifold $\Pi(\mu,\nu) \subset \P_2(\R^{2d})$. We define as usual the (quadratic) Wasserstein distance between any two probability measures $m$ and $m'$ on a common Euclidean space as
\begin{equation*}
\W_2^2(m,m') := \inf_{\gamma \in \Pi(m,m')}\int |x-y|^2\gamma(dx,dy).
\end{equation*}
The groundbreaking work of Jordan-Kinderlehrer-Otto \cite{jordan1998variational} showed that the time-marginals of the  Langevin dynamics \eqref{intro:Langevin} are a gradient flow in the Wasserstein space $(\P_2(\R^{2d}),\W_2)$ of probability measures. 
They formulate this rigorously in terms of a discrete-time (implicit Euler) approximation, now known as the \emph{JKO scheme}, discussed in Section \ref{se:JKO} below. 
Their perspective was further developed in the so-called \emph{Otto calculus} \cite{otto2001geometry}, which we take a moment to review. We tacitly assume throughout this section that probability measures have smooth positive densities.

The Otto calculus  is a powerful heuristic view of $(\P_2(\R^m),\W_2)$ as a (formal) Riemannian manifold. In order to illustrate this geometric formalism, we start by recalling that if $\nabla \theta$ is the Brenier map pushing forward $\rho$ onto $\rho'$, we can define the corresponding displacement interpolation by $\rho_t=(\mathrm{id}+t(\nabla\theta-\mathrm{id}))_{\#}\rho$. This curve is a constant speed geodesic between $\rho$ and $\rho'$, in the metric sense  \cite[Section 7.2]{AGSbook}. This suggests that we identify Brenier maps with tangent vectors and leads naturally to the following definition of tangent space
\begin{align}
\mathrm{Tan}_\rho\P_2(\R^m) = \overline{ \{\nabla\testf : \testf \in C^\infty_c(\R^m)\} }^{L^2(\rho;\R^m)}, \qquad \rho \in \P_2(\R^m). \label{intro:tangentspace}
\end{align}
The tangent space at each $\rho$ is equipped with the $L^2(\rho;\R^m)$ inner product, which defines a (formal) Riemannian metric sometimes known as the Otto metric.
Remarkably, Otto calculus facilitates several explicit calculations. In particular, it turns out that an absolutely continuous curve $(\rho_t)$  must be the weak solution of a continuity equation $\partial_t\rho_t+\nabla\cdot(v_t\rho_t)=0$ such that $v_t\in\mathrm{Tan}_{\rho_t}\P_2(\R^m)$  for a.e.\ $t$; see \cite[Theorem 8.3.1]{AGSbook} for a fully rigorous statement. Another remarkable fact is that, for a reasonable reference measure $\rho_*$, the gradient of the relative entropy functional $H(\cdot\,|\,\rho_*)$ at a measure $\rho$ can be computed explicitly as the vector field $\nabla_{\W}H(\cdot\,|\,\rho_*)(\rho) = \nabla\log(\rho/\rho_*)$; see \cite{jordan1998variational}.

With Otto calculus in hand, gradient flows on $(\P_2(\R^m),\W_2)$ can be defined in analogy with the Euclidean setting. We recall that the prototype of gradient flow is the Euclidean ODE
\begin{equation}\label{GF_toy_intro_sec2}
    \dot{Z}_t=-\nabla H(Z_t),
\end{equation}
where $H : \R^m \to \R$ is a suitable energy functional. In light of the above discussion, one defines a gradient flow on $\mathcal{P}_2(\R^m)$ for a given energy functional such as $\epsilon H(\cdot\,|\,\rho_*)$ via the formula
\begin{equation}\label{eq:Wasserstein_GF}
\partial_t \rho_t - \epsilon \nabla \cdot\big( \rho_t\,\nabla_{\W}H(\cdot\,|\,\rho_*)(\rho_t)  \big)=0.
\end{equation}
In other words, \eqref{eq:Wasserstein_GF} is the continuity equation $\partial_t\rho_t + \nabla \cdot (\rho_t v_t)=0$, with the vector field $v_t$ taken to be $v_t=-\epsilon \nabla_{\W}H(\cdot\,|\,\rho_*)(\rho_t)$.
Recalling the aforementioned formula for $\nabla_\W H(\cdot\,|\,\rho_*)$, we can write \eqref{eq:Wasserstein_GF} as
\begin{equation}\label{eq:Wasserstein_GF2}
\partial_t \rho_t = \epsilon \nabla \cdot\big( \rho_t\,\nabla \log(\rho_t/\rho_*)  \big) = -\epsilon \nabla \cdot (\rho_t\log\rho_*) + \epsilon \Delta \rho_t.
\end{equation}
When $\rho_*(dx,dy) \propto e^{-c(x,y)/\epsilon}\mu(dx)\nu(dy)$, this is precisely the Fokker-Planck equation associated with the Langevin dynamics \eqref{intro:Langevin}.

In the upcoming sections we shall make the case that the SDE \eqref{def:mainSDE} relates to the gradient flow of $\epsilon H(\cdot\,|\,\pi)$ in $\Pi(\mu,\nu)$, which we view as a Riemannian submanifold of $(\P_2(\R^{2d}),\W_2)$. 
In Section \ref{se:tangentspace}, we formally identify the tangent space at $P\in\Pi(\mu,\nu)$ as 
\begin{align}
&\mathrm{Tan}_P\Pi(\mu,\nu) \label{eq:tan_sp_coupl} \\
& \ \ := \bigg\{ v \in \mathrm{Tan}_P\P_2(\R^{2d}) : \int v(x,y) \cdot \begin{pmatrix}
\nabla \testf_1(x) \\ \nabla \testf_2(y) 
\end{pmatrix}\,P(dx,dy) =0, \ \forall \testf_1,\testf_2 \in C^\infty_c(\R^{d})\bigg\}.  \nonumber
\end{align}
Gradient flows on submanifolds of $\P_2(\R^{2d})$ can be understood in analogy with the finite dimensional case: If $M\subset \R^m$ is a submanifold of a Euclidean space, it is well known that the gradient flow of $H$ in $M$ is defined by
\begin{equation}\label{eq:Euclidean_GF}
    \dot{Z}_t=-P(Z_t)\nabla H(Z_t), 
\end{equation}
where $P(Z_t)$ is the orthogonal projection  $\R^m \to \mathrm{Tan}_{Z_t}M$. We review this in more detail in Section \ref{se:GF-couplings}, and we will see by analogy that  the gradient flow of $\epsilon H(\cdot\,|\,\pi)$ in $\Pi(\mu,\nu)$ should be identified with the PDE
\begin{align} 
\partial_t P_t = \epsilon \nabla \cdot \big(P_t\,\textsf{T}_{P_t}\nabla_{\W}H(\cdot\,|\,\pi)(P_t)\big )  = \epsilon \nabla \cdot \big(P_t\,\textsf{T}_{P_t}\nabla \log (P_t/\pi)\big ) , \label{def:GFPDE}
\end{align}
where $\textsf{T}_{P_t}$ denotes the orthogonal projection $L^2(P_t;\R^{2d}) \to \mathrm{Tan}_{P_t}\Pi(\mu,\nu)$. 
We then explain in Section \ref{se:d=1} that the time-marginal flow of the SDE \eqref{def:mainSDE} provides a solution of \eqref{def:GFPDE} when $d=1$. Section \ref{se:d>1} explains why the two dynamics are different though tightly related when $d\ge 2$.
Lastly, we explore some further topics of a geometric nature in the remaining Sections \ref{se:JKO}, \ref{se:geodesicconvexity}, and \ref{se:generalization}, which respectively discuss the JKO scheme, the geodesic non-convexity of $\Pi(\mu,\nu)$, and some generalities on gradient flows on certain submanifolds of Wasserstein space.

\subsection{The tangent space of $\Pi(\mu,\nu)$} \label{se:tangentspace}

In this section we justify the claimed formula \eqref{eq:tan_sp_coupl} for the tangent spaces of $\Pi(\mu,\nu)$.
Recalling the definition of $\mathrm{Tan}_P\P_2(\R^{2d})$ from \eqref{intro:tangentspace}, we see that \eqref{eq:tan_sp_coupl} is precisely the orthogonal complement in $\mathrm{Tan}_P\P_2(\R^{2d})$ of the set of the vector fields of the form $(x,y) \mapsto (\nabla\testf_1(x),\nabla\testf_2(y))$, where $\testf_1,\testf_2 \in C^\infty_c(\R^d)$. Perhaps more suggestively, we may write \eqref{eq:tan_sp_coupl}  as 
\begin{equation}
\mathrm{Tan}_P\Pi(\mu,\nu) = \big(\mathrm{Tan}_\mu\P_2(\R^{d}) \otimes  \mathrm{Tan}_\nu\P_2(\R^{d})\big)^{\perp}. \label{def:TanPi}
\end{equation}

To justify this definition of $\mathrm{Tan}_P\Pi(\mu,\nu)$, let us fix an absolutely continuous curve $t \mapsto P_t \in \P_2(\R^{2d})$ satisfying $P_0 \in \Pi(\mu,\nu)$. We will argue that $P_t \in \Pi(\mu,\nu)$ for all $t$ if and only if the tangent vector at time $t$ belongs to the set $\mathrm{Tan}_P\Pi(\mu,\nu)$ for almost every $t$. More precisely, recall from \cite[Theorem 8.3.1]{AGSbook}, that absolute continuity of $(P_t)$ is equivalent to the existence of an vector field $(t,x,y) \mapsto v_t(x,y)$ in $L^2(dtP_t(dx);\R^{2d})$ such that $v_t \in \mathrm{Tan}_{P_t}\P_2(\R^{2d})$ for a.e.\ $t$ and the continuity equation $\partial_t P + \nabla \cdot (Pv)=0$ holds in the sense of distributions. Write $P^1_t$ and $P^2_t$ for the first and second marginals of $P_t \in \P_2(\R^d \times \R^d)$. Apply the continuity equation to a test function of the form $(x,y) \mapsto \testf_1(x) + \testf_2(y)$ to get
\begin{align*}
\frac{d}{dt}\int_{\R^d} \testf_1\, dP^1_t + \int_{\R^d} \testf_2\, dP^2_t = \int_{\R^{2d}} v_t(x,y) \cdot \begin{pmatrix}
\nabla\testf_1(x) \\ \nabla\testf_2(y)
\end{pmatrix} \,P_t(dx,dy), \ \ a.e. \ t.
\end{align*}
Because $P_0 \in \Pi(\mu,\nu)$, we see that $P_t \in \Pi(\mu,\nu)$ for all $t$ if and only if the left-hand side is zero for all $(\testf_1,\testf_2)$ and a.e.\ $t$, which is equivalent to $v_t$ belonging to $\mathrm{Tan}_P\Pi(\mu,\nu)$ for a.e.\ $t$.

\begin{remark}
Let us stress that, in the geometric discussion of this section, we view $\Pi(\mu,\nu)$ as a \emph{Riemannian submanifold} but not as a \emph{metric subspace} of $\P_2(\R^{2d})$. In particular, $\Pi(\mu,\nu)$ is not equipped with the restriction to $\Pi(\mu,\nu)$ of the metric $\W_2$, but rather the tangent spaces $\mathrm{Tan}_P\Pi(\mu,\nu)$ are equipped with the restriction of the Riemannian (Otto) metric of $\mathrm{Tan}_P\P_2(\R^{2d})$. This (formal) Riemannian metric on $\Pi(\mu,\nu)$ induces a distance which is different from $\W_2$, given instead by
\begin{equation*}
(Q_0,Q_1) \mapsto \inf \bigg(\int_0^1\|v_t\|_{P_t}^2\,dt\bigg)^{1/2},
\end{equation*}
where the infimum is over all absolutely continuous curves $(P_t)_{t \in [0,1]}$ in $(\P_2(\R^{2d}),\W_2)$ satisfying $P_0=Q_0$, $P_1=Q_1$, $P_t \in \Pi(\mu,\nu)$ for all $t \in (0,1)$, and finally the continuity equation $\partial_tP + \nabla \cdot (Pv)=0$. In other words, one should restrict the Benamou-Brenier formula \cite{BenamouBrenier} for $\W_2$ to include only those paths that lie within $\Pi(\mu,\nu)$.
\end{remark}

\subsection{Gradient flows on $\Pi(\mu,\nu)$}\label{se:GF-couplings}

Having now identified the tangent spaces $\mathrm{Tan}_P\Pi(\mu,\nu)$, we next explain in more detail the derivation of the PDE \eqref{def:GFPDE}.
It is helpful to build intuition by reviewing the ideas in a finite dimensional setting.
Let $M=\{x \in \R^{2d} : g(x)=r_0\}$ be a submanifold determined by a smooth function $g : \R^{2d} \to \R^k$ and a fixed vector $r_0 \in \R^k$.
 Denoting by $|\cdot|$ and $\langle\cdot\,,\,\cdot\rangle$ the standard norm and inner product on $\mathbb{R}^{2d}$, we can view $M$ as a Riemannian submanifold of $(\mathbb{R}^{2d},\langle\cdot\,,\,\cdot\rangle)$ by considering the restriction of $\langle\cdot\,,\,\cdot\rangle$ onto $\mathrm{Tan}_xM$, where the tangent space $\mathrm{Tan}_{x}M$ is
defined as the set of velocities $\dot{Z}_0$ of smooth curves $(Z_t)_{t \in (-\delta,\delta)}$ contained in $M$.
In this simple setting, we have the following fundamental facts:
 \begin{itemize}
 \item $\mathrm{Tan}_xM = \mathrm{ker}(Dg(x))$ for all $x\in M$.
 \item  Let $\nabla H$ denote the (Euclidean) gradient of a smooth function $H : \R^{2d} \to \R$. Then the gradient of $H$ on the submanifold at $x\in M$ is $P(x)\nabla H(x)$, where  $P(x)$ is the orthogonal projection onto $\mathrm{Tan}_xM$. To justify this, let $(Z_t)_{t \in (-\delta,\delta)}$ be a smooth curve contained in $M$. Then
\begin{equation*}
\lim_{t\downarrow0} \frac{H(Z_t)-H(Z_0)}{t}=\langle\nabla H(Z_0),\dot{Z}_0\rangle=\langle P(Z_0)\nabla H(Z_0),\dot{Z}_0\rangle,
\end{equation*} 
with the last step using $\dot{Z}_0 \in \mathrm{Tan}_xM$.
 \item The gradient flow of $H$ in $M$ is given by 
 \begin{equation}\label{eq:toy_submanifold_gf}
 \dot{Z}_t = -P(Z_t)\nabla H(Z_t), 
 \end{equation}
where $\dot{Z}_t$ is the velocity computed under the standard Euclidean metric. This is a straightforward consequence of the fact that we have identified the gradient of $H$ in the submanifold $M$ and that the velocity of a curve in $M$ coincides with the velocity of the same curve viewed as a curve in the Euclidean space.
 \end{itemize}
Let us now return to our infinite-dimensional setting. Write $\Pi(\mu,\nu)=\{P \in \P_2(\R^{2d}) : \widehat{\Pi}(P)=(\mu,\nu)\}$, where $\widehat{\Pi} : \P_2(\R^{2d}) \to\P_2(\R^d)^2$ is the map which sends a probability measure to its pair of marginals. The natural analogues of the above statements in our context are the following:
\begin{itemize}
\item The tangent space $\mathrm{Tan}_{P}\Pi(\mu,\nu)$ at $P \in \Pi(\mu,\nu)$ coincides with $\mathrm{ker}(D\widehat{\Pi}_P)$.
 We will explain below that this is consistent with our definition of $\mathrm{Tan}_{P}\Pi(\mu,\nu)$ in \eqref{eq:tan_sp_coupl}.
\item 
As noted before, the Wasserstein gradient of relative entropy $P \mapsto H(P\,|\,\pi)$ in $\P_2(\R^{2d})$ is given by $ \nabla \log (P/\pi)$. Hence, the gradient in the submanifold $\Pi(\mu,\nu)$ is given by $\textsf{T}_P\nabla \log (P/\pi)$, where $\textsf{T}_P$ denotes the $L^2(P;\R^{2d})$-projection onto $\mathrm{Tan}_P\Pi(\mu,\nu)$ for each $P \in \Pi(\mu,\nu)$.
To justify this, let $(P_t)_{t \in (-\delta,\delta)}$ be a smooth curve in $\Pi(\mu,\nu)$, which must therefore satisfy the continuity equation $\partial_tP+\nabla \cdot (Pv)=0$ for some velocity field $(v_t)$. In the framework of Otto calculus, we have
 \begin{equation*}
 \frac{d}{dt}\bigg|_{t=0} \! H(P_t\,|\,\pi) = \int_{\mathbb{R}^{2d}} v_0\cdot \nabla\log (P_0/\pi) \, dP_0 = \int_{\mathbb{R}^{2d}} v_0\cdot \textsf{T}_{P_0}\nabla\log (P_0/\pi)\,dP_0,
\end{equation*}
with the last step using $v_0\in\mathrm{Tan}_{P_0}\Pi(\mu,\nu)$.
 \item The gradient flow of $\epsilon H(\cdot\,|\,\pi)$ is given by 
\begin{align}
\partial_t P_t =\epsilon \nabla \cdot \big(P_t\,\textsf{T}_{P_t}\nabla \log (P_t/\pi)\big).  \label{def:GF-PDE}
\end{align}
As before, this follows from the fact that $\textsf{T}_{P}\nabla\log(P/\pi)$ is the gradient of the relative entropy in the submanifold $\Pi(\mu,\nu)$ and that the velocity of a curve in $\Pi(\mu,\nu)$ coincides with the velocity of the same curve viewed as a curve in the $\mathcal{P}_2(\mathbb{R}^{2d})$ equipped with the Otto metric.
 \end{itemize}
Let us now further elaborate on the identity $\mathrm{Tan}_{P}\Pi(\mu,\nu)=\mathrm{ker}(D\widehat{\Pi}_{P})$.
First, in the setting of two finite-dimensional manifolds $M_1$ and $M_2$, recall that the differential $Dg_{y_0}$ of a smooth map $g : M_1 \to M_2$ at a point $y_0 \in M$ is the linear map $\mathrm{Tan}_{y_0}M_1 \to \mathrm{Tan}_{g(y_0)}M_2$ uniquely defined by the formula
\begin{equation*}
Dg_{y_0}(y_0') = \frac{d}{dt}\Big|_{t=0} g(y_t),
\end{equation*}
which must hold for all smooth curves $(y_t)$ in $M_1$ passing through $y_0$ at $t=0$. The differential $D\widehat{\Pi}_{P_0}$ at $P_0 \in \Pi(\mu,\nu)$ is similarly a linear map $\mathrm{Tan}_{P_0}\P_2(\R^{2d}) \to \mathrm{Tan}_{(\mu,\nu)}(\P_2(\R^d))^2 \cong \mathrm{Tan}_\mu\P_2(\R^d) \otimes \mathrm{Tan}_\nu\P_2(\R^d)$, where we implicitly view $(\P_2(\R^d))^2$ as a formal product of two Riemannian manifolds. 
Let $(P_t)$ be an absolutely continuous curve in $\P_2(\R^{2d})$ passing through $P_0$ at $t=0$, and let $v_t \in \mathrm{Tan}_{P_t}(\R^{2d})$ be the associated vector field such that the continuity equation $\partial_t P + \nabla \cdot (Pv)=0$ holds. Identifying $v_0$ with $P_0'$, we should define $D\widehat{\Pi}_{P_0}(v_0) := (d/dt)\widehat{\Pi}(P_t)|_{t=0}$ to parallel the finite-dimensional formula. The tangent vector $(d/dt)\widehat{\Pi}(P_t)|_{t=0}$ should be computed as the time-zero value of the vector field governing the continuity equation satisfied by the curve $(\widehat{\Pi}(P_t))$ in $(\P_2(\R^d))^2$. To this end, for smooth functions $h_1,h_2:\R^d\to\R$ of compact support we compute
\begin{align*}
\frac{d}{dt}\Big|_{t=0} \langle \widehat{\Pi}(P_t),(\testf_1,\testf_2)\rangle &:= \frac{d}{dt}\Big|_{t=0} \int_{\R^{2d}}(\testf_1(x) + \testf_2(y))\,P_t(dx,dy) \\
	&= \int_{\R^{2d}} v_0(x,y) \cdot \begin{pmatrix}
\nabla\testf_1(x) \\ \nabla\testf_2(y)
\end{pmatrix}\,P_0(dx,dy) \\
	&= \int_{\R^{2d}} (\textsf{I}-\textsf{T}_{P_0})v_0(x,y) \cdot \begin{pmatrix}
\nabla\testf_1(x) \\ \nabla\testf_2(y)
\end{pmatrix}\,P_0(dx,dy).
\end{align*}
Here $\textsf{I}$ denotes the identity operator, and thus $\textsf{I}-\textsf{T}_{P_0}$ is the projection in $L^2(P_0;\R^{2d})$ onto the closure of the space of vector fields of the form $(x,y) \mapsto (\nabla\testf_1(x),\nabla\testf_2(y))$, which can be identified with $\mathrm{Tan}_\mu\P_2(\R^d) \otimes \mathrm{Tan}_\nu\P_2(\R^d) = (\mathrm{Tan}_{P_0}\Pi(\mu,\nu))^\perp$.
This justifies the formula $D\widehat{\Pi}_{P_0}=\textsf{I}-\textsf{T}_{P_0}$, which yields the identity $\mathrm{ker}(D\widehat{\Pi}_{P_0})=\mathrm{Tan}_{P_0}\Pi(\mu,\nu)$.

We lastly note that, because $D\widehat{\Pi}_{P}$ is an orthogonal projection, we should (formally) interpret this ``marginal map" $\widehat{\Pi}$ as a Riemannian submersion, and then $\mathrm{Tan}_{P}\Pi(P^1,P^2)$ is the associated vertical space at $P$, where $P^1$ and $P^2$ are the marginals of $P$. 

\subsection{Projected Langevin dynamics as gradient flow in $d=1$} \label{se:d=1}

We have justified (formally, via Otto calculus) that the PDE \eqref{def:GF-PDE} is the gradient flow for $\epsilon H(\cdot\,|\,\pi)$ in the submanifold $\Pi(\mu,\nu)$. 
In dimension $d=1$, we will now connect this PDE to the SDE \eqref{def:mainSDE}. This will rely on the crucial simplification
\begin{equation}
\mathrm{Tan}_\mu\P_2(\R) := \overline{ \{\nabla\testf : \testf \in C^\infty_c(\R)\} }^{L^2(\mu)} = L^2(\mu), \label{pf:Tan-simplfication}
\end{equation}
and similarly $\mathrm{Tan}_\nu\P_2(\R) =L^2(\nu)$, 
which lets us rewrite \eqref{def:TanPi} as
\begin{equation*}
\mathrm{Tan}_P\Pi(\mu,\nu) = (L^2(\mu) \otimes L^2(\nu))^\perp.
\end{equation*}
Note then that a vector field $v=(v^1,v^2)$ in $L^2(P;\R^2)$ is in $\mathrm{Tan}_P\Pi(\mu,\nu)$ if and only if
\begin{equation*}
\int_{\R^2} \big(v^1(x,y)\testf_1(x) + v^2(x,y)\testf_2(y) \big)\,P(dx,dy) = 0,
\end{equation*}
for all bounded measurable $\testf_1,\testf_2$. In probabilistic notation, this means precisely that
\begin{equation*}
\E_P[v^1(X,Y)\,|\,X]=0, \qquad \E_P[v^2(X,Y)\,|\,Y]=0.
\end{equation*}
This leads to a simple expression for the projection,
\begin{equation}
\textsf{T}_Pv(X,Y) = \begin{pmatrix}
v^1(X,Y) - \E_P[v^1(X,Y)\,|\,X] \\ v^2(X,Y) - \E_P[v^2(X,Y)\,|\,Y]
\end{pmatrix}. \label{pf:simplification-TP}
\end{equation}
From this we deduce the following key identity: For $P \in \Pi(\mu,\nu)$ sufficiently regular, we have
\begin{align}
\textsf{T}_P &\nabla \log \frac{P}{\pi}(X,Y)  \label{def:GF-TP-projection} \\
&= \nabla\log P(X,Y) + \begin{pmatrix}
\epsilon^{-1} \big(\nabla_x c(X,Y) - \E_P[ \nabla_x c(X,Y)\,|\,X]\big) + \nabla U(X) \\
\epsilon^{-1} \big(\nabla_y c(X,Y) - \E_P[ \nabla_y c(X,Y)\,|\,Y]\big) + \nabla V(Y)
\end{pmatrix}. \nonumber
\end{align}
To see this, use  the specific form of $\pi$ given in \eqref{intro:potentials} to deduce
\begin{equation*}
\nabla_x\log\pi(X,Y) - \E_P[ \nabla_x\log\pi(X,Y) \,|\,X] = \epsilon^{-1}\E_P[ \nabla_x c(X,Y)\,|\,X] - \epsilon^{-1} \nabla_x c(X,Y),
\end{equation*}
and combine this with the marginalization identity
\begin{equation*}
\E_P[\nabla_x\log P(X,Y)\,|\,X] = \nabla \log\mu(X)=-\nabla U(X),
\end{equation*}
and with the analogous identities for $\nabla_y$ terms. 
Finally, the identity \eqref{def:GF-TP-projection} shows that the PDE \eqref{def:GF-PDE} can be written as
\begin{equation*}
\partial_t P_t(x,y) = \epsilon \Delta P_t(x,y) - \nabla \cdot \left( P_t(x,y) \begin{pmatrix}
 \E_{P_t}[ \nabla_x c(X,Y)\,|\,X=x] - \nabla_x c(x,y) - \epsilon \nabla U(x) \\
 \E_{P_t}[ \nabla_y c(X,Y)\,|\,Y=y] - \nabla_y c(x,y) - \epsilon \nabla V(y)
\end{pmatrix} \right).
\end{equation*}
By a standard application of It\^o's formula, this is precisely the (nonlinear) Fokker-Planck equation satisfied by the measure flow given by the solution of the SDE \eqref{def:mainSDE}.

\subsection{An alternative SDE in $d \ge 2$} \label{se:d>1}

The simplification \eqref{pf:Tan-simplfication} occurs only in $d=1$, not in $d \ge 2$, essentially due to the existence of non-conservative vector fields in $d \ge 2$. That is, $\mathrm{Tan}_\mu\P_2(\R^d)$ is a strict subspace of $L^2(\mu;\R^d)$, at least for absolutely continuous $\mu$.
For this reason, the simple expression \eqref{pf:simplification-TP} for the projection $\textsf{T}_P$ is no longer valid in $d \ge 2$, and it must be corrected as follows: The conditional expectation $\E_P[v^1(X,Y)\,|\,X]$ must be replaced by the projection in $L^2(\mu;\R^d)$ of the function $x \mapsto \E_P[v^1(X,Y)\,|\,X=x]$ onto the subspace $\mathrm{Tan}_\mu\P_2(\R^d)$. 
Applying this additional projection throughout leads to a different SDE system, which we discuss here but do not give any rigorous results.

To write the different SDE, we favor the following probabilistic notation.
For a probability space $(\Omega,\F,\PP)$ supporting  $\R^d$-valued random vectors $X$ and $Z$, with $Z$ being square-integrable, we define $\Enabla[Z\,|\,X]$ to be the $L^2(\PP)$-projection of $Z$ onto the $L^2(\PP)$-closure of $\{\nabla\testf(X) : \testf \in C^\infty_c(\R^d)\}$. In other words, $\Enabla[Z\,|\,X]$ is the unique element of this closure satisfying
\begin{align}
\E\big[ \Enabla[Z\,|\,X] \cdot \nabla \testf(X) \big] = \E\big[ Z \cdot \nabla \testf(X) \big], \quad \forall \testf \in C^\infty_c(\R^d). \label{def:Enabla}
\end{align}
It follows that $\Enabla[Z\,|\,X]$ is also the projection of $\E[Z\,|\,X]$ onto the same space.
In $d=1$, we have in fact $\Enabla[Z\,|\,X] = \E[Z\,|\,X]$, because the $L^2(\PP)$-closure of $\{\testf'(X) : \testf \in C^\infty_c(\R)\}$ is exactly the set of $X$-measurable elements of $L^2(\PP)$.
Note in general that $\Enabla[Z\,|\,X]$ defines an $X$-measurable random variable, which can thus be represented uniquely (up to a.s.-equality) by a Borel function of $X$; we abuse notation by writing $x \mapsto \Enabla[Z\,|\,X=x]$ for this function, as is common practice for the usual conditional expectation. For another perspective: If $X \sim \mu$, then $x \mapsto \Enabla[Z\,|\,X=x]$ is the element of $L^2(\mu)$ obtained as the $L^2(\mu)$-projection of $f$ onto the subspace $\mathrm{Tan}_\mu \P_2(\R^d)$, when $f$ is defined as the unique element (or, more precisely, equivalence class) of $L^2(\mu)$ satisfying $f(X)=\E[Z\,|\,X]$ a.s.
With the new notation at hand, we can state the following multidimensional generalisation of \eqref{pf:simplification-TP}:
\begin{align*}
\textsf{T}_Pv(X,Y) = \begin{pmatrix}
v^1(X,Y) - \E^{\nabla}_P[v^1(X,Y)\,|\,X] \\ v^2(X,Y) - \Enabla_P[v^2(X,Y)\,|\,Y]
\end{pmatrix}. 
\end{align*}

We may now write the proposed variant of the SDE \eqref{def:mainSDE}, in which the conditional expectations are replaced by this gradient-projection:
\begin{align}
\begin{split}
dX_t &= \Big( \Enabla[\nabla_x c(X_t,Y_t)\,|\,X_t]  -  \nabla_x c(X_t,Y_t) - \epsilon \nabla U(X_t)\Big)dt + \sqrt{2\epsilon }\,dW_t, \\
dY_t &= \Big( \Enabla[\nabla_y c(X_t,Y_t)\,|\,Y_t]  -  \nabla_y c(X_t,Y_t) - \epsilon \nabla V(Y_t)\Big)dt + \sqrt{2\epsilon }\,dB_t.
\end{split} \label{def:mainSDE-proj}
\end{align}
The marginal flow of this SDE is expected to solve \eqref{def:GF-PDE}, thus providing with a probabilistic representation of the gradient flow of the entropy in the submanifold $\Pi(\mu,\nu)$. In $d=1$, this coincides with \eqref{def:mainSDE}, but in general it is distinct in $d \ge 2$.

\begin{problem}
Does an analogue of Theorem \ref{th:main} holds for the SDE \eqref{def:mainSDE-proj}?
\end{problem}

We were unable to prove that a weak solution of the SDE \eqref{def:mainSDE-proj} exists, even when $\nabla c$ is smooth and bounded. A fundamental difficulty is that the gradient projection operator, unlike the ordinary conditional expectation, may not necessarily be bounded from $L^\infty$ to $L^\infty$. 
But if a solution of \eqref{def:mainSDE-proj} does exist, the other properties as in Theorem \ref{th:main} do not appear as difficult to show. In particular, the law of $(X_t,Y_t)$ remains in $\Pi(\mu,\nu)$ for each $t > 0$ if it is initialized as such at $t=0$. This follows from the  calculation sketched in \eqref{intro:marginalcalculation}, using \eqref{def:Enabla} in place of the tower property. 
In addition, analogous results to Theorem \ref{th:main}(3,4) and Section \ref{se:exprate} should hold with similar proofs.

\subsection{The JKO scheme} \label{se:JKO}

The well known JKO scheme \cite{jordan1998variational} for the Langevin dynamics \ref{intro:Langevin} takes the following form. 
Let $\rho_0 \in \P_2(\R^{2d})$, and for each $h > 0$ and $k \in \N$ define
\begin{equation}
\rho^h_{(k)} = \mathrm{argmin}_{\rho \in \P_2(\R^{2d})} \bigg(H(\rho\,|\,\rho_*) + \frac{1}{2h}\W_2^2(\rho,\rho^h_{(k-1)})\bigg), \label{def:JKO}
\end{equation}
with $\rho^h_{(0)}=\rho_0$, where again $\rho_*(dx,dy) \propto e^{-c(x,y)/\epsilon}\mu(dx)\nu(dy)$. Define interpolants $\rho^h_t = \rho^h_{(k)}$ for $t \in [kh,(k+1)h)$. Then it was shown in \cite{jordan1998variational} that $(\rho^h_t)_{t \ge 0}$ converges pointwise to the law of the solution of \ref{intro:Langevin} initialized from $\rho_0$.

It is natural in our context to try restricting the argmin in \eqref{def:JKO} to $\Pi(\mu,\nu)$. 
Precisely: Let $P_0 \in \P_2(\R^{2d})$, and for each $h > 0$ and $k \in \N$ define
\begin{equation}
P^h_{(k)} = \mathrm{argmin}_{P \in \Pi(\mu,\nu)} \bigg(H(P\,|\,\pi) + \frac{1}{2h}\W_2^2(P,P^h_{(k-1)})\bigg), \label{def:JKO-new}
\end{equation}
with $P^h_{(0)}=P_0$. Define interpolants $P^h_t = P^h_{(k)}$ for $t \in [kh,(k+1)h)$. 

\begin{problem}
Does $(P^h_t)_{t \ge 0}$ converge as $h\to 0$? What is its limit?
\end{problem}

We do not attack this natural problem here, but, because of the formal discussions above, we expect when $d=1$ that $(P^h_t)_{t \ge 0}$ converges pointwise to the flow $(P_t)_{t \ge 0}$ from Theorem \ref{th:main}.

\subsection{Geodesic convexity, or lack thereof} \label{se:geodesicconvexity}

The finite-dimensional analogy reveals the difficulty in obtaining exponential rates of convergence, as in our Theorem \ref{thm:suff_cond_new_LSI}.
In the finite-dimensional setting, if we know that $c$ is strongly convex, then the usual gradient flow \eqref{intro:Langevin-noiseless} converges exponentially fast to the unique global minimizer of $c$. This does not imply anything, however, about the gradient flow of $c|_M$ on the submanifold $M$, i.e., the dynamics  \eqref{eq:toy_submanifold_gf} above. Indeed, the latter can have many critical points, stable or unstable. Examples are abundant and simple; for instance, if $c(x,y)=|x|^2+|y|^2$ and $M$ is a centered sphere, then $c|_M$ is constant.

If we know that $c|_M$ is a \emph{geodesically} strongly convex function on the manifold $M$, then we can indeed deduce that $Z_t$ converges exponentially fast to the unique minimizer of $c$ over $M$.
But, unless $M$ is convex as a subset of $\R^{2d}$, there is apparently no simple relationship between the ordinary versus geodesic convexity of $c$, 
essentially because the relationship between the Euclidean Hessian of $c$ and the Riemannian Hessian of $c|_M$ is complex.

Returning to our infinite-dimensional setting, there are many functionals known since the work of McCann \cite{mccann1997convexity} to be geodesically convex on $(\P_2(\R^{2d}),\W_2)$, leading to exponential rates of convergence to equilibrium for Fokker-Planck equations \cite{carrillo2003kinetic}. But these results do not transfer easily to our setting because the space $\Pi(\mu,\nu)$ fails to be geodesically convex in $(\P_2(\R^{2d}),\W_2)$. 
We highlight this observation in the following lemma.
Note, in contrast, that $\Pi(\mu,\nu)$ is always trivially convex in the usual linear structure of the space of measures.

\begin{proposition} \label{pr:geodesicconvexity}
For $\mu,\nu \in \P_2(\R^d)$, the following are equivalent:
\begin{enumerate}[(1)]
\item $\Pi(\mu,\nu)$ is a geodesically convex subset of $(\P_2(\R^{2d}),\W_2)$.
\item $\Pi(\mu,\nu)$ is a singleton.
\item Either $\mu$ or $\nu$ is a point mass (or both).
\end{enumerate} 
\end{proposition}

Because of Proposition \ref{pr:geodesicconvexity}, we are unable to directly apply the general theory of gradient flows in metric spaces and Wasserstein space, developed in great detail and rigor in the now-canonical reference \cite{AGSbook}. Indeed, the bulk of this theory requires semiconvexity of the relevant functional $\J$ along geodesics.
Because $\Pi(\mu,\nu)$ fails to be geodesically convex, the geodesics in this submanifold (in the induced geometry) are different from the usual displacement interpolations in $\P_2(\R^{2d})$. 
We did not find a useful description of the geodesics of $\Pi(\mu,\nu)$, if they even exist.
In general, it is difficult to identify geodesics on submanifolds of Wasserstein space. The only example we know of is the remarkable work of Carlen-Gangbo \cite{CarlenGangbo}, which explicitly describes the geodesics of sets of measures with prescribed first and second moments.

\begin{problem}
Describe the geodesics of $\Pi(\mu,\nu)$.
\end{problem}

\begin{proof}[Proof of Proposition \ref{pr:geodesicconvexity}]
Note that $\Pi(\mu,\nu)$ is never empty, as it always contains $\mu \otimes \nu$.
The implications (3) $\Rightarrow$  (2) $\Rightarrow$ (1) are trivial. 

To show that (2) $\Rightarrow$ (3),  assume that neither $\mu$ nor $\nu$ is point mass. As point masses are the extreme points of $\P(\R^d)$, there exist probability measures $(\mu_1,\mu_2,\nu_1,\nu_2)$, necessarily with finite second moment, such that $\mu=\tfrac12(\mu_1+\mu_2)$ and $\nu=\tfrac12(\nu_1+\nu_2)$ and also $\mu_1\neq\mu_2$, $\nu_1\neq \nu_2$. The measure $\tfrac12 \mu_1 \otimes \nu_1 + \tfrac12\mu_2\otimes\nu_2$ belongs to $\Pi(\mu,\nu)$ and is distinct from $\mu \otimes \nu$, which shows that $\Pi(\mu,\nu)$ is not a singleton.

We lastly show that (1) $\Rightarrow$ (2). 
Let $P_0,P_1 \in \Pi(\mu,\nu)$. We will show that the assumption of geodesic convexity implies that  $P_1=P_0$. 
By \cite[Theorem 7.2.2]{AGSbook}, a constant-speed geodesic in $(\P_2(\R^{2d}),\W_2)$ between $P_0$ and $P_1$ takes the form $P_t = \mathrm{Law}(Z_t)$ where $Z_t:= (1-t)Z_0 + tZ_1$ for $t\in [0,1]$, where $(Z_0,Z_1)$ is a random variable whose law is an optimal coupling of $(P_0,P_1)$ for $\W_2$. Let us decompose the $\R^{2d}$-valued random vectors $Z_t$ into pairs of $\R^d$-valued random vectors, $Z_t=(X_t,Y_t)$. For $t=0$ or $t=1$, because $P_t$ belongs to $\Pi(\mu,\nu)$, we have $X_t \sim \mu$ and $Y_t \sim \nu$.
Now, the assumed geodesic convexity means that $P_t \in \Pi(\mu,\nu)$ and thus $X_t \sim \mu$ and $Y_t \sim \nu$ holds also for all intermediate $t \in (0,1)$. Let $m=\E X_0$. The function $t \mapsto \E|X_t-m|^2$ is constant, which, when combined with the identity
\begin{align*}
\E|X_t-m|^2 &= \E\big|(1-t) X_0 + t X_1) - m\big|^2 \\
	&= (1-t)^2\E|X_0-m|^2 + t^2\E|X_1-m|^2 + 2t(1-t)\E\big[(X_0-m)\cdot(X_1-m)\big],
\end{align*}
implies that
\begin{align*}
\E|X_0-m|^2=\E\big[(X_0-m)\cdot(X_1-m)\big].
\end{align*}
This is the equality case of Cauchy-Schwarz, and we deduce that $X_0-m$ and $X_1-m$ are proportional. Because they have the same law, they must in fact be identical, i.e., $X_1=X_0$. Similarly, $Y_1=Y_0$, and it follows that $P_1=P_0$.
\end{proof}

\begin{remark}
An inspection of the proof of Proposition \ref{pr:geodesicconvexity} reveals a stronger fact: The set $\Pi(\mu,\nu)$ contains no geodesics of $(\P_2(\R^{2d}),\W_2)$ of positive length.
\end{remark}

\begin{remark}
Instead of studying the gradient flow of the functional $\J$ on the submanifold $\Pi(\mu,\nu)$, it is tempting to alternatively study the gradient flow for the functional $\J+\iota$ on the entire space $\P_2(\R^{2d})$, where $\iota$ is defined to be zero on $\Pi(\mu,\nu)$ and $\infty$ elsewhere. The potential benefit would be to exploit the well-developed theory of gradient flows on $\P_2(\R^{2d})$ in \cite{AGSbook}. The difficulty, however, is that the functional $\iota$ is poorly behaved, for instance not $\lambda$-convex for any $\lambda \in \R$ due to Proposition \ref{pr:geodesicconvexity}.
\end{remark}

\subsection{A route to generalization} \label{se:generalization}

It is natural to wonder to what extent the analysis of the paper can extend beyond the setting of entropic optimal transport, to other constrained optimization problems on the space of probability measures.
Gradient flows in a few other submanifolds of Wasserstein space have been studied before. In \cite{CarlenGangbo} the submanifold is sphere-like, in \cite{lambert2022variational} it is the (geodesically convex) set of Gaussian measures, and in \cite{caglioti2009constrained,eberle2017gradient} it is determined by a finite set of conserved observables.
These settings appear to require case-by-case analysis, except perhaps when the submanifold is geodesically convex and the general theory of \cite{AGSbook} is applicable.

A general theory of gradient flows on submanifolds $\P_2$ appears to be out of reach, given the variety of geometries which might appear, and given the distinct technical challenges faced in our work and in \cite{CarlenGangbo}.
Let us nonetheless outline in this section a general framework which shows some promise, which is the setting of \emph{linear} constraints.

Consider a set $\Pi$ determined by a family of linear constraints, $\Pi=\{P \in \P(\R^d) : \langle P,F\rangle=0, \ \forall F \in \F\}$ for some vector space $\F$ of bounded smooth functions. 
In the case where $\Pi=\Pi(\mu,\nu)$ is the space of couplings, we can take $\F$ to consist of functions of the form $(x,y) \mapsto f(x)-\langle\mu,f\rangle + g(y)-\langle\nu,g\rangle$.
The set $\Pi$ is always convex, in the usual vector space sense, but typically not geodesically convex in $(\P_2(\R^d),\W_2)$.
Similar arguments to those of Section \ref{se:tangentspace} show that, if $\Pi$ is viewed as a submanifold of $\P_2(\R^d)$, then we should identify
\begin{equation*}
\mathrm{Tan}_P\Pi = (\overline{\nabla \F}^{L^2(P;\R^d)})^\perp,
\end{equation*}
for $P \in \Pi$, where $\nabla F = \{\nabla F : F \in \F\}$.  Set $\textsf{S}_P$ to be the $L^2(P;\R^d)$-projection onto $\mathrm{Tan}_P\Pi$.
We are thus led as in Section \ref{se:GF-couplings} to consider the following PDE, as the gradient flow for $H(\cdot\,|\,\rho)$ on the submanifold $\Pi$:
\eqref{def:GF-PDE}:
\begin{equation*}
\partial_t P_t = \nabla \cdot \big(P_t\textsf{S}_{P_t}\nabla \log(P_t/\rho)\big).
\end{equation*}
The corresponding energy decay identity is 
\begin{equation}
\frac{d}{dt}H(P_t\,|\,\rho) = -\int_{\R^d}\big|  \textsf{S}_{P_t} \nabla \log( P_t/\rho)\big|^2\,dP_t. \label{eq:generalization2}
\end{equation}

Beyond these formal considerations, it appears difficult at this level of generality to do any systematic analysis, or to prove anything analogous to Theorem \ref{th:main}, as we rely in many places on specific properties of conditional expectation which are not shared by other projection operators. Specifically, our existence and uniqueness arguments in Section \ref{se:wellposedness} require a certain continuity of $\textsf{S}_P$ with respect to $P$ (as in Proposition \ref{pr:condexp-continuity} below), and the (uniform) boundedness of $\textsf{S}_P$ as an operator on $L^\infty$. Our long-time convergence proof in Section \ref{se:longtime} requires a certain coercivity of the right-hand side of \eqref{eq:generalization2} as a function of $P_t$ (as in Proposition \ref{pr:Istability} below).

\section{Well-posedness of the SDEs} \label{se:wellposedness}

In this section we prove parts (1) and (2) of Theorem \ref{th:main}. They will follow respectively from Theorem \ref{th:wellposed-unbounded} and Lemma \ref{le:mimicking-proj} below. 
In fact, we study a generalization in which $\nabla c$ is replaced by some Borel function $(\fx,\fy) : \R^{2d} \to \R^{2d}$, which need not be a gradient, and this generality will be useful for some approximation arguments. Throughout the section, $(U,V)$ obey the assumptions of Theorem \ref{th:main}, though this could be generalized.

Consider the SDE
\begin{align}
\begin{split}
dX_t &= \Big( \E[\fx(X_t,Y_t) \,|\,X_t]  - \fx(X_t,Y_t) - \nabla U(X_t)\Big)dt + \sqrt{2}dW_t, \\
dY_t &= \Big( \E[\fy(X_t,Y_t) \,|\,Y_t]  - \fy(X_t,Y_t) - \nabla V(Y_t)\Big)dt + \sqrt{2}dB_t.
\end{split} \label{def:mainSDE-proof}
\end{align}
We will show weak existence and uniqueness for this SDE system. By choosing $(\fx,\fy) = \epsilon^{-1}\nabla c$ and performing the time change $t \mapsto \epsilon t$, this will show the claims in Theorems \ref{th:main}(1).

\begin{definition} \label{def:weaksolution}
A ``weak solution" of \eqref{def:mainSDE-proof} is defined to be a filtered probability space $(\Omega,\F,\FF,\PP)$ supporting independent $d$-dimensional $\FF$-Brownian motions $W$ and $B$, a continuous $\FF$-adapted process $(X,Y)$ of dimension $2d$, and two $\FF$-progressively measurable processes $\widehat\fx$ and $\widehat\fy$ each of dimension $d$ such that the following hold:
\begin{itemize}
\item For each $T > 0$,
\begin{align*}
\E\int_0^T\left(\big|\widehat\fx_t - \fx(X_t,Y_t) -\nabla U(X_t)\big| + \big|\widehat\fy_t - \fy(X_t,Y_t) -\nabla V(Y_t)\big|\right)\,dt < \infty.
\end{align*}
\item The SDE holds: 
\begin{align*}
dX_t &= \big( \widehat\fx_t - \fx(X_t,Y_t) - \nabla U(X_t)\big)dt + \sqrt{2}\,dW_t, \\
dY_t &= \big( \widehat\fy_t  - \fy(X_t,Y_t) - \nabla V(Y_t)\big)dt + \sqrt{2}\,dB_t.
\end{align*}
\item We have $\widehat\fx_t = \E[\fx(X_t,Y_t)\,|\,X_t]$ and $\widehat\fy_t = \E[\fy(X_t,Y_t)\,|\,Y_t]$ a.s.\ for a.e.\ $t \ge 0$.
\end{itemize}
For a given initial distribution $P_0$, ``uniqueness in law for initial distribution $P_0$" means that the law of $(X_t,Y_t)_{t \ge 0}$ on $C(\R_+;\R^{2d})$ is the same for any weak solution satisfying $(X_0,Y_0) \sim P_0$.
\end{definition}

The strategy of the existence and uniqueness proof is as follows. 
We begin by proving existence in the case that the functions $(\fx,\fy)$ are bounded. This uses Schauder's fixed point theorem and requires a strong ($L^p$-norm) compactness for the densities of candidate solutions, for which boundedness is helpful in appealing to interior H\"older regularity estimates. To handle the case of unbounded $(\fx,\fy)$, we truncate the coefficients and show that the resulting sequence of laws on path space is Cauchy with respect to the total variation norm. This strong mode of convergence (as opposed to weak convergence) is needed in order to deal with the condition expectation terms and establish that the limit is indeed a solution of the desired equation.

\subsection{Preservation of marginals}
We begin by showing that weak solutions of \eqref{def:mainSDE-proof} leaves the set of couplings $\Pi(\mu,\nu)$ invariant. Part (2) of Theorem \ref{th:main} follows immediately from this lemma, and this structure will be heavily exploited in every step of the well-posedness proofs, with the exception of the proof of existence in the case of bounded $(\fx,\fy)$. Recall that we assume $(\fx,\fy)$ to be Borel, and $(U,V)$ are as in Theorem \ref{th:main}.

Under these assumptions, we note here for later use that the following SDEs  admit unique in law weak solutions starting from any initial position:
\begin{align}
d\bar{X}_t &= -\nabla U(\bar{X}_t)dt + \sqrt{2}\,dW_t, \label{def:X-SDE-proofs} \\
d\bar{Y}_t &= -\nabla V(\bar{Y}_t)dt + \sqrt{2}\,dB_t. \label{def:Y-SDE-proofs}
\end{align}
Indeed,  $(\nabla U,\nabla V)$ are locally Lipschitz, which implies well-posedness up to a possible explosion time. And $\nabla^2 U$ and $\nabla^2V$ are bounded from below in semidefinite order, which is well known to prevent any explosion.

\begin{lemma} \label{le:mimicking-proj}
If $(X,Y)$ is a weak solution of \eqref{def:mainSDE-proof} satisfying $X_0 \sim \mu$ and $Y_0 \sim \nu$, then $X_t \sim \mu$ and $Y_t \sim \nu$ for all $t \ge 0$.
\end{lemma}
\begin{proof}
Apply It\^o's formula to a test function $\testf \in C^\infty_c(\R^d)$ to get
\begin{align*}
\frac{d}{dt}\E[\testf(X_t)] &= \E\bigg[ \nabla\testf(X_t) \cdot \Big( \E[\fx(X_t,Y_t) \,|\,X_t]  - \fx(X_t,Y_t) - \nabla U(X_t)\Big) + \Delta \testf(X_t) \bigg] \\
	&= \E\big[ -\nabla\testf(X_t) \cdot \nabla U(X_t)  + \Delta \testf(X_t) \big].
\end{align*}
In other words, the law $\mu_t$ of $X_t$ is a weak solution of the Fokker-Planck equation associated with the SDE \eqref{def:X-SDE-proofs}:
\begin{align*}
\langle \mu_t,\testf\rangle = \langle \mu_0,\testf\rangle  + \int_0^t \langle \mu_t,-\nabla\testf \cdot \nabla U  + \Delta \testf\rangle, \quad t > 0, \ \testf \in C^\infty_c(\R^d).
\end{align*}
This PDE (with initial condition $\mu$) is uniquely satisfied when $\mu_t=\mu$ for all $t > 0$; this follows from the assumptions on $(U,V)$ and \cite[Theorem 4.1.11]{BKRSbook}. Hence, $X_t \sim \mu$ for all $t > 0$. A similar argument shows $Y_t \sim \nu$.
\end{proof}

\subsection{Preliminary lemmas}  \label{se:preliminaries}

We first quote two ``regularity" properties of conditional expectation. The first deals with measurability: the conditional expectation $\E[\fx(X_t,Y_t)\,|\,X_t]$ can always be viewed as a (a.s.\ uniquely defined) measurable function of $X_t$, for each fixed $t$, but the \emph{joint} measurability is less obvious, or at least less standard.
The following proposition is used implicitly throughout the paper, with conditional expectations always taken to mean a jointly measurable version:

\begin{proposition}[Proposition 5.1 of \cite{BrunickShreve}] \label{pr:measurability}
Suppose a probability space $(\Omega,\F,\PP)$ supports processes $X=(X_t)_{t \ge 0}$ and $Z=(Z_t)_{t \ge 0}$ with values in $\R^d$, which are jointly measurable in $(t,\omega)$. Assume $\E\int_0^T|Z_t|\,dt < \infty$ for each $T > 0$. Then there exists a Borel function $\widehat{z} : \R_+ \times \R^d \to \R^d$ such that $\widehat{z}(t,X_t) = \E[Z_t\,|\,X_t]$ a.s.\ for a.e.\ $t$.
\end{proposition}

We next quote a continuity property of conditional expectation with respect to the measure:

\begin{proposition}[Theorem 3.1 of \cite{CrimaldiPratelli}] \label{pr:condexp-continuity}
Let $(\Omega,\F,\PP)$ be a probability space. Let $X,Z : \Omega \to \R^d$ be measurable, with $Z$ bounded. For $n \in \N$ let $\PP_n$ be a probability measure on $(\Omega,\F)$ with $\PP_n \to \PP$ in total variation. Then $\E_{\PP_n}[Z\,|\,X] \to \E_{\PP}[Z\,|\,X]$ in $L^p(\Omega,\PP;\R^d)$ for every $p \in [1,\infty)$.\footnote{To be completely precise, in case $\PP_n$ is not absolutely continuous with respect to $\PP$ for each $n$, we should say that there exist bounded Borel functions $\hat{z}_n : \R^d \to \R^d$ such that $\hat{z}_n(X)=\E_{\PP_n}[Z\,|\,X]$ holds $\PP_n$-a.s.\ for each $n$, and $\hat{z}_n(X) \to \E_{\PP}[Z\,|\,X]$ in $L^2(\Omega,\PP;\R^d)$. This detail will be inconsequential in applications.}
\end{proposition}

We nex state an entropy estimate which will be used repeatedly in the sequel. It is essentially a simple (and well known) consequence of Girsanov's theorem, but there is some subtlety in the case of unbounded coefficients. In the following, for a Borel measurable $b : [0,T] \times \R^k \to \R^k$, let us say that ``$X$ is a weak solution of the SDE($b,\epsilon$)" if $X$ is a continuous adapted process, defined on some filtered probability space supporting a Brownian motion $W$, such that
\begin{align*}
dX_t = b(t,X_t)dt + \sqrt{2\epsilon}\, dW_t
\end{align*}
and also $\int_0^T |b(t,X_t)|\,dt < \infty$ a.s.
We say the SDE($b,\epsilon$) is \emph{well-posed} if for each $(s,x) \in [0,T] \times \R^k$ there exists a unique in law weak solution of the SDE
\begin{align*}
dX_t = b(t,X_t)dt + \sqrt{2\epsilon}\, dW_t, \ t \in (s,T], \ X_s=x.
\end{align*}

\begin{lemma}[Lemma 4.4(i) of \cite{lacker2021hierarchies}] \label{le:entropy-pathspace}
Let $k \in \N$ and $T > 0$, and let $b^1,b^2 : [0,T] \times \R^k \to \R^k$ be Borel measurable. Suppose that the SDE$(b^2,\epsilon)$ is well-posed.
For each $i=1,2$, let $Z^i$ be a weak solution of the SDE$(b^i,\epsilon)$, and let  $P^i \in \P(C([0,T];\R^k))$ denote its law.  Assume $Z^1_0 \stackrel{d}{=} Z^2_0$, and also
\begin{align}
\E\int_0^T|b^1(t,Z^i_t)-b^2(t,Z^i_t)|^2\,dt < \infty, \quad i=1,2. \label{asmp:entropy-girsanov}
\end{align}
Then
\begin{align*}
H(P^1\,|\,P^2) = \frac{1}{4\epsilon}\E\int_0^T |b^1(t,Z^1_t)-b^2(t,Z^1_t)|^2\,dt.
\end{align*}
\end{lemma}

\begin{remark} \label{re:lingrowth}
The well-posedness assumption for SDE($b^2,\epsilon$) holds if $b^2$ is bounded or is measurable with linear growth. Indeed, this follows from Girsanov's theorem, with existence explained in \cite[Proposition 5.3.6]{KaratzasShreve} and uniqueness in 
\cite[Theorem 7.7]{LiptserShiryaev}.
\end{remark}

We finally state a version of a classical interior H\"older estimate from parabolic PDE theory. It will be used for compactness in the existence proof below, and again for a certain uniform continuity in Section \ref{se:longtime}. 
The essential point is that sufficient local  integrability of the drift implies a Gaussian upper bound on the density, which combines with interior H\"older estimates (\`a la De Giorgi-Nash-Moser, but easier because the diffusion matrix is constant) to yield H\"older estimates in bounded domains, in which the constants depend only on  the integrability of the drift and the size of the domain. See \cite[Corollary 6.4.3]{BKRSbook} for a statement which cover our needs, which we summarize in the following Theorem \ref{th:Holder-estimate}; the assumption $C(r) < \infty$ is precisely what is needed for \cite[Corollary 6.4.3]{BKRSbook}. See also the classical paper of Aronson \cite{aronson1968non} or the book of Lieberman \cite[Sections VI.5--7]{lieberman1996second} for similar results.

\begin{theorem}  \label{th:Holder-estimate}
Let $k \in \N$, $p > k + 2$, and $\sigma > 0$.
Let $b : \R_+ \times \R^k \to \R^d$ be Borel measurable.
Suppose that there exists a weak solution of the SDE
\begin{equation*}
dX_t = b(t,X_t)dt + \sigma dW_t
\end{equation*}
such that 
\begin{equation*}
C(r) := \sup_{t > 0}\bigg(\E\big[b(t,X_t)1_{\{|X_t| \le r\}}\big] + \int_{\{|x| \le r\}}|b(t,x)|^p\,dx\bigg) < \infty, \quad \forall r > 0.
\end{equation*} 
Then the law $P_t$ of $X_t$ admits a density $P_t(\cdot)$ for each $t > 0$. Moreover, for each $\delta,R > 0$ there exist constants $(\alpha,\beta,K)$ depending only on $(C(r))_{r > 0}$ and $(\sigma,R,\delta)$ such that
\begin{align*}
|P(t,x)-P(t',x')| \le K(|x-x'|^\alpha + |t-t'|^{\beta}), 
\end{align*}
for each $x,x' \in \R^k$ with norm at most $R$, and each $t,t' \ge \delta$ with $|t-t'| \le 1$.
\end{theorem}

\subsection{The bounded case}

We are now ready to prove well-posedness in the bounded case:

\begin{proposition} \label{pr:wellposed-bounded}
Suppose the assumptions on $(U,V)$ of Theorem  \ref{th:main} hold, and also that $(\fx,\fy)$ is bounded. Then the SDE  \eqref{def:mainSDE-proof} admits a unique in law weak solution for any initial distribution $P_0 \in \Pi(\mu,\nu)$.
\end{proposition}

\subsection*{Existence in the bounded case}
We first prove existence of a solution, up to any fixed time horizon $T > 0$. Let $\C_T^k=C([0,T];\R^k)$ for each $k \in \N$. For $P \in \P(\C_T^k)$ let $P_t$ denotes the time-$t$ marginal of $P$, i.e., the pushforward through $P$ of the map $\C_T^{k} \ni z \mapsto z_t \in \R^{k}$.

Let $P^* \in \P(\C_T^{2d})$ denote the joint law of $(\bar{X},\bar{Y})$, the unique solution of \eqref{def:X-SDE-proofs} and \eqref{def:Y-SDE-proofs} initialized from $(\bar{X}_0,\bar{Y}_0) \sim P_0$.
For $P \in \P(\C_T^{2d})$, we define $\Phi(P) \in \P(\C_T^{2d})$ as the law of the solution of the SDE
\begin{align*}
dX_t &= \Big( \E_{P}[\fx(X_t,Y_t) \,|\,X_t]  - \fx(X_t,Y_t) - \nabla U(X_t)\Big)dt + \sqrt{2}dW_t, \\
dY_t &= \Big( \E_{P}[\fy(X_t,Y_t) \,|\,Y_t]  - \fy(X_t,Y_t) - \nabla V(Y_t)\Big)dt + \sqrt{2}dB_t.
\end{align*}
Note that this SDE is well-posed by Girsanov's theorem and the well-posedness of the SDEs \eqref{def:X-SDE-proofs} and \eqref{def:Y-SDE-proofs}.
To be clear, we are ``freezing the nonlinearity" here; the measure $P$ governing the conditional expectations is not the same as the law $\Phi(P)$ of the solution. When these two measures match, i.e., when $P$ is a fixed point of $\Phi$, we have a weak solution of the desired SDE.

The conditional expectation $\E_{P}[\fx(X_t,Y_t) \,|\,X_t]$ admits a version (selected arbitrarily) which is a bounded Borel measurable function of $(t,X_t)$, by Proposition \ref{pr:measurability}, and analogously for $\E_{P_t}[\fy(X_t,Y_t) \,|\,Y_t]$. The Radon-Nikodym derivative $d\Phi(P)/dP^*$ is easily identified using Girsanov's theorem, and it follows easily from boundedness of $(\fx,\fy)$ that
\begin{equation}
M_q := \sup_{P \in \P(\C_T^{2d})}\|d\Phi(P)/dP^*\|_{L^q(P^*)} < \infty, \quad \forall q \in [1,\infty). \label{def:Mq}
\end{equation}
It will follow from Schauder's theorem that $\Phi$ admits a fixed point if we can show that $\Phi$ is weakly continuous on the set $K$ defined as the closed convex hull of $\Phi(K_0)$, where $K_0:=\{P \in \P(\C_T^{2d}) : \|dP/dP^*\|_{L^2(P^*)} \le M_2\}$; note that $K$ is weakly compact because $K_0$ is weakly compact and because $\Phi(K_0) \subset K_0$ and by \eqref{def:Mq}.

We next apply Theorem \ref{th:Holder-estimate} to deduce that, for any open ball $S$ in $\R^{2d}$ and any $ s > 0$, the densities $(t,x,y) \mapsto (\Phi(P))_t(x,y)$ are H\"older continuous on $[s,T] \times S$, uniformly over all $P \in \P(\C_T^{2d})$. Indeed, Theorem \ref{th:Holder-estimate} applies because $(\fx,\fy)$ is bounded, and because $(\nabla U,\nabla V)$ is locally bounded.
Moreover, the uniform H\"older continuity implies that  $\{\Phi(P)_T(0,0) : P \in \P(\C_T^{2d})\} \subset \R$ must be bounded, as otherwise otherwise would contradict the fact that $\Phi(P)_t$ are probability measures.
This  H\"older estimate and boundedness are preserved by convex combinations and weak limits and thus satisfied by all measures in $K$.
Hence, if $P^n \to P$ in $K$, it follows from the Arzel\`a-Ascoli theorem that the density $P^n_t(x,y)$ converges to $P_t(x,y)$ uniformly on compact subsets of $(0,T] \times \R^{2d}$. In particular, $P^n_t \to P_t$ in total variation.

Now, to prove continuity of $\Phi$ on $K$, let $P^n,P \in K$ with $P^n \to P$ weakly. 
As noted in the previous paragraph, it holds that $P^n\to P$ in total variation.
Using the entropy estimate of Lemma \ref{le:entropy-pathspace}, and letting $(X,Y)=(X_t,Y_t)_{t \in [0,T]}$ denote the canonical process on $\C_T^{2d}$,  we have
\begin{align*}
H(\Phi(P)\,|\,\Phi(P^n))	&= \frac14\int_0^T\E_{\Phi(P)}\bigg[ \Big|\E_{P}[\fx(X_t,Y_t) \,|\,X_t] - \E_{P^n}[\fx(X_t,Y_t) \,|\,X_t]\Big|^2 \\
	&\qquad\qquad\qquad\quad +\Big|\E_{P}[\fy(X_t,Y_t) \,|\,Y_t] - \E_{P^n}[\fy(X_t,Y_t) \,|\,Y_t]\Big|^2\bigg]\,dt.
\end{align*}
We claim that this vanishes as $n\to\infty$, which will show that $\Phi(P^n) \to \Phi(P)$.
Because $\fx$ is bounded, it follows from Proposition \ref{pr:condexp-continuity} that
\begin{align*}
\E_{P^n}[\fx(X_t,Y_t) \,|\,X_t] \to \E_{P}[\fx(X_t,Y_t) \,|\,X_t]
\end{align*}
in $P$-measure. Because $P_t$ and $\Phi(P)_t$ have positive (Lebesgue) densities, the above convergence holds in $\Phi(P)$-measure as well.  The claim now follows from dominated convergence. \hfill \qedsymbol

\subsection*{Uniqueness in the bounded case}
Technically, this is a special case of the uniqueness proof given in the unbounded case for Theorem \ref{th:wellposed-unbounded} below, but it is helpful to understand the main idea first in the simpler case of bounded coefficients.
Consider two solutions with the same initial condition,  with laws $P,Q \in \P(\C_T^{2d})$. Let $P[t] \in \P(\C_t^{2d})$ denote the image under the restriction map $z \mapsto z|_{[0,t]}$, for $t \in [0,T]$. 
Note that $X_t \sim \mu$ and $Y_t \sim \nu$ under each $P$ and $Q$, for each $t$.
We then have by Lemma \ref{le:entropy-pathspace} that
\begin{align*}
 H(Q[t]\,|\,P[t]) = \frac{1}{4} \int_0^t\E_Q\Big[ &\big| \E_P[\fx(X_s,Y_s)\,|\,X_s]- \E_Q[\fx(X_s,Y_s)\,|\,X_s]\big|^2 \\
 	&+ \big|\E_P[\fy(X_s,Y_s)\,|\,Y_s]-\E_Q[\fy(X_s,Y_s)\,|\,Y_s]\big|^2\Big]\,ds.
\end{align*}
Let $P_{s,X_s}$ denote the (regular) conditional law of $Y_s$ given $X_s$ under $P$, and similarly for $Q$. Using Pinsker's inequality, we have
\begin{align*}
\E_Q \big| \E_P[\fx(X_s,Y_s)\,|\,X_s]- \E_Q[\fx(X_s,Y_s)\,|\,X_s]\big|^2 &= \E_Q \big| \langle P_{s,X_s} - Q_{s,X_s}, \fx(X_s,\cdot)\rangle\big|^2 \\
	&\le 2\||\fx|^2\|_\infty \E_Q \big[H(Q_{s,X_s}\,|\,P_{s,X_s})\big].
\end{align*}
Lemma \ref{le:mimicking-proj} ensures that $Q_s,P_s\in\Pi(\mu,\nu)$. Using the chain rule followed by the data processing inequality of relative entropy,
\begin{align*}
\E_Q \big[H(Q_{s,X_s}\,|\,P_{s,X_s}) \big] = H(Q_s\,|\,P_s) \le H(Q[s]\,|\,P[s]).
\end{align*}
Combine the last two inequalities, along with the analogous inequalities for the $\fy$ term, to get
\begin{align*}
 H(Q[t]\,|\,P[t]) &\le  \kappa \int_0^t H(Q_s\,|\,P_s)\,ds \le \kappa\int_0^t H(Q[s]\,|\,P[s])\,ds,
\end{align*}
where $\kappa = (\||a|^2\|_\infty + \||b|^2\|_\infty)/2$.
Uniquenes now follows from Gronwall's inequality. \hfill \qedsymbol

\subsection{The unbounded case}

With well-posedness proven for bounded $(\fx,\fy)$, we now turn to the unbounded case. Part (1) of Theorem \ref{th:main} follows immediately from the following:

\begin{theorem} \label{th:wellposed-unbounded}
Suppose the assumptions on $(U,V)$ of Theorem  \ref{th:main} hold. Assume that the functions $\fx$ and $\fy$ are measurable, and that there exists $r_* > 0$ such that
\begin{align}
\sup_{Q \in \Pi(\mu,\nu)} \int_{\R^d \times \R^d} \big(e^{r_*|\fx(x,y)|^2} + e^{r_*|\fy(x,y)|^2}\big)Q(dx,dy) < \infty. \label{asmp:subgaussian}
\end{align}
Then the SDE  \eqref{def:mainSDE-proof} admits a unique in law weak solution for any initial distribution in $\Pi(\mu,\nu)$.
Moreover, if $P_t(x,y)$ denotes the time-$t$ density of the solution, then
for each $\delta > 0$ and each compact set $S \subset \R^{2d}$, we may find $K,\alpha,\beta > 0$ such that 
\begin{align*}
|P_t(x,y)-P_t(x',y')| \le K(|x-x'|^\alpha + |t-t'|^\beta),
\end{align*}
for each $x,x' \in S$ and each $t,t' \ge \delta$ with $|t-t'| \le 1$.
\end{theorem}

A typical example of when the assumption \eqref{asmp:subgaussian} holds is when $\mu$ and $\nu$ are subgaussian and $(\fx,\fy)$ have linear growth, in the sense that for some constant $L$ we have
\begin{equation}
|\fx(x,y)|+|\fy(x,y)| \le L(1+|x|+|y|), \quad \forall x,y \in \R^d. \label{asmp:lingrowth}
\end{equation}

\subsection*{Existence in the unbounded case} \label{se:existence-unbounded}

We prove existence by truncating $(a,b)$, deducing existence from the bounded case of Proposition \ref{pr:wellposed-bounded}, and then carefully taking limits. Fix $T > 0$ arbitrarily.
For each $n \in \N$ and $(x,y) \in \R^{2d}$, let $\fx_n(x,y) \in \R^d$ denote the projection of $\fx(x,y)$ onto the centered ball of radius $n$. Define $\fy_n(x,y)$ similarly. Then $(\fx_n,\fy_n)$ is bounded, so for each $n$ there exists a unique solution of the corresponding SDE \eqref{def:mainSDE-proof}, by Proposition \ref{pr:wellposed-bounded}. Let $P^n \in \P(\C_T^{2d})$ denote the law  of the solution, where again $\C_T^{2d} := C([0,T];\R^{2d})$.

We first establish compactness. Let $P^* \in \P(\C_T^{2d})$ again denote the joint law of $(\bar{X},\bar{Y})$, the unique solution of \eqref{def:X-SDE-proofs} and \eqref{def:Y-SDE-proofs} initialized from $(\bar{X}_0,\bar{Y}_0) \sim P_0$. Using Lemma \ref{le:entropy-pathspace}, we have
\begin{align*} 
H(P^n[T]\,|\,P^*[T]) &=  \frac14 \int_0^T\E_{P^n}\Big[ \big|\E_{P^n}[\fx_n(X_t,Y_t)\,|\,X_t]-\fx_n(X_t,Y_t)\big|^2  \\
	&\qquad\qquad\qquad + \big|\E_{P^n}[\fy_n(X_t,Y_t)\,|\,Y_t]-\fy_n(X_t,Y_t)\big|^2\Big]\,dt \\
	&\le \frac14 \int_0^T\E_{P^n}\Big[|\fx_n(X_t,Y_t)|^2 + |\fy_n(X_t,Y_t)|^2\Big]\,dt \\
	&\le \frac14 \int_0^T\E_{P^n}\Big[|\fx(X_t,Y_t)|^2 + |\fy(X_t,Y_t)|^2\Big]\,dt.
\end{align*}
This quantity is bounded uniformly in $n$ thanks to the assumption \eqref{asmp:subgaussian}. Because sub-level sets of entropy are weakly compact, we deduce that $(P^n[T])$ is a tight sequence in $\P(\C_T^{2d})$ for each $T > 0$, and thus $(P^n)$ is a tight sequence in $\P(C(\R_+;\R^{2d}))$, where $C(\R_+;\R^{2d})$ is equipped with the topology of uniform convergence on compacts.

We next apply Theorem \ref{th:Holder-estimate} to establish stronger compactness. Let $\widehat{\fx}^n_t(x)=\E_{P^n}[\fx(X_t,Y_t)\,|\,X_t=x]$. Then, since $P^n_t \in \Pi(\mu,\nu)$, Jensen's inequality yields
\begin{align*}
\int_{\R^d}|\widehat{\fx}^n_t|^p\,d\mu \le \E_{P^n}|\fx(X_t,Y_t)|^p \le \sup_{P \in \Pi(\mu,\nu)} \int_{\R^{2d}}|\fx|^p\,dP =: C_{\fx}^p, \ \ \forall p \ge 1, \ t \ge 0.
\end{align*}
Let $B^k_r$ denote the centered ball of radius $r$ in $\R^k$, and let $|B^k_r|$ denote its Lebesgue measure, for any $k \in \N$.
Because $\mu$ is locally bounded away from zero, we deduce that
\begin{align*}
 \int_{B_r^{2d}}|\widehat{\fx}^n_t(x)|^p\,dxdy \le  C_{\mu,r}|B^d_r| \int_{B_r^d}|\widehat{\fx}^n_t(x)|^p\,\mu(dx)  \le C_{\mu,r}C_{\fx}^p,
\end{align*}
where $C_{\mu,r}=1/\inf\{\mu(x) : x \in B_r^d\} > 0$. This and the analogous estimate for the conditional expectation terms involving $\fy$ allow us to apply Theorem \ref{th:Holder-estimate}, with the constants being uniform with respect to $n$. That is, on compact subsets of $(0,\infty) \times \R^{2d}$, the densities $(t,x,y) \mapsto P^n_t(x,y)$ are uniformly H\"older continuous, and they are thus also uniformly bounded because they are probability measures. From this it follows that $\{P^n_t : n \in \N\}$ is norm-precompact in $L^1(\R^{2d})$.

To summarize: By tightness, we may extract a  weakly convergent subsequence (relabeled) of $P^n \to P$, and by strong compactness $P^n_t$ converges in total variation to $P_t$ for each $t > 0$.

We show next that $P$ is the law of a solution of the desired SDE, by using the martingale problem formulation of Stroock and Varadhan \cite{StroockVaradhan}. Let $\testf \in C^\infty_c(\R^{2d})$.
Under $P^n$, the process $(M^n_t)_{t \in [0,T]}$ defined by
\begin{align*}
\testf(X_t,Y_t) - \int_0^t &\bigg( \big(\E_{P^n}[\fx_n(X_s,Y_s)\,|\,X_s] - \fx_n(X_s,Y_s) - \nabla U(X_s)\big) \cdot \nabla_x\testf(X_s,Y_s)  + \Delta_x \testf(X_s,Y_s) \\
& + \big(\E_{P^n}[\fy_n(X_s,Y_s)\,|\,Y_s] - \fy_n(X_s,Y_s) - \nabla V(Y_s)\big) \cdot \nabla_y\testf(X_s,Y_s)  + \Delta_y \testf(X_s,Y_s)\bigg)ds
\end{align*}
is a $P^n$-martingale. Similarly, define $M_t$ to be
\begin{align*}
\testf(X_t,Y_t) - \int_0^t &\bigg( \big(\E_{P}[\fx(X_s,Y_s)\,|\,X_s] - \fx(X_s,Y_s) - \nabla U(X_s)\big) \cdot \nabla_x\testf(X_s,Y_s)  + \Delta_x \testf(X_s,Y_s) \\
& + \big(\E_{P}[\fy(X_s,Y_s)\,|\,Y_s] - \fy(X_s,Y_s) - \nabla V(Y_s)\big) \cdot \nabla_y\testf(X_s,Y_s)  + \Delta_y \testf(X_s,Y_s)\bigg)ds.
\end{align*}
We will show that $M$ is a $P$-martingale. An intermediate process will be useful: define $\widetilde{M}^n_t$ exactly like $M^n_t$ but with $(\fx_n,\fy_n)$ replaced by $(\fx,\fy)$ (though conditional expectations are still under $P^n$). Fix $0 \le t_1 < t_2 \le T$, and let $Z$ be any bounded $(X_s,Y_s)_{s \le t_1}$-measurable random variable. We claim that
\begin{align}
\lim_{n\to\infty}\E_{P^n}[(M^n_{t_2}-M^n_{t_1})Z] = \E_P[(M_{t_2}-M_{t_1})Z]. \label{pf:M-limit}
\end{align}
Since the left-hand side is zero, this will show as desired that $M$ is a $P$-martingale. To show \eqref{pf:M-limit}, we first note that $\E_{P^n}|M^n_t-\widetilde{M}^n_t| \to 0$ because
\begin{align}
\E_{P^n}|(\fx-\fx_n)(X_t,Y_t)| &= \E_{P^n}[| \fx(X_t,Y_t)|1_{\{|\fx(X_t,Y_t)| > n\}}] \to 0, \label{pf:EPn|a-a_n|->0}
\end{align}
which follows from the assumption \eqref{asmp:subgaussian}. Hence, we must only show that
\begin{align}
\lim_{n\to\infty}\E_{P^n}[(\widetilde{M}^n_{t_2}-\widetilde{M}^n_{t_1})Z] = \E_P[(M_{t_2}-M_{t_1})Z]. \label{pf:tildeM-limit}
\end{align}
To this end, note that $\widetilde{M}^n_t$ and $M_t$ differ only in the conditional expectations terms. Two of the terms where they do not differ, $\nabla U \cdot \nabla_x\testf$ and $\nabla V \cdot \nabla_y\testf$, are \emph{bounded} functions of $(X_t,Y_t)$ since $\testf$ has compact support and $(\nabla U,\nabla V)$ are locally bounded. The other term where $\widetilde{M}^n_t$ and $M_t$ do not differ is $\fx \cdot \nabla_x\testf + \fy \cdot\nabla_y\testf$, which may not be bounded but is uniformly integrable over $(n,t)$ thanks to the assumption \eqref{asmp:subgaussian}. Hence, the only difficulty is to show for each $t$ and each bounded random vector $Z$  we have
\begin{align}
\E_{P^n}[Z \cdot \E_{P^n}[\fx(X_t,Y_t)\,|\,X_t]] \to \E_{P}[Z\cdot \E_{P}[\fx(X_t,Y_t)\,|\,X_t]], \label{pf:condexp-conv1}
\end{align}
and similarly for the $\fy$ term.

We first note that \eqref{pf:condexp-conv1} is true if $\fx$ is bounded. Indeed, then $\E_{P^n}[\fx(X_t,Y_t)\,|\,X_t]$ is uniformly bounded in $n$, and the total variation convergence $P^n \to P$ implies
\begin{align*}
\left|\E_{P^n}[Z \cdot \E_{P^n}[\fx(X_t,Y_t)\,|\,X_t]] - \E_{P}[Z \cdot \E_{P^n}[\fx(X_t,Y_t)\,|\,X_t]]\right| \to 0.
\end{align*}
The total variation convergence $P^n \to P$ also implies
\begin{align*}
\E_{P}&\left|\E_{P^n}[\fx(X_t,Y_t)\,|\,X_t] - \E_{P}[\fx(X_t,Y_t)\,|\,X_t]\right|  \to 0
\end{align*}
by Proposition \ref{pr:condexp-continuity}.
Hence, \eqref{pf:condexp-conv1} holds if $\fx$ is bounded.
In the general case, we proceed by truncation. Let  $r > 0$. We have
\begin{align*}
\E_{P^n}[Z \cdot \E_{P^n}[\fx(X_t,Y_t)\,|\,X_t]] &= \E_{P^n}[Z \cdot \E_{P^n}[\fx(X_t,Y_t)1_{\{|\fx(X_t,Y_t)| \le r\}}\,|\,X_t]] \\
	&\quad + \E_{P^n}[Z \cdot \E_{P^n}[\fx(X_t,Y_t) 1_{\{|\fx(X_t,Y_t)| > r\}} \,|\,X_t]].
\end{align*}
The first term converges to $\E_{P}[Z \cdot \E_{P}[\fx(X_t,Y_t)1_{\{|\fx(X_t,Y_t)| \le r\}}\,|\,X_t]]$ by the argument for the bounded case. The second term vanishes as $r\to\infty$, uniformly in $(n,t)$, thanks to the boundedness of $Z$ and the integrability assumption \eqref{asmp:subgaussian}, which proves \eqref{pf:condexp-conv1} in the general case.

We have thus shown that $(M_t)_{t \in [0,T]}$ is a martingale, for each choice of $\testf \in C^\infty_c(\R^{2d})$. By the usual equivalence between weak solutions of SDEs and martingale problems \cite[Section V.20]{RogersWilliams}, we deduce that $P$ is in fact the law of a solution to \eqref{def:mainSDE}.

\subsection*{Uniqueness in the unbounded case}

We lastly prove uniqueness, extending the argument given for the case of bounded $(\fx,\fy)$. 
Let $P=\lim_{n\to\infty}P^n$ denote the solution law constructed above (with the convergence holding along a subsequence). Let $Q \in \P(C(\R_+;\R^{2d}))$ be the law of any other solution.
Note that $Q$, $P$, and $P^n$ for each $n\in \N$ all have the same marginals $X_t \sim \mu$ and $Y_t \sim \nu$, for each $t \ge 0$.
We first claim that
\begin{align}
\begin{split}
H(Q[T]\,|\,P[T]) \le \frac14 \int_0^T\E_Q\Big[ &\big|\E_P[\fx(X_t,Y_t)\,|\,X_t] - \E_Q[\fx(X_t,Y_t)\,|\,X_t]\big|^2 \\
 	&+ \big|\E_P[\fy(X_t,Y_t)\,|\,Y_t]-\E_Q[\fy(X_t,Y_t)\,|\,Y_t]\big|^2 \Big]\,dt. 
\end{split} \label{pf:uniq-ent-est1}
\end{align}

\begin{proof}[Proof of \eqref{pf:uniq-ent-est1}]
Unfortunately, it is not clear that this follows directly from Lemma \ref{le:entropy-pathspace};  the well-posedness assumption therein is difficult to check because $\E_P[\fx(X_t,Y_t)\,|\,X_t]$ may not have much regularity as a function of $X_t$. Instead, we will first apply Lemma \ref{le:entropy-pathspace} with $P^n$ in place of $P$, and then take limits. Applying Lemma \ref{le:entropy-pathspace} yields
\begin{align}
H(Q[T]\,|\,P^n[T]) \le \ &\frac14 \E_Q\int_0^T\Big( \Big|\E_Q[\fx(X_t,Y_t)\,|\,X_t] - \E_{P^n}[\fx_n(X_t,Y_t)\,|\,X_t] -(\fx-\fx_n)(X_t,Y_t)\Big|^2 \nonumber \\
	&+ \Big|\E_Q[\fy(X_t,Y_t)\,|\,Y_t] - \E_{P^n}[\fy_n(X_t,Y_t)\,|\,Y_t] -(\fy-\fy_n)(X_t,Y_t)\Big|^2\Big)\,dt. \label{pf:uniq-ent-est1-approx}
\end{align}
Lower semicontinuity of relative entropy yields $H(Q[T]\,|\,P[T]) \le \liminf_{n\to\infty}H(Q[T]\,|\,P^n[T])$. To take limits on the right-hand side, we claim that
\begin{align}
\lim_{n\to\infty} \E_Q\big|\E_{P}[\fx(X_t,Y_t)\,|\,X_t] - \E_{P^n}[\fx_n(X_t,Y_t)\,|\,X_t]\big|^2 = 0. \label{pf:ent-est-approx1}
\end{align}
To see this, we note that it follows from the equality of $X_t$-marginals under $Q$ and $P^n$ and from \eqref{pf:EPn|a-a_n|->0} that
\begin{align*}
\E_Q\big|\E_{P^n}[\fx(X_t,Y_t)\,|\,X_t] - \E_{P^n}[\fx_n(X_t,Y_t)\,|\,X_t]\big|^2 &\le \E_{P^n}\big|(\fx - \fx_n)(X_t,Y_t) \big|^2 \to 0.
\end{align*}
Moreover, since $P^n \to P$ in total variation, we have by Proposition \ref{pr:condexp-continuity} that
\begin{align*}
\E_Q\big|\E_{P}[\fx(X_t,Y_t)\,|\,X_t] - \E_{P^n}[\fx(X_t,Y_t)\,|\,X_t]\big| = \E_P\big|\E_{P}[\fx(X_t,Y_t)\,|\,X_t] - \E_{P^n}[\fx(X_t,Y_t)\,|\,X_t]\big| &\to 0,
\end{align*}
using a truncation argument similar to that leading to \eqref{pf:condexp-conv1}.
Combining the last two displays proves \eqref{pf:ent-est-approx1}. In addition, we have seen by now how to use the given marginals of $X_t$ and $Y_t$ along with the finite moments implied by \eqref{asmp:subgaussian} to deduce limits like $\E_Q|(\fx-\fx_n)(X_t,Y_t)|^2 \to 0$. We may now safely take limits on the right-hand side of\eqref{pf:uniq-ent-est1-approx} to deduce \eqref{pf:uniq-ent-est1}.
\end{proof}

We now start the main line of the uniqueness proof, beginning from the estimate \eqref{pf:uniq-ent-est1}. 
Let $P_{t,X_t}$ denote the regular conditional law of $Y_t$ given $X_t$ under $P$, and similarly for $Q$.
Let $r > 0$ and use the weighted Pinsker inequality \cite[Theorem 2.1(ii)]{bolley2005weighted} to get
\begin{align*}
\big|\E_P[\fx(X_t,Y_t)\,|\,X_t] - \E_Q[\fx(X_t,Y_t)\,|\,X_t]\big|^2 &\le \left(2r^{-1} + \Lambda_r(t,X_t) \right)H(Q_{t,X_t}\,|\,P_{t,X_t}),
\end{align*}
where we define
\begin{align*}
\Lambda_r(t,X_t) := \frac{2}{r}\log \E_P \Big[e^{r|\fx(X_t,Y_t) - \E[\fx(X_t,\cdot)\,|\,X_t]|^2}  \,\Big|\, X_t\Big].
\end{align*}
By the chain rule for relative entropy, $ \E_Q H(Q_{t,X_t}\,|\,P_{t,X_t}) = H(Q_t\,|\,P_t)$.
Hence, for any $m > 0$,
\begin{align}
\E_Q &\big|\E_P[\fx(X_t,Y_t)\,|\,X_t] - \E_Q[\fx(X_t,Y_t)\,|\,X_t]\big|^2 \label{pf:uniq-ent-est1-2-1} \\
	&\le (2r^{-1}+m) H(Q_t\,|\,P_t) + \E_Q\Big[ \big|\E_P[\fx(X_t,Y_t)\,|\,X_t] - \E_Q[\fx(X_t,Y_t)\,|\,X_t]\big|^2 1_{\{\Lambda_r(t,X_t) > m\}}\Big]. \nonumber
\end{align}

To bound the second term in \eqref{pf:uniq-ent-est1-2-1}, note that $Q$ and $P$ have the same $X_t$-marginal. Use Jensen and Cauchy-Schwarz to get
\begin{align*}
\E_Q &\Big[ \big|\E_P[\fx(X_t,Y_t)\,|\,X_t] - \E_Q[\fx(X_t,Y_t)\,|\,X_t]\big|^2 1_{\{\Lambda_r(t,X_t) > m\}}\Big] \\
	&\le 2\E_Q\Big[ \big|\E_Q[\fx(X_t,Y_t)\,|\,X_t]\big|^2 1_{\{\Lambda_r(t,X_t) > m\}}\Big] + 2\E_P\Big[ \big|\E_P[\fx(X_t,Y_t)\,|\,X_t]\big|^2 1_{\{\Lambda_r(t,X_t) > m\}}\Big] \\
	&\le  2\E_Q\left[ |\fx(X_t,Y_t)|^2 1_{\{\Lambda_r(t,X_t) > m\}}\right] + 2\E_P\left[ |\fx(X_t,Y_t)|^2 1_{\{\Lambda_r(t,X_t) > m\}}\right] \\
	&\le C^{\fx}_t P( \Lambda_r(t,X_t) > m)^{1/2}.
\end{align*}
where we define
\begin{align*}
C^{\fx}_t := 2\E_Q\left[ |\fx(X_t,Y_t)|^4\right]^{1/2} + 2\E_P\left[ |\fx(X_t,Y_t)|^4\right]^{1/2}.
\end{align*} 
Use the inequalities of Markov, Jensen, and Cauchy-Schwarz to get
\begin{align*}
P( \Lambda_r(t,X_t) > m) &\le P\left( \E_P\left[ e^{ r|\fx(X_t,Y_t) - \E_P[\fx(X_t,Y_t)\,|\,X_t]|^2} \,\bigg|\,X_t\right] > e^{rm/2}\right) \\
	&\le e^{-rm/2}\E_P\left[ e^{ r|\fx(X_t,Y_t) - \E_P[\fx(X_t,Y_t)\,|\,X_t]|^2} \right] \\
	&\le e^{-rm/2}\E_P\left[ e^{ 4r|\fx(X_t,Y_t)|^2} \right].
\end{align*}
Finally, apply our bounds for the right-hand side of \eqref{pf:uniq-ent-est1-2-1} in the original inequality  returning to \eqref{pf:uniq-ent-est1}, and similarly for the $\fy$ term. This yields 
\begin{align*}
H(Q[T]\,|\,P[T]) \le \frac14\int_0^T \bigg( &2(2r^{-1}+m)H(Q_t\,|\,P_t) + C^{\fx}_t e^{-rm/4}\E_P\left[ e^{ 4r|\fx(X_t,Y_t)|^2} \right]^{1/2} \\
	& + C^{\fy}_t e^{-rm/4}\E_P\left[ e^{ 4r|\fy(X_t,Y_t)|^2} \right]^{1/2} \bigg)\,dt,
\end{align*}
where $C^{\fy}_t$ is defined exactly like $C^{\fx}_t$ except with $\fx$ replaced by $\fy$.  
The assumption \eqref{asmp:subgaussian} easily implies that $C:= \sup_{t \ge 0} (C^{\fx}_t \vee C^{\fy}_t) < \infty$.
Using $H(Q_t\,|\,P_t) \le H(Q[t]\,|\,P[t])$ and Gronwall, we deduce
\begin{align*}
H(Q[T]\,|\,P[T]) \le Ce^{ r^{-1} T + \frac14(2T-r)m } \int_0^T  \bigg( \E_P\left[ e^{ 4r|\fx(X_t,Y_t)|^2} \right]^{1/2} +  \E_P\left[ e^{ 4r|\fy(X_t,Y_t)|^2} \right]^{1/2} \bigg)\,dt.
\end{align*}
The two expectations on the right-hand side are bounded uniformly in $t$ if $r \le r_*/4$ by the assumption \eqref{asmp:subgaussian}. Hence, if we choose such an $r$ and then choose $T < r/2$, we may send $m\to\infty$ to get $H(Q[T]\,|\,P[T])=0$.

This proves uniqueness on a sufficiently small time interval. The length $T$ of this interval depends only on $r_*$, and, in particular, not on the initial distribution $P_0$. We may thus apply the same argument on successive time intervals to deduce that $Q[T]=P[T]$ for any $T > 0$, proving that $Q=P$. \hfill\qedsymbol

\subsection{Regularity of the density of the solution} \label{se:entropycomesdown}

In this section we collect some regularity properties of the solution of \eqref{def:mainSDE-proof} which will be useful in Section \ref{se:longtime}. They follow quickly from known results on regularization effects of Fokker-Planck equations, borrowed from \cite{BKRSbook}. 
We first to show that the entropy immediately ``comes down from infinity," in the sense that $H(P_t\,|\,\mu \otimes \nu)$ is finite for all $t>0$ even if not for $t=0$. 

\begin{lemma} \label{le:entropycomesdown}
Suppose the assumptions of Theorem \ref{th:wellposed-unbounded} hold.
For a given initial distribution $P_0 \in \Pi(\mu,\nu)$, let $P_t$ denote the time-$t$ law of the unique solution of the SDE \eqref{def:mainSDE-proof}. Then $P_t(x,y)$ is uniformly bounded on $(t,x,y) \in [1,\infty) \times \R^{2d})$, and $H(P_t\,|\,\mu \otimes \nu) < \infty$ for a.e.\ $t > 0$.
\end{lemma}
\begin{proof}
In order to apply known regularity results, we need to establish some integrability for the drift appearing in \eqref{def:mainSDE-proof},
\begin{align*}
\mathcal{D}(t,x,y)  := \begin{pmatrix}
\E[\fx(X_t,Y_t) \,|\,X_t=x]  - \fx(x,y) - \nabla U(x) \\
\E[\fy(X_t,Y_t) \,|\,Y_t=y]  - \fy(x,y) - \nabla V(y)
\end{pmatrix}.
\end{align*}
This is where we make use of the assumptions of the fourth bullet point of Theorem \ref{th:main}, which we recall here:
\begin{enumerate}[(i)]
\item $\nabla U \in L^p(\mu)$ and $\nabla V \in L^p(\nu)$ for some $p > 2(d + 1)$.
\item $\mu$ and $\nu$ have finite differential entropy, i.e., $U \in L^1(\mu)$ and $V \in L^1(\nu)$.
\end{enumerate}
Using Jensen's inequality, the integrability assumption \eqref{asmp:subgaussian}, and the property (i), we easily find
\begin{align*}
\sup_{t \ge 0}\E|\mathcal{D}(t,X_t,Y_t)|^p < \infty.
\end{align*}
This allows us to apply
Corollary 7.2.2 of \cite{BKRSbook} to deduce the claimed boundedness of the density on $[1,\infty) \times\R^{2d}$. Note that the density is continuous \cite[Corollary 6.4.3]{BKRSbook}, so this is truly a uniform bounded, not just an a.e.\ bound.
It follows that
\begin{align*}
\int_{\R^{2d}} P_t(z)\log^+ P_t(z)\,dz \le  \log^+\|P_t\|_\infty < \infty,
\end{align*}
for $t \ge 1$, and thus
\begin{align*}
H(P_t\,|\,\mu \otimes \nu) &= \int_{\R^{2d}} P_t(x,y)\log \frac{P_t(x,y)}{\mu(x)\nu(y)}\,dxdy \\
	&\le  \log^+ \|P_t\|_\infty  - \int_{\R^{d}}\mu(x)\log\mu(x) dx - \int_{\R^{d}}\nu(y)\log \nu(y)dy
\end{align*}
is finite, because the final two differential entropy terms are finite by  (ii). 
\end{proof}

A final lemma shows how to use smoothness of the densities allows to upgrade weak convergence of $P_t$ as $t\to\infty$ to convergence in relative entropy:

\begin{lemma} \label{le:entropyconvergence}
Suppose the assumptions of Theorem \ref{th:wellposed-unbounded} hold. For a given initial distribution $P_0 \in \Pi(\mu,\nu)$, let $P_t(x,y)$ denote the time-$t$ density of the unique solution of the SDE \eqref{def:mainSDE-proof}. Suppose $P_t$ converges weakly to some $P_\infty \in \P(\R^{2d})$ as $t\to\infty$, where $P_\infty$ has a positive continuous density satisfying
\begin{equation}
\sup_{t \ge 1}\int |\log P_\infty|^2dP_t < \infty. \label{asmp:entropyconvergence}
\end{equation}
Then $H(P_t\,|\,P_\infty) \to 0$ as $t\to\infty$.
\end{lemma}
\begin{proof}
Using the H\"older estimate from Theorem \ref{th:wellposed-unbounded}, for any compact set $S \subset \R^{2d}$, the set $(P_t|_{S})_{t \ge 1}$ is pre-compact  in $C(S)$. By the Arzel\`a-Ascoli theorem and a diagonal argument, for any diverging sequence of times there exists a locally uniformly convergent subsequence. By the assumed weak convergence, these subsequential limits must coincide with $P_\infty$, and we deduce the local uniform convergence of $P_t \to P_\infty$. Next, recall from Lemma \ref{le:entropycomesdown} that $P_t(x,y)$ is bounded on $[1,\infty) \times \R^{2d}$, say by a constant $C$. 
For any $r > 0$, write
\begin{equation*}
H(P_t\,|\,P_\infty) = \int_{B_r^{2d}} P_t \log \frac{P_t }{P_\infty } + \int_{(B_r^{2d})^c}  P_t \log \frac{P_t }{P_\infty}.
\end{equation*}
The first term vanishes as $t\to\infty$, for each $r > 0$, thanks to the local uniform convergence and the fact that $P_\infty$ is bounded away from zero on $B_r^{2d}$. For the second term, note for any $\delta > 0$ that
\begin{equation*}
\log \frac{P_t}{P_\infty} \le \log C + \frac{1}{2\delta} + \frac{\delta}{2}|\log P_\infty|^2.
\end{equation*}
From this we deduce that
\begin{equation*}
\limsup_{t \to\infty}H(P_t\,|\,P_\infty) \le \Big(\log C + \frac{1}{2\delta}\Big)\limsup_{t\to\infty}P_t((B_r^{2d})^c) + \frac{\delta}{2}\sup_{t \ge 1}\int |\log P_\infty|^2dP_t.
\end{equation*}
By tightness, the first term on the right-hand side vanishes as $r \to \infty$. Then send $\delta \to 0$ so that the last term vanishes by assumption \eqref{asmp:entropyconvergence}.
\end{proof}

\section{Long-time convergence} \label{se:longtime}

This section proves the convergence claims (3,4) in Theorems \ref{th:main}.
Recall the formula for $\pi$ from \eqref{intro:potentials}, and recall the definitions of the relevant variant of relative Fisher information: For $P\in\P(\R^{2d})$ with $R = \log dP/d\pi$, define
\begin{align*}
\overline{I}(P\,|\,\pi) &:= \E_{P}\Big[  \big|\nabla_x R - \E_P[\nabla_x R\,|\,X]\big|^2 + \big|\nabla_y R - \E_P[\nabla_y R\,|\,Y]\big|^2 \Big],
\end{align*}
if the weak gradient $\nabla R$ exists and belongs to $L^2(P)$, and otherwise $\overline{I}(P\,|\,\pi)=\infty$.
Throughout this section, we use the same notation $(X_t,Y_t)$ and $(P_t)_{t \ge 0}$ as in Theorem \ref{th:main}. It will be convenient to define also
\begin{align}
f_t(x) := \E[\nabla_xc(X_t,Y_t)\,|\,X_t=x], \qquad g_t(y) := \E[\nabla_yc(X_t,Y_t)\,|\,Y_t=y]. \label{def:ft,gt}
\end{align}
The proofs will be based on three key ingredients, which we state first in the following lemmas and proposition below.

\begin{lemma}[Entropy dynamics] \label{le:entropydynamics}
For each $t > s \ge 0$, we have $H(P_t\,|\,\pi) < \infty$ and
\begin{align}
H(P_t\,|\,\pi) + \epsilon \int_s^t \overline{I}(P_r\,|\,\pi)\,dr \le H(P_s\,|\,\pi). \label{def:entropydecay}
\end{align}
\end{lemma}

\begin{lemma}[Uniform continuity] \label{le:W1continuity}
The following maps are uniformly continuous:
\begin{align*}
[0,\infty) &\ni t \mapsto P_t \in (\P_2(\R^{2d}),\W_2) \\
[1,\infty) &\ni t \mapsto f_t \in L^2(\mu) \\
[1,\infty) &\ni t \mapsto g_t \in L^2(\nu).
\end{align*}
\end{lemma}

\begin{proposition}[Stability of $\overline{I}$] \label{pr:Istability}
Let $P_n\in\Pi(\mu,\nu)$ satisfy $\sup_n H(P_n\,|\,\pi) < \infty$ and $\overline{I}(P_n\,|\,\pi) \to 0$. Then $\W_2(P_n,\pi)\to 0$, and also $dP_n/d\pi \to dP/d\pi$ weakly in $L^1(\pi)$. Moreover, 
\begin{align*}
\E_{P_n}[\nabla_x c(X,Y)\,|\,X] &\to \nabla\varphi(X), \quad \text{in } L^2(\mu), \\
\E_{P_n}[\nabla_y c(X,Y)\,|\,Y] &\to \nabla\psi(Y), \quad \text{in } L^2(\nu).
\end{align*}
\end{proposition}

Before proving these three ingredients, we first show how they lead quickly to the proofs of parts (3,4) of Theorems \ref{th:main}.

\begin{proof}[Proof of Theorem \ref{th:main}(3,4)]
We only prove the claimed convergence of $P_t$ to $\pi$ and of $f_t$ to $\nabla\varphi$, as the $g_t \to \nabla\psi$ convergence is treated analogously.
Let $\eta > 0$. By Lemma \ref{le:W1continuity}, we may find $\delta > 0$ such that 
\begin{equation}
\W_2(P_t,P_s) \le \eta, \ \text{ and } \ \|f_t-f_s\|_{L^2(\mu)} \le \eta, \ \text{ for all } t,s \ge 1, \ |t-s|\le\delta. \label{pf:eq:unifcont}
\end{equation}

By Lemma \ref{le:entropydynamics}, even though $H(P_t\,|\,\pi)$ may be infinite at $t=0$, it is finite for all $t > 0$, and we may thus shift time and assume it is finite at $t=0$.
Lemma \ref{le:entropydynamics} also tells us that $H(P_t\,|\,\pi)$ is decreasing in $t$, and it must therefore converge to a limit $H_* \in [0,\infty)$. Using the formula of Lemma \ref{le:entropydynamics} and Tonelli's theorem,
\begin{align*}
H(P_0\,|\,\pi) - H_* &\ge \epsilon\int_0^\infty \overline{I}(P_u\,|\,\pi)\,du = \epsilon\int_0^\delta \sum_{k=0}^\infty \overline{I}(P_{u+k\delta }\,|\,\pi)\,du.
\end{align*}
As this quantity is finite, we deduce that $\lim_{k\to\infty}\overline{I}(P_{u+k\delta }\,|\,\pi) = 0$ for a.e.\ $u \in [0,\delta]$.
By Proposition \ref{pr:Istability}, we deduce that 
\begin{align*}
\lim_{k\to\infty} \W_2(P_{u+k\delta},\pi)=0, \qquad \lim_{k\to\infty} \|f_{u+k\delta} - \nabla\varphi\|_{L^2(\mu)} =0,
\end{align*}
for a.e.\ $u \in [0,\delta]$.
Fix arbitrarily some $u \in [0,\delta]$ for which these limits hold. Then there exists $k^*\in \N$ such that $\W_2(P_{u+k\delta },\pi) \le \eta$ for all integers $k \ge k^*$. For any $t \ge 1 \vee (k^*\delta)$, we may find an integer $k \ge k^*$ such that $|t-(u+k\delta)| \le \delta$, and we deduce from \eqref{pf:eq:unifcont} that
\begin{align*}
\W_2(P_t,\pi) &\le \W_2(P_t,P_{u+k\delta}) + \W_2(P_{u+k\delta},\pi) \le 2\eta, \\
\|f_t - \nabla \varphi\|_{L^2(\mu)} &\le \|f_t - f_{u+k\delta}\|_{L^2(\mu)} + \|f_{u+k\delta} - \nabla \varphi\|_{L^2(\mu)} \le 2\eta.
\end{align*}
This shows that $\W_2(P_t,\pi)\to 0$ and $\|f_t-\nabla\varphi\|_{L^2(\mu)}\to 0$, along the entire sequence. 

Finally, we prove the claimed convergence in entropy.  This follows immediately from Lemma \ref{le:entropyconvergence} upon checking that the condition \eqref{asmp:entropyconvergence} holds with $P_\infty = \pi$. This is straightforward using the explicit form for $\log \pi$ afforded by \eqref{intro:potentials}, because $|\nabla U| \in L^2(\mu)$ and $|\nabla V| \in L^2(\nu)$ by assumption, because $\nabla c$ has linear growth, and because $\mu$ and $\nu$ are subgaussian. (See also Lemma \ref{le:momentbounds} below for even stronger integrability properties.)
\end{proof}

\subsection{Differentiability of the potentials} \label{se:differentiabilitypotentials}

Here we record an identity already noted in \eqref{intro:pi-condexp} for the derivatives of the Schr\"odinger potentials, as well one not yet mentioned for the second derivatives. These are essentially known (see, e.g., \cite[Proposition 2]{NilesWeed-Pooladian}), and they are quite straightforward if one does not worry about differentiating under the integral sign. To be careful about the latter point, we give a detailed proof in Appendix \ref{ap:differentiability}.
Recall that a function $f : \R^d \to \R$ between Euclidean spaces is said to have weak gradient $g : \R^d \to \R^d$ if $f$ and $g$ are measurable and locally integrable, and if
\begin{align*}
\int_{\R^d} f\,\mathrm{div} (h) = -\int_{\R^d} g \cdot h,
\end{align*}
for all smooth functions $h : \R^d \to \R^d$ of compact support. In the following, we write
\begin{align*}
\Cov_P(A,B\,|\,X) := \E[A \otimes B\,|\,X] - \E[A\,|\,X] \otimes \E[B\,|\,X], \quad \Cov_P(A\,|\,X) := \Cov_P(A,A\,|\,X),
\end{align*}
for two random vectors $A$ and $B$ of the same dimension.

\begin{proposition} \label{pr:derivatives-potentials}
In the sense of weak derivatives, we have the identities
\begin{equation}
\nabla\varphi(x)=\E_{\pi}[\nabla_xc(X,Y)\,|\,X=x], \qquad \nabla\psi(y)=\E_{\pi}[\nabla_yc(X,Y)\,|\,Y=y], \label{intro:pi-condexp1}
\end{equation}
as well as
\begin{equation}
\begin{split}
\nabla^2\varphi(x) &= \E_\pi[\nabla_{xx}^2 c(X,Y)\,|\,X=x] - \epsilon^{-1}\Cov_\pi(\nabla_x c(X,Y)\,|\,X=x), \\
\nabla^2\psi(y) &= \E_\pi[\nabla_{yy}^2 c(X,Y)\,|\,Y=y] - \epsilon^{-1}\Cov_\pi(\nabla_y c(X,Y)\,|\,Y=y).
\end{split} \label{intro:pi-condexp2}
\end{equation}
\end{proposition}

\subsection{Integrability and continuity lemmas}

We begin with an integrability lemma that will be used repeatedly. The choice of Frobenius norm $\|\cdot\|_{\mathrm{Frob}}$ in (3)  is immaterial, merely for the sake of concreteness.

\begin{lemma} \label{le:momentbounds}
There exists $r > 0$ such that the following estimates hold:
\begin{enumerate}
\item Integrability for couplings: $\sup_{P \in \Pi(\mu,\nu)}\E_P[ e^{r|\nabla c(X,Y)|^2}] < \infty$.
\item First derivatives of potentials:
\begin{align*}
\int_{\R^d} e^{r|\nabla \varphi(x)|^2}\,\mu(dx) < \infty, \qquad \int_{\R^d} e^{r|\nabla \psi(y)|^2}\,\nu(dy) < \infty,
\end{align*}
\item Second derivatives of potentials: 
\begin{align*}
\int_{\R^d} e^{r\|\nabla^2 \varphi(x)\|_{\mathrm{Frob}}}\,\mu(dx) < \infty, \qquad \int_{\R^d} e^{r\|\nabla^2 \psi(y)\|_{\mathrm{Frob}}}\,\nu(dy) < \infty.
\end{align*}
\item The potentials $\varphi$ and $\psi$ are $C^1$, with H\"older continuous first derivatives.
\end{enumerate}
\end{lemma}
\begin{proof}
The first claim follows from the  linear growth of $\nabla c$ (due to boundedness of $\nabla^2c$) and the assumed subgaussianity of $\mu$ and $\nu$. 
Using the identities \eqref{intro:pi-condexp1} of Proposition \ref{pr:derivatives-potentials}, the claim (2) follows easily from (1) by Jensen's inequality.
Similarly, using the identites \eqref{intro:pi-condexp2} of Proposition \ref{pr:derivatives-potentials} and the boundedness of $\nabla^2 c$, the claim (3) follows from (1) by Jensen's inequality. 
To prove (4), note that $\mu$ is continuous and thus bounded away from zero on compacts, so that (2) and (3) imply that the first and second order derivatives of $\varphi$ are in $L^p_{\mathrm{loc}}(\R^d)$ for every $1 \le p < \infty$. The claimed continuity follows from Morrey's inequality \cite[Section 5.6.2]{EvansPDE}.
\end{proof}

\begin{proof}[Proof of Continuity Lemma \ref{le:W1continuity}]
Using the obvious coupling $((X_t,Y_t),(X_s,Y_s))$ of $(P_t,P_s)$,
\begin{align*}
\W_2^2&(P_t,P_s)  \le  \E\big[|X_t-X_s|^2 + |Y_t-Y_s|^2\big] \\
	&\le  2\bigg|\E\int_s^t \Big( \E[\nabla_x c(X_u,Y_u)\,|\,X_u]  -  \nabla_x c(X_u,Y_u) - \epsilon \nabla U(X_u)\Big)\,du\bigg|^2 \!\!+ 2\sqrt{2\epsilon}\E|W_t-W_s|^2 \\
	&\quad +2\bigg|\E\int_s^t \Big( \E[\nabla_y c(X_u,Y_u)\,|\,Y_u]  -  \nabla_y c(X_u,Y_u) - \epsilon \nabla V(Y_u)\Big)\,du\bigg|^2 \!\!+ 2\sqrt{2\epsilon}\E|B_t-B_s|^2.
\end{align*}
The first claim of the lemma now follows quickly from Lemma \ref{le:momentbounds}(1). For the second claim, note that $P_t(x,y)/\mu(x)$ is the conditional density of $Y_t$ given $X_t=x$, and so
\begin{align*}
\int_{\R^d}|f_t - f_s|\,d\mu &= \int_{\R^d}\bigg|\int_{\R^d} \nabla_xc(x,y) \big(P_t(x,y)-P_s(x,y)\big)\,dy\big|\,dx.
\end{align*}
Recalling that $B_r^d$ is the centered ball of radius $r$ in $\R^d$, we deduce from Lemma \ref{le:momentbounds}(1) that
\begin{align*}
\int_{(B_r^d \times B_r^d)^c} & |\nabla_xc(x,y) |\big|P_t(x,y)-P_s(x,y)\big|\,dy dx \\
	&\le \E\big[|\nabla_xc(X_t,Y_t)|1_{\{|X_t| > r \text{ or } |Y_t| > r\}}+|\nabla_xc(X_s,Y_s)|1_{\{|X_s| > r \text{ or } |Y_s| > r\}}\big]
\end{align*}
tends to zero as $r\to\infty$, uniformly in $t,s > 0$. Hence, for any $\eta > 0$, we may find $r > 0$ such that
\begin{align*}
\int_{\R^d}|f_t - f_s|\,d\mu &\le \eta + \int_{B_r^d}\bigg|\int_{B_r^d} \nabla_xc(x,y) \big(P_t(x,y)-P_s(x,y)\big)\,dy\bigg|\,dx,
\end{align*}
for all $t,s > 0$. Now, using the H\"older estimate from Theorem \ref{th:wellposed-unbounded}, we may find $K,\beta > 0$ such that $|P_t(x,y)-P_s(x,y)| \le K|t-s|^\beta$ for all $t,s \ge 1$ with $|t-s| \le 1$. Because $\nabla_x c$ is locally bounded, the claim follows easily. The $g$ term is handled in exactly the same manner.
\end{proof}

\subsection{A key identity for log densities} \label{se:keyidentity}

The following two lemmas summarize crucial calculations that underlie Lemma \ref{le:entropydynamics} and Proposition \ref{pr:Istability}, respectively, and will appear again later.

\begin{lemma} \label{le:conditioning-identity}
Let $P \in \Pi(\mu,\nu)$ satisfy $P \ll \pi$ and $\nabla\log dP/d\pi \in L^1(P;\R^{2d})$. Then
\begin{align*}
\E_P\bigg[\nabla_x\log \frac{dP}{d\pi}(X,Y)\,\Big|\,X \bigg] &= \epsilon^{-1}\big( \E_P[\nabla_xc(X,Y)\,|\,X] - \nabla \varphi(X)\big) \\
\E_P\bigg[\nabla_y\log \frac{dP}{d\pi}(X,Y)\,\Big|\,Y \bigg] &= \epsilon^{-1}\big( \E_P[\nabla_yc(X,Y)\,|\,Y] - \nabla \psi(Y)\big).
\end{align*}
\end{lemma}
\begin{proof}
We prove only the first identity, as the second is analogous.
Let us identify $P$ with its density, $P(dx,dy)=P(x,y)dxdy$.
From Lemma \ref{le:momentbounds} and  the  formula \eqref{intro:potentials} for $\pi$, it follows easily that $|\nabla \log \pi| \in L^1(P)$.
Hence, the assumption $|\nabla\log dP/d\pi| \in L^1(P)$  implies $|\nabla\log P| \in L^1(P)$.
Since $P$ has first marginal $\mu(dx)=e^{-U(x)}dx$, we have
\begin{equation*}
\log \int_{\R^d} P(x,y)\,dy = -U(x).
\end{equation*}
Differentiate to find
\begin{equation*}
\E_{P}\big[ \nabla_x\log P(X,Y)\,|\,X=x] = \frac{\int_{\R^d} \nabla_x P(x,y)\,dy}{\int_{\R^d}  P(x,y)\,dy} = -\nabla U(x).
\end{equation*}
(See Appendix \ref{ap:differentiability}, particularly \eqref{ap:weakdiffcondition}, for details on how to rigorously justify this exchange of weak derivative and integral.)
On the other hand, the explicit form \eqref{intro:potentials} of $\pi$ shows that
\begin{align*}
\E_{P}\big[ \nabla_x\log \pi(X,Y)\,|\,X] &= \E_P\Big[ \epsilon^{-1}\big(\nabla \varphi(X) - \nabla_x c(X,Y)\big) - \nabla U(X) \,|\, X\Big] \\
	&= \epsilon^{-1}\Big(\nabla \varphi(X) - \E_P[\nabla_xc(X,Y)\,|\,X]\Big) - \nabla U(X).
\end{align*}
Take the difference between the previous two identities to complete the proof.
\end{proof}

Note that the identities \eqref{intro:pi-condexp} follow as a special case of Lemma \ref{le:conditioning-identity} by taking $P=\pi$. 
We next need a differentiated version of the formula of Lemma \ref{le:conditioning-identity}. We only state the $\nabla_x$ case, as the $\nabla_y$ case is completely analogous.

\begin{lemma} \label{le:Df}
Let $P \in \Pi(\mu,\nu)$ satisfy $\overline{I}(P\,|\,\pi) < \infty$, and set
\begin{align*}
f(x) = \E_P\bigg[ \nabla_x \log\frac{dP}{d\pi}(X,Y)\,\Big|\,X=x\bigg].
\end{align*}
Then $f$ is weak differentiable, and its weak Jacobian matrix is given by
\begin{align}
Df(x) &= \epsilon^{-1}\Cov_P\bigg(\nabla_xc(X,Y),\,\nabla_x\log \frac{dP}{d\pi}(X,Y)\,\Big|\,X=x\bigg)  -\epsilon^{-2}\Cov_P(\nabla_xc(X,Y) \,|\,X=x) \nonumber  \\
	&\qquad + \epsilon^{-1}\E_P[\nabla_{xx}^2c(X,Y)\,|\,X=x] - \epsilon^{-1}\nabla^2\varphi(x). \label{eq:Df-formula}
\end{align}
In particular, there is a constant $C$ which does not depend on the choice of $P$ such that
\begin{align}
\int \|Df\|_{\mathrm{op}}\,d\mu &\le C + C \overline{I}(P\,|\,\pi). \label{ineq:Df-bound}
\end{align}
\end{lemma}
\begin{proof}
It is an easy consequence of Lemma \ref{le:momentbounds} that $|\nabla \log\pi | \in L^2(P)$.
Thus $\overline{I}(P\,|\,\pi) < \infty$ implies $|\nabla\log dP/d\pi| \in L^2(P)$, which in turn implies $|\nabla\log P| \in L^2(P)$.
We first use Lemma \ref{le:conditioning-identity} to write
\begin{align*}
f(x) &= \epsilon^{-1} \int_{\R^d} \nabla_x c(x,y)\,\frac{P(x,y)}{\mu(x)}\,dy - \epsilon^{-1} \nabla\varphi(x).
\end{align*}
In the sense of weak derivatives, we have
\begin{align*}
\nabla_x \int_{\R^d} \nabla_x c(x,y)\,\frac{P(x,y)}{\mu(x)}\,dy &= \int_{\R^d} \nabla_{xx}^2 c(x,y)\,\frac{P(x,y)}{\mu(x)}\,dy \\
	&\quad + \int_{\R^d} \nabla_x c(x,y) \otimes \bigg(\nabla_x\log\frac{P(x,y)}{\mu(x)}\bigg)\frac{P(x,y)}{\mu(x)}\,dy
\end{align*}
Indeed, using the fact that $\mu$ is locally bounded from above and below away from zero, and recalling  also that $|\nabla_x\log P|$ and $|\nabla \log \mu|$ are in $L^2(P)$, it is easy to justify this interchange of integral and weak derivative (see Appendix \ref{ap:differentiability} for related discussion). Noting that
\begin{align*}
\int_{\R^d} \bigg(\nabla_x\log\frac{P(x,y)}{\mu(x)} \bigg)\frac{P(x,y)}{\mu(x)}\,dy = 0,
\end{align*}
we have
\begin{align*}
\int_{\R^d} \nabla_x c(x,y) \otimes \bigg(\nabla_x\log\frac{P(x,y)}{\mu(x)}\bigg)\frac{P(x,y)}{\mu(x)}\,dy = \Cov_P(\nabla_xc(X,Y),\,\nabla_x\log P(X,Y)\,|\,X=x).
\end{align*}
Put together the preceding identities to deduce
\begin{align*}
Df(x) &= \epsilon^{-1}\Cov_P(\nabla_xc(X,Y),\,\nabla_x\log P(X,Y)\,|\,X=x)  - \epsilon^{-1}\nabla^2\varphi(x) \\
	&\qquad + \epsilon^{-1}\E_P[\nabla_{xx}^2c(X,Y)\,|\,X=x].
\end{align*}
Next, notice that
\begin{align*}
\Cov_P(\nabla_xc(X,Y),\,\nabla_x\log P(X,Y)\,|\,X) &= \Cov_P\bigg(\nabla_xc(X,Y),\,\nabla_x\log \frac{dP}{d\pi}(X,Y)\,\Big|\,X\bigg) \\
	&\quad + \Cov_P(\nabla_xc(X,Y),\,\nabla_x\log \pi(X,Y)\,|\,X).
\end{align*}
Using the identity
\begin{align*}
\nabla_x\log\pi(x,y)=\epsilon^{-1}\nabla\varphi(x)-\epsilon^{-1}\nabla_xc(x,y)-\nabla U(x),
\end{align*}
we deduce
\begin{align*}
\Cov_P(\nabla_xc(X,Y),\,\nabla_x\log \pi(X,Y)\,|\,X) &= -\epsilon^{-1}\Cov_P(\nabla_xc(X,Y) \,|\,X).
\end{align*}
This yields the claimed identity. To prove the ``in particular" claim, set $R=\log dP/d\pi$, and note for any unit vector $v \in \R^d$ that
\begin{align*}
v^\top &\Cov_P\big(\nabla_xc(X,Y),\,\nabla_x R(X,Y) \,|\,X\big)v \\
	&\le \frac12 \Var_P\big(v \cdot \nabla_xc(X,Y) \,|\,X\big) +  \frac12 \Var_P\big(v \cdot \nabla_xR(X,Y) \,|\,X\big) \\
	&\le \frac12 \E_P\big[|\nabla_xc(X,Y)|^2 \,|\,X\big] + \frac12 \E_P\Big[ \big|\nabla_x R(X,Y) - \E_P[\nabla_x R(X,Y)\,|\,X]\big|^2 \,\big|\, X\Big].
\end{align*}
Hence,
\begin{align*}
\E \Big\|\Cov_P\big(\nabla_xc(X,Y),\,\nabla_x R \,|\,X\big)\Big\|_{\mathrm{op}} &\le \frac12\E_P[|\nabla_xc(X,Y)|^2] + \frac12 \overline{I}(P\,|\,\pi).
\end{align*}
The remaining terms in the identity \eqref{eq:Df-formula} for $Df(x)$ are clearly bounded in $L^1(\mu)$ uniformly in $P$, thanks to Lemma \ref{le:momentbounds} and the assumption that $\nabla^2c$ is bounded.
\end{proof}

\subsection{Entropy dynamics}

\begin{proof}[Proof of Lemma \ref{le:entropydynamics}]
It follows from Lemma \ref{le:entropycomesdown} that $H(P_t\,|\,\mu \otimes \nu) < \infty$ for a.e.\ $t > 0$.
Since $P_t \in \Pi(\mu,\nu)$, the precise form of $\pi$ implies the identity 
\begin{align*}
H(P_t\,|\,\pi) = H(P_t\,|\,\mu \otimes \nu) - H(\pi\,|\,\mu \otimes \nu) + \epsilon^{-1}\int c\,dP_t - \epsilon^{-1}\int c\,d\pi.
\end{align*}
Since $\nabla^2 c$ is bounded  and $(\mu,\nu)$ have finite second moments, we have $\int c\,dQ < \infty$ for every $Q \in \Pi(\mu,\nu)$.
Hence, $H(P_t\,|\,\mu \otimes \nu) < \infty$ implies  $H(P_t\,|\,\pi) < \infty$. We can upgrade from ``a.e.\ $t > 0$" to ``every $t > 0$" once the inequality is proven, using lower semicontinuity of relative entropy and weak continuity of $t \mapsto P_t$.

Let us first give the formal argument for the main inequality, ignoring questions of smoothness: Suppose $(Q^i_t)_{t \ge 0}$ for $i=1,2$ are weakly continuous flows of probability measures on $\R^{2d}$ which solve the Fokker-Planck equation
\begin{align*}
\partial_tQ^i = -\mathrm{div}(Q^ib^i) + \epsilon\Delta Q^i = -\mathrm{div}(Q^i(b^i - \epsilon\nabla \log Q^i)),
\end{align*}
for a time-dependent vector field $b^i$. Then
\begin{align}
\frac{d}{dt}H(Q^1_t\,|\,Q^2_t) = \int (b^1_t - b^2_t) \cdot \nabla \log \frac{dQ^1_t}{dQ^2_t}\,dQ^1_t - \epsilon \int \bigg|\nabla \log\frac{dQ^1_t}{dQ^2_t}\bigg|^2\,dQ^1_t. \label{pf:entropydynamics-formal1}
\end{align}
The computation goes as follows, omitting time indices, with all integrals taken with respect to to Lebesgue measure:
\begin{align*}
\frac{d}{dt}H(Q^1\,|\,Q^2) &= \frac{d}{dt}\int \log\frac{Q^1}{Q^2}\,Q^1 \\
	&= \int  \Big(\partial_t Q^1 - \partial_tQ^2\frac{Q^1}{Q^2} + \log\frac{Q^1}{Q^2}\partial_tQ^1 \Big) \\
	&= 0 + \int \Big(\frac{Q^1}{Q^2}\mathrm{div}(Q^2(b^2 - \epsilon\nabla \log Q^2)) - \log\frac{Q^1}{Q^2}\mathrm{div}(Q^1(b^1 - \epsilon\nabla \log Q^1))\Big) \\
	&= \int \Big(-\nabla\Big(\frac{Q^1}{Q^2}\Big) \cdot (b^2 - \epsilon\nabla \log Q^2)Q^2 + \nabla \Big(\log\frac{Q^1}{Q^2}\Big) \cdot (b^1 - \epsilon\nabla \log Q^1)Q^1\Big) \\
	&= \int \nabla\Big(\log\frac{Q^1}{Q^2}\Big) \cdot \Big(\epsilon\nabla \log Q^2 -b^2 + b^1 - \epsilon\nabla \log Q^1\Big)Q^1.
\end{align*}
This is exactly \eqref{pf:entropydynamics-formal1}.

We next show how \eqref{pf:entropydynamics-formal1} applies in our context.
Let $f_t(x)=\E[\nabla_x c(X_t,Y_t)\,|\,X_t=x]$ and $g_t(y)=\E[\nabla_y c(X_t,Y_t)\,|\,Y_t=y]$. Note that $\pi$ is the law of a stationary solution of the SDE
\begin{align}
\begin{split}
dX_t &= \bigg( \nabla\varphi(X_t)  - \nabla_x c(X_t,Y_t) - \epsilon \nabla U(X_t)\bigg)dt + \sqrt{2\epsilon }dW_t, \\
dY_t &= \bigg(\nabla\psi(Y_t)  - \nabla_y c(X_t,Y_t)  - \epsilon \nabla V(Y_t)\bigg)dt + \sqrt{2\epsilon }dB_t.
\end{split} \label{def:piSDE}
\end{align}
This is exactly the same drift as in the original SDE \eqref{def:mainSDE}, except with $\nabla\varphi$ and $\nabla\psi$ replacing the conditional expectations terms, which we abbreviated as $f_t$ and $g_t$ in \eqref{def:ft,gt}.
Hence, the entropy identity \eqref{pf:entropydynamics-formal1} specializes to
\begin{align*}
\frac{d}{dt}H(P_t\,|\,\pi) = \E_{P}\bigg[&\big(f_t(X_t) - \nabla \varphi(X_t)\big) \cdot \nabla_x\log\frac{dP_t}{d\pi}(X_t,Y_t) \\
	& + \big(g_t(Y_t) - \nabla \psi(Y_t)\big) \cdot \nabla_y\log\frac{dP_t}{d\pi}(X_t,Y_t) \\
	& - \epsilon \Big|\nabla_x\log\frac{dP_t}{d\pi}(X_t,Y_t)\Big|^2 - \epsilon \Big|\nabla_y\log\frac{dP_t}{d\pi}(X_t,Y_t)\Big|^2 \bigg].
\end{align*}
Since $P_t \in \Pi(\mu,\nu)$, we have from Lemma \ref{le:conditioning-identity} that
\begin{align*}
\E_P\bigg[\nabla_x\log \frac{dP_t}{d\pi}(X_t,Y_t)\,\Big|\,X_t \bigg] &= \epsilon^{-1}\big( f_t(X_t) - \nabla \varphi(X_t)\big), \\
\E_P\bigg[\nabla_y\log \frac{dP_t}{d\pi}(X_t,Y_t)\,\Big|\,Y_t \bigg] &= \epsilon^{-1}\big( g_t(Y_t) - \nabla \psi(Y_t)\big).
\end{align*}
We deduce that
\begin{align*}
\frac{d}{dt}H(P_t\,|\,\pi) = \epsilon \E_{P}\bigg[& \nabla_x\log\frac{dP_t}{d\pi}(X_t,Y_t) \cdot \E_P\bigg[\nabla_x\log \frac{dP_t}{d\pi}(X_t,Y_t)\,\Big|\,X_t \bigg] \\
	& + \nabla_y\log\frac{dP_t}{d\pi}(X_t,Y_t) \cdot \E_P\bigg[\nabla_y\log \frac{dP_t}{d\pi}(X_t,Y_t)\,\Big|\,Y_t \bigg] \\
	& - \Big|\nabla_x\log\frac{dP_t}{d\pi}(X_t,Y_t)\Big|^2 - \Big|\nabla_y\log\frac{dP_t}{d\pi}(X_t,Y_t)\Big|^2 \bigg].
\end{align*}
Conditioning shows that the right-hand side is precisely $-\epsilon \bar{I}(P_t\,|\,\pi)$, 
which implies the claimed inequality \eqref{def:entropydecay} upon integrating.

To make the formula \eqref{pf:entropydynamics-formal1} rigorous, one must of course argue that the integrals are all well-defined, and that the densities $Q^1$ and $Q^2$ are positive and sufficiently smooth in order to justify differentiating under the integral sign. We omit the details of this tedious but well understood task, which can be accomplished by a mollification of $(Q^1,Q^2)$ and a smooth approximation of the function $x \mapsto x\log x$. The differential formula \eqref{pf:entropydynamics-formal1} is not the easiest to work with during these approximations; more stable is the integrated form relaxed to an inequality,
\begin{align}
H(Q^1_t\,|\,Q^2_t) + \epsilon\int_s^t \int \bigg(\bigg|\nabla \log\frac{dQ^1_r}{dQ^2_r}\bigg|^2 - (b^1_r - b^2_r) \cdot \nabla \log \frac{dQ^1_r}{dQ^2_r}\bigg)\,dQ^1_r\,dr  \le H(Q^1_s\,|\,Q^2_s). \label{pf:entropydynamics-general}
\end{align}
We refer to \cite[Proof of Lemma 2.4]{BogRocSha} for the details of this kind of approximation.
This approximation argument also shows, as a byproduct, that the Fisher information $\int |\nabla \log(dQ^1_r/dQ^2_r)|^2dQ^1_r$ is well defined and finite for a.e.\ $r$, which in our context justifies the application of Lemma \ref{le:conditioning-identity} above. Another route to justifying this kind of formula uses the subdifferential calculus on Wasserstein space, such as  \cite[Theorem 10.4.6]{AGSbook}.
\end{proof}

\subsection{Key properties of $\overline{I}$}

This section proves the crucial Proposition \ref{pr:Istability}. We first treat a special case as a warm-up, by showing that $\overline{I}(P\,|\,\pi)=0$ implies $P=\pi$.
In this section we favor an alternative expression for $\bar{I}$, which is somewhat more enlightening:
\begin{align*}
\overline{I}(P\,|\,\pi) &= \E_{P}\bigg[  \Var_P\bigg(\nabla_x \log\frac{dP}{d\pi} \,\Big|\,X\bigg) + \Var_P\bigg(\nabla_y \log\frac{dP}{d\pi} \,\Big|\,Y\bigg) \bigg].
\end{align*}
Here the variance of a random vector is defined as the sum of the variances of the coordinates.

\begin{proposition}  \label{pr:Imin=0}
Let $P\in\Pi(\mu,\nu)$. Then $\overline{I}(P\,|\,\pi)=0$ if and only if $P=\pi$.
\end{proposition}
\begin{proof}
The ``if" implication is obvious.
Suppose $\overline{I}(P\,|\,\pi)=0$. Then $P \ll \pi$ is absolutely continuous with respect to Lebesgue measure. Let $R=\log dP/d\pi$. Then $\nabla R$ exists in $L^2(\pi)$, and 
\begin{equation*}
\Var_{P}\big(\nabla_x R(X,Y) \,\big|\,X\big) = \Var_{P}\big(\nabla_y R(X,Y) \,\big|\,Y\big) =0, \quad a.s.
\end{equation*}
We deduce that there exist Borel functions $g,h : \R^d \to \R^d$ such that 
\begin{equation}
\nabla_x R(x,y) = g(x), \qquad \nabla_y R(x,y) = h(y), \quad a.e. \ (x,y) \in \R^{2d}. \label{pf:I=01}
\end{equation}
We claim that such a function must be of the form $R(x,y) = F(x) + G(y)$ for some Borel functions $F,G : \R^d \to \R$. Once this is established, it follows immediately from the uniqueness of Schr\"odinger potentials \cite[Theorem 2.1(b)]{nutz2021introduction} that $P=\pi$.

To prove this claim, note by Fubini that \eqref{pf:I=01} implies that the sets $\{x \in \R^d : \nabla_x R(x,y)=g(x), \ a.e. \ y\}$ and $\{y \in \R^d :  \nabla_y R(x,y)= h(y), \ a.e. \ x\}$ have full measure. We may thus choose $(x_0,y_0)$ such that $g(x)=\nabla_x R(x,y_0)$ for a.e.\ $x$ and $h(y)=\nabla_y R(x_0,y)$ for a.e.\ $y$. Let $F(x) := R(x,y_0)$ and $G(y):=R(x_0,y)$.  Letting $\bar{R}(x,y) := F(x)+G(y)$, we see that the weak gradient of $R-\bar{R}$ vanishes (a.e.). This implies $R-\bar{R}$ is a.e.\ constant, which proves the claim.
\end{proof}

\begin{proof}[Proof of Proposition \ref{pr:Istability}]
Because $P_n \in \Pi(\mu,\nu)$ and $\mu$ and $\nu$ have finite moments of every order, the claimed $\W_2$-convergence will follow if we show that $P_n \to \pi$ weakly.
The assumption $\sup_n H(P_n\,|\,\pi) < \infty$ implies that the sequence $(P_n)$ is tight, so it suffices to show that any subsequential weak limit point must equal $\pi$. 
We may thus assume henceforth that $P_n \to P$ weakly for some $P \in \Pi(\mu,\nu)$. 
The assumption $\sup_nH(P_n\,|\,\pi) < \infty$ then  implies also the ($L^1(\pi)$,weak) convergence of $dP_n/d\pi$ to $dP/d\pi$.
It remains to show that necessarily $P=\pi$. 
Define $f_n \in L^2(\mu;\R^d)$ by
\begin{equation*}
f_n(x) := \E_{P_n}\bigg[\nabla_x \log\frac{dP_n}{d\pi}(X,Y)\,\Big|\,X=x\bigg].
\end{equation*}
\textbf{Step 1.} We claim first that there exist $f \in \cap_{p \ge 1}L^p(\mu;\R^d)$ such that, along a subsequence, $f_n \to f$ in $L^p(\mu;\R^d)$ norm for each $p \ge 1$. We first use Lemma \ref{le:conditioning-identity} to write
\begin{align}
f_n(x) &= \epsilon^{-1} \E_{P_n}[\nabla_x c(X,Y)\,|\,X=x] - \epsilon^{-1} \nabla\varphi(x). \label{pf:fn-identity}
\end{align}
Note by Lemma \ref{le:momentbounds} that
\begin{equation}
\sup_n\int_{\R^d}|f_n|^p\,d\mu < \infty, \quad \forall p \ge 1. \label{pf:fn-momentbound}
\end{equation}
Next, use the inequality \eqref{ineq:Df-bound} of Lemma \ref{le:Df},
\begin{align*}
\sup_n \int_{\R^d}\|Df_n\|_{\mathrm{op}}\,d\mu  \le C + C\overline{I}(P_n\,|\,\pi).
\end{align*}
The right-hand side is uniformly bounded because $\overline{I}(P_n\,|\,\pi) \to 0$ by assumption. Since the density of $\mu$ is continuous and strictly positive, it is bounded away from zero on any ball $B \subset \R^d$, and we deduce that $\{Df_n|_B\}$ is a bounded sequence in $L^1(B,dx;\R^{d \times d})$.
Similarly, $\{f_n\}$ is bounded in $L^1(\mu;\R^d)$ by \eqref{pf:fn-momentbound}, and so $\{f_n|_B\}$ is also bounded in $L^1(B,dx;\R^d)$.
In particular, for each ball $B$, $\{f_n|_B\}$ is bounded in the Sobolev space $W^{1,1}(B;\R^d)$, which embeds compactly in $L^1(B,dx;\R^d)$ by the Rellich-Kondrachov theorem \cite[Section 5.7, Theorem 1]{EvansPDE}.  Hence, with a diagonal argument, we may find $f \in L^1_{\mathrm{loc}}(\R^d;\R^d)$ and a subsequence (relabeled) such that $\|(f_n-f)1_B\|_{L^1(\R^d;\R^d)} \to 0$ for every ball $B$. The convergence in $\mu$-measure of $f_n$ to $f$ now follows from the local boundedness of the density $\mu$.
The $p$-moment bound of \eqref{pf:fn-momentbound} lets us upgrade from convergence in measure to $L^p(\mu;\R^d)$-convergence $f_n \to f$; in particular, $f \in L^p(\mu;\R^d)$.

{\ }

\noindent\textbf{Step 2.} We claim next that it necessarily holds $\mu$-a.s.\ that
\begin{equation}
f(x) = \epsilon^{-1} \E_{P}[\nabla_x c(X,Y)\,|\,X=x] - \epsilon^{-1} \nabla\varphi(x). \label{pf:step2-f}
\end{equation}
To see this, let $h \in C^\infty_c(\R^{d};\R^d)$, and note that the identity \eqref{pf:fn-identity} implies
\begin{equation}
0 = \E_{P_n}\big[h(X) \cdot (\nabla\varphi(X) + \epsilon f_n(X) - \nabla_xc(X,Y))\big]. \label{pf:fn-limit1}
\end{equation}
Since $h$ and $\nabla c$ are continuous and $h \cdot \nabla c$ is bounded, it follows from the weak convergence $P_n \to P$ that
\begin{equation*}
\E_{P_n}[h(X) \cdot \nabla_xc(X,Y)] \to \E_{P}[h(X)  \cdot  \nabla_xc(X,Y)].
\end{equation*}
Since $P_n$ has first marginal $\mu$ and $| \nabla \varphi| \in L^1(\mu)$ by Lemma \ref{le:momentbounds}, we deduce from the $L^1(\mu)$-convergence of $f_n\to f$ and boundedness of $h$  that
\begin{equation*}
\E_{P_n}\big[h(X) \cdot (\nabla\varphi(X) + \epsilon f_n(X))\big] = \int_{\R^d} h \cdot (\nabla\varphi + \epsilon f_n) \,d\mu \to \int_{\R^d} h \cdot (\nabla\varphi + \epsilon f) \,d\mu.
\end{equation*}
Thus, taking limits in \eqref{pf:fn-limit1} and using $P \in \Pi(\mu,\nu)$, we get
\begin{equation*}
0 = \E_{P}\big[h(X) \cdot (\nabla\varphi(X) + \epsilon f(X) - \nabla_xc(X,Y))\big].
\end{equation*}
As $h \in C^\infty_c(\R^d;\R^d)$ was arbitrary, the claim \eqref{pf:step2-f} follows.

{\ }

\noindent\textbf{Step 3.} We next prove the identity of weak derivatives,
\begin{align*}
\nabla_x \log \frac{dP}{d\pi}(x,y) = f(x).
\end{align*}
Equivalently, we will show that $\nabla_x (dP/d\pi)(x,y) = (dP/d\pi)(x,y) f(x)$. To do this, we must show that for any $h$ belonging to the space $C^1_c(\R^{2d};\R^d)$ of continuously differentiable functions $\R^{2d} \to \R^d$ of compact support, we have
\begin{align*}
-\int_{\R^{2d}} \mathrm{div}_x\, h(x,y)\, \frac{dP}{d\pi}(x,y)\,dxdy &= \int_{\R^{2d}} h(x,y) \cdot f(x)\, \frac{dP}{d\pi}(x,y)\,dxdy.
\end{align*}
Note that $\pi$ is continuously differentiable by Lemma \ref{le:momentbounds}(4) and bounded away from zero on compact sets, so the map $C^1_c(\R^{2d};\R^d) \ni h \mapsto \pi h \in C^1_c(\R^{2d};\R^d)$ is bijective. Hence, replacing $h$ by $\pi h$ in the above equation,  it suffices to show that
\begin{align}
-\E_P\Big[\mathrm{div}_x\, h(X,Y)+ \nabla_x\log \pi(X,Y) \cdot h(X,Y)\Big] &= \E_P[h(X,Y) \cdot f(X)]. \label{pf:deriv-identity1}
\end{align}
To prove \eqref{pf:deriv-identity1}, we define $\xi_n(x,y) := \nabla_x \log(dP_n/d\pi)(x,y)-f_n(x)$, and note that  the same logic which led to \eqref{pf:deriv-identity1} yields
\begin{align}
-\E_{P_n}\Big[\mathrm{div}_x\, h(X,Y)+ \nabla_x\log \pi(X,Y) \cdot h(X,Y)\Big] &= \E_{P_n}\big[h(X,Y) \cdot \big(f_n(X)+\xi_n(X)\big)\big]. \label{pf:deriv-identity2}
\end{align}
It remains to let $n \to \infty$ on both sides of \eqref{pf:deriv-identity2} to prove \eqref{pf:deriv-identity1}.
It is immediate that the left-hand side of \eqref{pf:deriv-identity2} converges to that of \eqref{pf:deriv-identity1}, since $\mathrm{div}_x\,h$ and $\nabla_x\log \pi \cdot h$ are continuous with compact support. We take limits on the right-hand side in several steps. Note first that $\E_{P_n}|\xi_n(X,Y)|^2 = \overline{I}(P_n\,|\,\pi) \to 0$ by assumption, so it remains to show that $\E_{P_n} [h(X,Y) \cdot f_n(X)] \to \E_{P} [h(X,Y) \cdot f(X)]$. Because $h$ is bounded and $f_n \to f$ in $L^1(\mu)$, we have
\begin{align*}
|\E_{P_n}[h(X,Y) \cdot (f_n(X) - f(X)) ]| \le \|h\|_\infty \int_{\R^d} |f_n-f|\,d\mu \to 0.
\end{align*}
Hence, to prove \eqref{pf:deriv-identity1} it remains to show that
\begin{align}
\int_{\R^{2d}} h(x,y) \cdot f(x) \,(P_n-P)(dx,dy) \to 0. \label{pf:deriv-identity3}
\end{align} 
Hence, for each $a > 0$, 
\begin{align*}
\int_{\R^{2d}} h(x,y) \cdot f(x)1_{\{|f(x)| \le a\}} \,(P_n-P)(dx,dy) \to 0.
\end{align*} 
We know from Step 1 that $f \in L^1(\mu)$, and since $P_n$ and $P$ have $\mu$ as the first marginal, we deduce that
\begin{align*}
\sup_n \bigg|\int_{\R^{2d}} h(x,y) \cdot f(x)1_{\{|f(x)| > a\}} \,(P_n-P)(dx,dy)\bigg| &\le 2\|h\|_\infty \int_{\R^d} |f| 1_{\{|f| > a\}}\,d\mu \to 0
\end{align*}
as $a \to \infty$.  The claim \eqref{pf:deriv-identity3} follows, completing the proof of Step 3.

{\ }

\noindent\textbf{Conclusion of the proof.} By Step 3, and the analogous argument for the $\nabla_y$ term, we have shown that there exist $f \in L^1(\mu)$ and $g \in L^1(\nu)$ such that
\begin{align*}
\nabla_x \log\frac{dP}{d\pi}(x,y) = f(x), \ \mu-a.s., \qquad  \nabla_y \log\frac{dP}{d\pi}(x,y) = g(y), \ \nu-a.s.
\end{align*}
It follows, as in the proof of Proposition \ref{pr:Imin=0}, that there exist Borel functions $F,G$ such that $\log (dP/d\pi)(x,y) = F(x)+G(y)$, and thus $P=\pi$ by uniqueness of Schr\"odinger potentials.
\end{proof}

\section{Exponential convergence and a new logarithmic Sobolev inequality} \label{se:expconvergence}

In this section we first prove Theorem \ref{thm:suff_cond_new_LSI}. Then, we proceed to the proof of Corollary \ref{intro:corollary_log_conc} and Theorem \ref{thm:conv_prof_rigorous}, that is a more general version of Theorem \ref{thm:conv_prof_informal} above. Both results provide lower bounds on the LSI constant for the conditional distributions $\pi^x$ and $\pi^y$, with the assumptions of Theorem \ref{thm:conv_prof_rigorous} being weaker than those of Corollary \ref{intro:corollary_log_conc}. It can be read directly from \eqref{intro:potentials} that the conditional distribution $\pi^y$ can be written in the form of a Gibbs measure as 
\begin{equation*}
\pi^y(dx) = \frac{1}{Z^y} \exp\Big(\frac{1}{\epsilon}(\varphi(x)-c(x,y)) -\nabla U(x)\Big)\,dx,
\end{equation*}
where $Z^y$ is a normalization constant. Therefore, in view of showing that $\pi^y$ satisfies LSI, one is naturally led to try to obtain some concavity properties of the Schr\"odinger potential $\varphi$. When marginals are strictly log-concave one can indeed show, using the results of \cite{chewi2022entropic}, that $\pi^y(dx)$ is a strongly log-concave probability measure. Then, lower bounds on the LSI constant follow directly from the Bakry-\'Emery criterion. In the more delicate situation when only bounds on the integrated convexity and concavity profile of $U$ are available, which corresponds to the setting of Theorem \ref{thm:conv_prof_rigorous} the proof consists of two independent steps. The first step is to leverage the results of \cite{conforti2022weak,conforti2023convergence} to establish integrated convexity bounds on $-\varphi$ that are strong enough to imply that $\pi^y$ has an asymptotically positive convexity profile uniformly in $y$. The second step consists in establishing LSI for a probability measure with an asymptotically positive convexity profile, which is a question of independent interest. This is achieved at Theorem \ref{th:newLSI-kappa} below, thus completing the proof of Theorem \ref{thm:conv_prof_rigorous}.

\subsection{Proof of Theorem \ref{thm:suff_cond_new_LSI}} \label{se:newLSIproof-suff}
\begin{proof}
    We deduce from the basic properties of conditional expectation that the projected LSI \eqref{intro:newLSI} is equivalent to 
    \begin{equation}\label{eq:suff_cond_LSI_1}
        \begin{split}
        H(P\,|\,\pi)&+ \frac{\epsilon}{r}\E_{P}[|\E_{P}[\nabla_x R(X,Y)|X]|^2]\\
        &+\frac{\epsilon}{r}\E_{P}[|\E_{P}[\nabla_y R(X,Y)|Y]|^2] \leq \frac{\epsilon}{r}I(P\,|\,\pi), \qquad \forall P \in \Pi(\mu,\nu).
        \end{split}
    \end{equation}
Fix $P \in \Pi(\mu,\nu)$, and assume that $I(P\,|\,\pi) < \infty$ as otherwise there is nothing to prove.
Using Lemma \ref{le:conditioning-identity} in the case $c(x,y)=|x-y|^2/2$, along with the identites \eqref{intro:pi-condexp}, we have
     \begin{align*}
     \E_{P}[\nabla_xR(X,Y) |X] &= \epsilon^{-1}(\E_{P}[Y|X]-\E_{\pi}[Y|X] ), \\
        \E_{P}[\nabla_yR(X,Y) |Y] &= \epsilon^{-1}(\E_{P}[X|Y]-\E_{\pi}[X|Y] ).
     \end{align*}
     Thus, \eqref{eq:suff_cond_LSI_1} is equivalent to
        \begin{equation}\label{eq:suff_cond_LSI_2}
        \begin{split}
          H(P\,|\,\pi)&+ \frac{1}{  \epsilon r}\E_{P}[|\E_{P}[Y|X]-\E_{\pi}[Y|X]|^2]\\
          &+\frac{1}{\epsilon r}\E_{P}[|\E_{P}[X|Y]-\E_{\pi}[X|Y]|^2] \leq \frac{\epsilon}{r}I(P\,|\,\pi), \qquad \forall P\in\Pi(\mu,\nu).
        \end{split}
        \end{equation}
        We will bound separately the left and right hand side in the above inequality. Let $x\mapsto P^{x}(dy)$ denote a version of the conditional distribution of $Y$ given $X$ under $P$. Using $\pi,P\in\Pi(\mu,\nu)$ and the assumed LSI for the conditional measures $\pi^x$, we have
    \begin{equation*}
        \begin{split}
            \E_P[|\nabla_y R(X,Y)|^2]&= \int\Big( \int \Big|\nabla_y        \log \frac{dP^x}{d\pi^x}(y)\Big|^2 P^x(dy)\Big) \mu(dx)\\
            &\geq\kappa^{Y|X}\int H(P^x|\pi^x) \,\mu(dx) =  \kappa^{Y|X}H(P\,|\,\pi).
         \end{split}
    \end{equation*}
    Arguing in the same way,
    \begin{equation*}
         \E_P[|\nabla_x R(X,Y)|^2] \geq   \kappa^{X|Y}H(P\,|\,\pi).
    \end{equation*}
    Then, observing that $r>0$ under the current assumptions, we have 
    \begin{equation}\label{eq:suff_cond_LSI_3}
        \frac{\epsilon}{r}I(P\,|\,\pi) \geq \frac{\epsilon}{r}(\kappa^{Y|X}+\kappa^{X|Y})H(P\,|\,\pi).
    \end{equation}
To handle the terms on the left hand side of \eqref{eq:suff_cond_LSI_2} containing conditional expectations,  we first recall that since  $\pi^x$ satisfies $\lsi{\kappa^{Y|X}}$ uniformly on $x$, then the Otto-Villani Theorem \cite[Thm. 1]{otto2000generalization} gives that the following Talagrand inequality holds uniformly in $x \in \R^d$:
\begin{equation*}
\mathcal{W}_2^2(q,\pi^x) \leq \frac{4}{\kappa^{Y|X}}H(q|\pi^x), \quad \forall q\in\mathcal{P}(\R^d).
\end{equation*}
Then  the following inequalities hold $P$-a.s.:
\begin{equation}\label{eq:talagrand_conditional}
    |\E_{P}[Y|X]-\E_{\pi}[Y|X]|^2  \leq \mathcal{W}^2_2(P^X,\pi^X) \leq \frac{4}{\kappa^{Y|X}}H(P^X|\pi^X)
\end{equation}
Therefore, recalling that $r>0$, we arrive at
\begin{equation}
\begin{split}
  H(P\,|\,\pi)&+ \frac{1}{  \epsilon r}\E_{P}[|\E_{P}[Y|X]-\E_{\pi}[Y|X]|^2] +\frac{1}{\epsilon r}\E_{P}[|\E_{P}[X|Y]-\E_{\pi}[X|Y]|^2]\\
          &\leq \Big(1+\frac{4}{\epsilon r}\Big(\frac{1}{\kappa^{X|Y}}+\frac{1}{\kappa^{Y|X}}\Big)\Big)   H(P\,|\,\pi). 
 \end{split}\label{eq:suff_cond_LSI_4}
\end{equation}
The definition of $r$ in \eqref{intro:newLSI_r} is precisely what makes the right-hand sides of  \eqref{eq:suff_cond_LSI_3} and \eqref{eq:suff_cond_LSI_4} equal. This yields \eqref{eq:suff_cond_LSI_2}, thus proving at once that $\pi$ satisfies the projected LSI \eqref{intro:newLSI} with constant $r$ given by \eqref{intro:newLSI_r} and the exponential decay of the entropy \eqref{eq:exp_decay_entropy}.  Let us now proceed to prove \eqref{eq:exp_conv_drift}. We only prove the first inequality, the proof of the second one being conceptually identical. To do, we simply observe that since $P_t\in\Pi(\mu,\nu)$, combining \eqref{eq:talagrand_conditional} with \eqref{intro:pi-condexp} we obtain that $\mu-$a.e.
\begin{equation*}
 |\mathbb{E}[X_t-Y_t|X_t=x] - \nabla\varphi(x)|^2  \leq\frac{4}{\kappa^{Y|X}} H(P^{x}_0|\pi^x)
\end{equation*}
Integrating on both sides with respect to $\mu$ and using \eqref{eq:exp_decay_entropy}, the conclusion follows.
\end{proof}

\subsection{Proof of Corollary \ref{intro:corollary_log_conc}} \label{se:logconcave-proofs}
In order to prove Corollary \ref{intro:corollary_log_conc} we make use of  a result of \cite{chewi2022entropic} which derived lower bounds on the Hessian of Schr\"odinger potentials in the case of strongly log-concave marginals.

\begin{theorem}[Theorem 5 in \cite{chewi2022entropic}]\label{thm:chewi_pool}
Let $c(x,y)=|x-y|^2/2$. Assume that $U$ and $V$ are twice continuously differentiable and such that for all $x,y\in\R^d$ 
\begin{equation*}
\alpha_U \mathrm{I} \preceq    \nabla^2 U(x) \preceq \beta_U \mathrm{I}, \quad \alpha_V \mathrm{I} \preceq   \nabla^2 V(y) \preceq \beta_{V} \mathrm{I}
\end{equation*}
in semidefinite order. Then, if we define
\begin{equation}\label{eq:tilde_pot}
    \tilde{\varphi} = |\cdot|^2/2-\varphi(\cdot), \quad \tilde{\psi}=|\cdot|^2/2-\psi(\cdot)
\end{equation}
it holds for all $x,y \in \R^d$, in semidefinite order, that
\begin{equation*}
    \nabla^2 \tilde{\varphi}(x) \succeq \frac{1}{2} \Big(\sqrt{4\alpha_U/\beta_V+\epsilon^2\alpha^2_U}-\epsilon\alpha_U \Big)\mathrm{I}\quad \text{and} \quad
    \nabla^2 \tilde{\psi} (y)\succeq \frac{1}{2} \Big(\sqrt{4\alpha_V/\beta_U+\epsilon^2\alpha^2_V}-\epsilon\alpha_V \Big)\mathrm{I}.
\end{equation*}
\end{theorem}

\begin{proof}[Proof of Corollary \ref{intro:corollary_log_conc}]
By the formula \eqref{intro:potentials}, a version of the regular conditional distribution $y\mapsto \pi^{y}$ is given by
    \begin{align*}
        \pi^y(dx)&= \frac{1}{Z^y}\exp\Big(\frac{\varphi(x)-|x|^2/2+x\cdot y}{\epsilon}-\nabla U(x)\Big)dx \\
        	&= \frac{1}{Z^y}\exp\Big(\frac{x\cdot y - \tilde{\varphi}(x)}{\epsilon}-\nabla U(x)\Big)dx,
    \end{align*}
    where $Z^y$ is a normalization constant.
    By Theorem \ref{thm:chewi_pool}, we have
    \begin{equation*}
    -\nabla_x^2\log \pi^y(x) = \epsilon^{-1}\nabla^2\tilde\varphi(x) + \nabla^2 U(x) \ge  (\kappa^{X|Y}/2)\,\mathrm{I}
    \end{equation*}
    in semidefinite order,
where $\kappa^{X|Y}$ was defined in \eqref{def:LSI-logco-kappa}.
    We deduce from the Bakry-\'Emery sufficient condition for LSI that $\pi^{y}$ satisfies $\lsi{\kappa^{X|Y}}$ for each $y$.
    Arguing in the same way, we obtain that $\pi^{x}$ satisfies $\lsi{\kappa^{Y|X}}$ uniformly in $x$.
    It remains to use the assumption \eqref{eq:log_conc_suff_cond}  to prove that the assumption $(\kappa^{X|Y}\kappa^{Y|X})^{1/2}\epsilon > 2$ of Theorem \ref{thm:suff_cond_new_LSI} is satisfied.
    It is equivalent to show that there exists $\theta>0$ such that 
    \begin{equation*}
        \epsilon\kappa^{X|Y} > 2\theta \quad \text{and} \quad \epsilon\kappa^{Y|X}> 2/\theta.
    \end{equation*}
    Noting that $x=\epsilon\kappa^{X|Y}/2$ (resp.\ $y=\epsilon\kappa^{Y|X}/2$) is the unique (positive) root of the increasing function $\R_+ \ni x \mapsto \alpha_U^{-1}x -\beta_V^{-1}x^{-1}-\epsilon$ (resp.\ $y \mapsto \alpha_V^{-1}y -\beta_U^{-1}y^{-1}=\epsilon$), it is equivalent to find $\theta > 0$ such that
    \begin{equation*}
        \epsilon > \theta\alpha^{-1}_U - \theta^{-1}\beta^{-1}_V  \quad \text{and} \quad 
        \epsilon > \theta^{-1}\alpha^{-1}_V - \theta\beta^{-1}_U.
    \end{equation*}
    The existence of such a $\theta > 0$ is equivalent to the inequality
    \begin{equation*}
        \epsilon>\inf_{\theta>0} \max\{\theta\alpha^{-1}_U - \theta^{-1}\beta^{-1}_V, \theta^{-1}\alpha^{-1}_V - \theta\beta^{-1}_U\}.
    \end{equation*}
    The infimum is attained when the two arguments of $\max\{\cdot,\cdot\}$ are equal, or 
    \begin{equation*}
        \theta^{\star}= \Big( \frac{\alpha^{-1}_V +\beta^{-1}_V}{\alpha^{-1}_U +\beta^{-1}_U}\Big)^{1/2},
    \end{equation*}
    and the value of the infimum is 
    \begin{equation*}
        \frac{\alpha^{-1}_U\alpha^{-1}_V-\beta^{-1}_U\beta^{-1}_V}{(\alpha^{-1}_U+\beta^{-1}_U)^{1/2}(\alpha^{-1}_V+\beta^{-1}_V)^{1/2}}.
    \end{equation*}
    This shows that \eqref{eq:log_conc_suff_cond} is equivalent to $(\kappa^{X|Y}\kappa^{Y|X})^{1/2}\epsilon > 2$, completing the proof.
    \end{proof}

\subsection{Asymptotically log-concave marginals}\label{sec:proof_conc_prof}
The main result of this section is Theorem \ref{thm:conv_prof_rigorous} below, which is a more detailed version of Theorem \ref{thm:conv_prof_informal}.

\subsubsection{The convexity profile of Schr\"odinger potentials} 
In this section we illustrate how to translate lower bounds on the integrated convexity profile of $U$ and $V$ into lower bounds on the integrated convexity profile of a modified version $(\bar{\varphi},\bar{\psi})$ of the potentials $(\varphi,\psi)$, namely 
\begin{equation}\label{def:bar_potentials}
    \bar\varphi = -\epsilon^{-1}\varphi+U, \quad \bar\psi = -\epsilon^{-1}\psi+V.
\end{equation}
From now on, we shall refer to $\bar\varphi,\bar\psi$ as to Schr\"odinger potentials. Integrated convexity estimates for Schr\"odinger potentials have been recently established in \cite{conforti2022weak,conforti2023convergence}; they constitute the backbone of the proof of Theorem \ref{thm:conv_prof_rigorous} and we report them at Theorem \ref{gio2023:thm}, following their notation along the way. To prepare for this result, we begin by introducing two sets of functions, $\mathcal{G}$ and $\hat{\mathcal{G}}$ below, that we will need in order to state our assumptions on $U$ and $V$.
  \begin{align}
   \mathcal{G}&=\left\{g\in\mathcal{C}^2((0,\infty),\,\R_+)\, :\, r\mapsto r^{1/2}\,g(r^{1/2})\text{ is non-decreasing and concave,}\quad \lim_{r\downarrow0}r\,g(r)=0\right\}, \\
\hat{\mathcal{G}}&:=\biggl\{g\in\mathcal{G} \ \ \text{ bounded such that } g' \ge 0 \text{ and } 2g{''}+g\,g{'}\leq 0 \biggr\}\,.
\end{align}
Note that every function $\hat{g} \in \hat{\mathcal{G}}$ is concave, because $2g'' \le -gg' \le 0$.
Moreover, we introduce for a given potential $W\in\mathcal{C}^1(\R^d,\R)$ the function $\ell_W$, that quantifies integrated concavity properties of $W$. Its definition is an obvious adaptation of that of $\kappa_{W}$.
\begin{equation*}
\ell_W(r):= \sup\biggl\{\frac{\langle\nabla W(x)-\nabla W(y),\,x-y\rangle}{|x-y|^2}:\quad |x-y|=r\biggr\}\,.
\end{equation*}
Throughout this section, we make the following set of assumptions.

\begin{assumption} \label{ass:conv_prof_rig}
{\ }
\begin{enumerate}[(i)]
\item There exist $\alpha_U\in(0,\infty)$, $\hat g_U\in\hat{\mathcal{G}}$, $\beta_V\in(0,\infty]$, and $g_V\in\mathcal{G}$ such that for all $r>0$
    \begin{equation}\label{A:kappahat}
    \kappa_{U}(r)\geq \alpha_U-r^{-1}\,\hat g _{U}(r)\quad\text{ and }\quad \ell_{V}(r)\leq \beta_V+r^{-1}\,g_V(r)\,.
           \end{equation}
           Moreover, the (concave) function $\hat{g}_U$ admits finite right-derivative at $0$, denoted $\hat{g}_U'(0)$.
\item There exist $\alpha_V\in(0,\infty)$, $\hat g_V\in\hat{\mathcal{G}}$, $\beta_U\in(0,\infty]$, and $g_U\in\mathcal{G}$ such that for all $r>0$
    \begin{equation*} 
    \kappa_{V}(r)\geq \alpha_V-r^{-1}\,\hat g _{V}(r)\quad\text{ and }\quad \ell_{U}(r)\leq \beta_U+r^{-1}\,g_U(r)\,.
    \end{equation*}
    Moreover, the (concave) function $\hat{g}_V$ admits finite right-derivative at $0$, denoted $\hat{g}_V'(0)$.
\end{enumerate}
\end{assumption}
Note that when $\beta_V=\infty$ the assumption on $\ell_V$ in \eqref{A:kappahat} trivializes and is always satisfied (and analogously for $U$ in place of $V$). However, if $\beta_V$ is finite, we will obtain sharper lower bounds on the convexity profile of Schr\"odinger potentials.
At this stage, one might wonder why we state our assumptions in terms of the sets $\mathcal{G}$ and $\hat{\mathcal{G}}$ instead of   the more easily understandable form \eqref{def:as_conv_prof}. To address this, we first note that Assumption \ref{ass:conv_prof_rig} is no less general than \eqref{def:as_conv_prof}:

\begin{proposition}[Prop 5.1 in \cite{conforti2022weak}]
    \label{prop:gen_ass}
Assume that $W$ satisfies \eqref{def:as_conv_prof} for some $\alpha_W>0,L_W,R_W\geq0$.
Then 
\begin{equation*}
\kappa_W(r) \geq \alpha_W-r^{-1}\hat{g}_{\hat{L}}(r) \quad \forall r>0,
\end{equation*}
where for any $L\geq0$ the function $\hat{g}_L\in\hat{\mathcal{G}}$ given by $\hat{g}_L(r) = 2\sqrt{L}\tanh(r\sqrt{L} )$, 
and $\hat{L}$ is given by 
\begin{equation*} 
\hat{L}:=\inf\{L:R^{-1}_W\hat{g}_{L}(R_W)\geq L_W\} \ \text{ if } R_W > 0, \text{ and } \hat{L}=0  \ \text{ if } R_W=0.
\end{equation*}
\end{proposition}

However, the main reason for introducing $\mathcal{G}$ and $\hat{\mathcal{G}}$ is that any element $\hat{g}\in\hat{\mathcal{G}}$ defines a set of functions that is stable under the action of the mapping $h\mapsto -\log S_{\epsilon}e^{-h}$, where $(S_t)_{t\ge 0}$ is the semigroup generated by a standard Brownian motion on $\R^d$, that is to say
\begin{equation}
S_tf(x) = \frac{1}{{(2\pi t)}^{d/2}} \int e^{-f(y) - |y-x|^2/2t } dy. \label{def:BrownianSemigroup}
\end{equation}

\begin{theorem}[Theorem 2.1 in \cite{conforti2022weak}]\label{thm:HJB}
    For any function $\hat{g}\in\hat{\mathcal{G}}$, consider the class
    \begin{equation}\label{def:stable_set}
    \mathcal{F}_{\hat{g}}:=\{h\in\mathcal{C}^1(\R^d)\colon\kappa_h(r)\geq- r^{-1}\,\hat{g}(r)\quad\forall\,r>0\}.
    \end{equation}
    Then we have
    \begin{equation}\label{eq:invariance}
    h\in\mathcal{F}_{\hat{g}}\Rightarrow -\log S_te^{-h}\in\mathcal{F}_{\hat{g}}\quad\forall  t \ge 0.
    \end{equation}
\end{theorem}

Actually, Theorem 2.1 of \cite{conforti2022weak} deals only with the $\hat{g}=\hat{g}_L$ of the form in Proposition \ref{prop:gen_ass}, but we refer to Theorem 14 of \cite{conforti2023convergence} and the subsequent paragraph for an explanation of the straightforward extension to general $\hat{g} \in \hat{\mathcal{G}}$. In the special case $\hat{g} \equiv 0$, \eqref{eq:invariance} simplifies to the well known fact that $-\log S_t e^{-h}$ is convex when $h$ is convex, which can be deduced quickly from the Pr\'ekopa-Leindler inequality. Theorem \ref{thm:HJB} is a significant generalization of this  fact. 

Using the modified potentials from \eqref{def:bar_potentials},
 the Schr\"odinger equations \eqref{eq:Schrsyst} can be equivalently written as 
\begin{equation*}
\bar\varphi = U + \log S_{\epsilon}\exp(-\bar\psi), \qquad     \bar\psi = V+ \log S_{\epsilon}\exp(-\bar\varphi).
\end{equation*}
This explains why Theorem \ref{thm:HJB} is relevant to the goal of finding lower bounds on the integrated convexity profiles $\kappa_{\bar\varphi},\kappa_{\bar\psi}$. This goal is achieved at Theorem \ref{gio2023:thm} below. In turn, observing that $\pi$ rewrites as 
\begin{equation}\label{eq:static_schr_br_dec}
\pi(dx\,dy)\propto \exp(-\bar\varphi(x)-\bar\psi(y) - |x-y|^2/2\epsilon )\,dx\,dy,
\end{equation}
lower bounds on $\kappa_{\bar\varphi},\kappa_{\bar\psi}$ inform us about the log-concavity properties of $\pi$ and its conditional distributions, which is exactly what we need to understand for showing that the conditional distributions of $\pi$ satisfy a logarithmic Sobolev inequality.

Let us now proceed to state the announced lower bounds on the profiles $\kappa_{\varphi},\kappa_{\psi}$. To do so, we need to introduce some more notation. Given $\beta\in(0,\infty]$, $\alpha>-\epsilon^{-1}$, $g\in\mathcal{G}$, and $\hat g\in\hat{\mathcal{G}}$, we define the functions $F_\beta^{g,\hat g}(\alpha,\cdot)$, $ G_\beta^{g,\hat g}(\alpha,\cdot):\R_+\longrightarrow\R_+$ as follows
    \begin{equation}\label{def:F:G}
    \begin{aligned}
    F_\beta^{g,\hat g}(\alpha,t)=&\,\beta \, +\frac{t}{\epsilon(1+\epsilon\alpha)}+t^{1/2}\,g(t^{1/2})+\frac{t^{1/2}\,\hat g(t^{1/2})}{(1+\epsilon\alpha)^2}\,,\\
    G_\beta^{g,\hat g}(\alpha,u)=&\,\inf\{t\geq 0\,\colon\, F_\beta^{g,\hat g}(\alpha,t)\geq u\}\,,
\end{aligned}
\end{equation}  
with the convention $F_\beta^{g,\hat g}(\alpha,t)\equiv\infty, G_\beta^{g,\hat g}(\alpha,u)\equiv0$ whenever $\beta=\infty$. 

\begin{theorem}[Extracted from Theorem 11 in \cite{conforti2023convergence}, see also Thm. 1.1. in \cite{conforti2022weak} ]\label{gio2023:thm}
Consider the Schr\"odinger potentials $\bar\varphi,\bar\psi$ defined at \eqref{def:bar_potentials}.
\begin{enumerate}
   \item Let Assumption \ref{ass:conv_prof_rig} (i) hold. Then we have  for all $r>0$
\begin{equation}\label{kappa:psi0}
\kappa_{\bar\varphi}(r)\geq \alpha_{\bar\varphi}-r^{-1}\hat g _{U}(r) \,,
\end{equation}
where $\alpha_{\bar\varphi}\in(\alpha_U-\epsilon^{-1},\alpha_U-\epsilon^{-1}+(\beta_V\,\epsilon^2)^{-1}]$ is the smallest solution to the fixed point equation
\begin{equation}\label{def:sistem:alpha:psi}
\alpha= \alpha_U-\epsilon^{-1}+(2 \epsilon^2)^{-1}G_{\beta_V}^{g_V,\hat g _{U}}(\alpha,2).
\end{equation}
\item Let Assumption \ref{ass:conv_prof_rig} (ii) hold. Then we have for all $r>0$
\begin{equation}\label{eq_bound_gio_teo:varphi}
\kappa_{\bar\psi}(r)\geq \alpha_{\bar\psi}-r^{-1}\hat g _{V}(r)\,,
\end{equation}
where $\alpha_{\bar\psi}\in(\alpha_V-\epsilon^{-1},\alpha_V-\epsilon^{-1}+(\beta_U\,\epsilon^2)^{-1}]$ is the smallest solution of the fixed point equation
\begin{equation}\label{def:sistem:alpha:varphi}
\alpha= \alpha_V-\epsilon^{-1}+(2 \epsilon^2)^{-1} G_{\beta_U}^{g_U,\hat g _{V}}(\alpha,2).
\end{equation}
\end{enumerate}
\end{theorem}

\subsubsection{Logarithmic Sobolev inequality for conditional distributions via heat flow maps}
With preliminary results from \cite{conforti2022weak,conforti2023convergence} now summarized, we are ready to state and prove the main result of this section:

\begin{theorem}\label{thm:conv_prof_rigorous}
{\ }
    \begin{enumerate}
        \item  If Assumption \ref{ass:conv_prof_rig}(i) holds, then $\pi^y$ satisfies $\lsi{\kappa^{X|Y}}$  uniformly in $y\in \R^d$ with
        \begin{equation}\label{eq:conditional_kappa_1}
        \kappa^{X|Y} = 2\big(\alpha_{\bar\varphi}+\epsilon^{-1}\big)\exp\Big(-\frac{\hat{g}'_U(0)}{\alpha_{\bar\varphi}+\epsilon^{-1}}\Big) .
    \end{equation}
    In particular, $\kappa^{X|Y}$ depends only on $\alpha_U,\hat{g}_U,\beta_V,g_V$. 
        \item  If Assumption \ref{ass:conv_prof_rig}(ii) holds, then $\pi^x$ satisfies $\lsi{\kappa^{Y|X}}$ uniformly in $x\in \R^d$ with
        \begin{equation}\label{eq:conditional_kappa_2}
        \kappa^{Y|X} = 2\big(\alpha_{\bar\psi}+\epsilon^{-1}\big)\exp\Big(-\frac{\hat{g}'_V(0)}{\alpha_{\bar\psi}+\epsilon^{-1}}\Big).
    \end{equation}
     In particular, $\kappa^{Y|X}$ depends only on $\alpha_V,\hat{g}_V,\beta_U,g_U$. 
    \end{enumerate}
   In particular, if 
   \begin{equation}\label{eq:int_conv_prof_suff_cond}
       \epsilon > (\alpha_U\alpha_V)^{-1/2} \exp\Big( \frac{\hat{g}'_U(0)}{2\alpha_U} +\frac{\hat{g}'_V(0)}{2\alpha_V} \Big),
   \end{equation}
   then \eqref{intro:newLSI} holds for some positive constant $r>0$ given by \eqref{intro:newLSI_r}. Furthermore, the exponential $L^2$-convergence \eqref{eq:exp_conv_drift} of the conditional expectations $\mathbb{E}[X_t-Y_t\,|\,X_t]$ and $\mathbb{E}[Y_t-X_t\,|\,Y_t]$ to the gradients of the Schr\"odinger potentials also holds.
   \end{theorem}

Theorem \ref{thm:conv_prof_rigorous} will be deduced quickly using the preliminary results above in combination with the following result  of independent interest, which we state separately to emphasize that it has nothing to do with entropic optimal transport or the SDE studied in this paper. It shows that asymptotic positivity of the integrated convexity profile of the log-density of a probability measure is enough to establish that the measure is a Lipschitz image of a Gaussian. In the following, let $\gamma$ be the standard Gaussian measure on $\R^d$.

\begin{theorem} \label{th:newLSI-kappa}
Let $\alpha > -1$ and $\hat{g} \in \hat{\mathcal{G}}$. Assume $\hat{g}$ admits finite right-derivative at $0$, denoted $\hat{g}'(0)$. Let $\rho \in \P(\R^d)$ satisfy $\rho \ll \gamma$, and assume $\bar{h}:=-\log d\rho/d\gamma$ satisfies $\kappa_{\bar{h}}(r) \ge \alpha - r^{-1} \hat{g}(r)$ for all $r > 0$. Then there exists an $L$-Lipschitz map pushing forward $\gamma$ to $\rho$, where
\begin{equation}
L = (\alpha+1)^{-1/2}\exp\Big(\frac{\hat{g}'(0)}{2(\alpha+1)}\Big).
\end{equation}
In particular, $\rho$ satisfies $\lsi{2L^{-2}}$.
\end{theorem}
\begin{proof}[Proof of Theorem \ref{thm:conv_prof_rigorous}]
Taking $\rho=\pi^y$ in Theorem \ref{th:newLSI-kappa}, notice that
\begin{align*}
    \bar{h}(x) &:=-\log \frac{d\pi^y}{d\gamma}(x) \stackrel{\eqref{eq:static_schr_br_dec}}{=} \bar\varphi(x) + \frac{1}{2\epsilon}|x-y|^2  - \frac12|x|^2 + \log  Z^y, \\ 
    Z^y &:= (2\pi)^{-d/2}\int \exp(-\bar\varphi(\bar{x}) +|\bar{x}-y|^2/2\epsilon)\,dx.
\end{align*} 
It follows from Theorem \ref{gio2023:thm} that, for all $r > 0$,
\begin{equation*}
\kappa_{\bar{h}}(r) \ge \kappa_{\bar{\varphi}}(r) + \epsilon^{-1} - 1 \ge \alpha_{\bar\varphi} + \epsilon^{-1} - 1 - r^{-1}\hat g _{U}(r).
\end{equation*}
Theorem \ref{thm:conv_prof_rigorous}(i) now follows immediately from Theorem \ref{th:newLSI-kappa}, by taking $\hat{g}=\hat{g}_U$ and $\alpha=\alpha_{\bar\varphi} + \epsilon^{-1} - 1$. The proof of Theorem \ref{thm:conv_prof_rigorous}(ii) is analogous. The fact that \eqref{eq:int_conv_prof_suff_cond} is a sufficient condition for the validity of \eqref{intro:newLSI} with constant given by \eqref{intro:newLSI_r} is a direct consequence of Theorem \ref{thm:suff_cond_new_LSI} and the fact that $\alpha_{\bar\varphi}\geq \alpha-\epsilon^{-1},\alpha_{\bar\psi}\geq \alpha_V-\epsilon^{-1}$, which is a consequence Theorem \ref{gio2023:thm}.
\end{proof}

The rest of this section is devoted to the proof of Theorem \ref{th:newLSI-kappa}.
We follow the technique of heat flow maps first introduced by Kim and Milman in \cite{kim2012generalization}: this powerful method combines ideas from time reversal and optimal transport to construct a deterministic map pushing forward the Gaussian distribution onto a target measure, thus allowing to transfer a wealth of properties, including $\mathrm{LSI}$, from the Gaussian onto the target. Roughly speaking, the basic idea is to interpolate between the target distribution and the Gaussian law following the marginal flow of the Ornstein-Uhlenbeck process. Reverting the time and representing the time-reversed flow through the associated velocity field in the sense of optimal transport, one obtains a deterministic map, called the heat flow map, pushing the Gaussian onto the target.
Recent works,  see \cite{mikulincer2021brownian,mikulincer2022lipschitz,fathi2023transportation}, have successfully applied this construction to show that probability measures in a certain class satisfy $\mathrm{LSI}$. However, to the best of our knowledge and understanding, this is the first time when the method is successfully applied to construct Lipschitz maps under the only assumption that the target distribution has an asymptotically positive convexity profile. 

Before proving Theorem \ref{thm:conv_prof_rigorous} we shall need two preparatory results. The first is a technical lemma about Lipschitz transport maps and the second lemma contains the main estimate needed for the proof of the Theorem.

\begin{lemma}[Lemma 1 in \cite{mikulincer2022lipschitz}]\label{lemma:Lipschitz_stability}
Let $(\gamma^N)_{N\geq1}$, $(\rho^N)_{N\geq1}$, $\gamma$, and $\rho$ be probability measures on $\R^d$ such that 
\begin{itemize}
\item For all $N$ there exist a $L_N$-Lipschitz map pushing forward $\gamma^N$ onto $\rho^N$,
\item $L:=\lim\sup_{N\rightarrow\infty}L_N<\infty$, and
\item $\gamma_N$  converges to $\gamma$ and $\rho_N$ converges to $\rho$ in distribution as $N\rightarrow \infty$.
\end{itemize}
Then there exists $L$-Lipschitz map pushing forward $\gamma$ onto $\rho$.
\end{lemma}

\begin{lemma} \label{lemma:log_hess_bound}
Grant the assumptions of Theorem \ref{th:newLSI-kappa}. Let $\sigma_t = 1-e^{-2t}$ for $t \ge 0$.
Consider the SDE
\begin{equation}\label{eq:OU_X}
    dX_t= -X_tdt +\sqrt{2}dB_t, \quad X_0\sim \rho.
\end{equation}
Denote by $p_t$ the density of the law of $X_t$.
    Then, for all $t>0$ and $z\in\R^d$, in semidefinite order
    \begin{equation}\label{eq:log_hess_bound_p}
-\frac{e^{-2t}}{\sigma_t} \mathrm{I} \preceq \nabla^2\log \frac{dp_t}{d\gamma}(z)  \preceq -\frac{e^{-2t} }{1+\sigma_t \alpha}\Big(\alpha-\frac{\hat{g}'(0)}{1+\sigma_t\alpha}\Big)\mathrm{I}
    \end{equation}
\end{lemma}
\begin{proof}
Let $(S_t)_{t\geq0}$ be the semigroup generated by Brownian motion as defined in \eqref{def:BrownianSemigroup}. 
Denote by $(X^{\infty}_t)_{t\geq0}$ the stationary reversible Ornstein-Uhlenbeck process, i.e., the solution of \eqref{eq:OU_X} for the initial condition $X_0\sim\gamma$. We have by construction that, for any fixed $t>0$, the Radon-Nikodym derivative of the law of $(X_s)_{s \in [0,t]}$ on $C([0,T];\R^d)$ with respect to the law of $(X^\infty_s)_{s \in [0,t]}$ is given by $\exp(-\bar{h}(X^{\infty}_0))$. Noting that $dp_t/d\gamma$ is the Radon-Nikodym derivative of the law of $X_t$ with respect to the law of $X^\infty_t$, we deduce
\begin{equation}\label{eq:marg_density_cond_exp}
\begin{split}
 \frac{dp_t}{d\gamma}(z) &= \E[\exp(-\bar{h}(X^{\infty}_0))|X^{\infty}_t=z]\\&
 = \E[\exp(-\bar{h}(X^{\infty}_t))|X^{\infty}_0=z]=S_{\sigma_t}\exp(-\bar{h})(e^{-t}z),
\end{split}
\end{equation}
where to obtain the last identity we used the explicit representation of the Ornstein-Uhlenbeck transition density. At this point, the lower bound in \eqref{eq:log_hess_bound_p} is a consequence of the following known inequality, valid for all $\tau>0$ (see \cite[ Eq. 3.3]{mikulincer2021brownian}\footnote{Note that the paper \cite{mikulincer2021brownian} uses the notation $v(t,x):=\nabla\log P_{1-t}f(x)$ at equation 3.3 } for example)
\begin{equation*}
\nabla^2 \log S_{\tau} e^{-\bar{h}}  \succeq -\frac{1}{\tau}\mathrm{I}.
\end{equation*}
We now proceed to establish the upper bound in \eqref{eq:log_hess_bound_p}. Owing to the notation of Theorem \ref{thm:HJB},  the assumption $\kappa_{\bar{h}}(r) \geq \alpha -r^{-1}\hat{g}_{U}(r)$ for all $r>0$ can be equivalently rewritten as
\begin{equation}\label{eq:h_inv}
h\in\mathcal{F}_{\hat{g}}, \quad \text{where} \quad h(x) =\bar{h}(x)-\frac12 \alpha |x|^2 .
\end{equation}
Rewriting the semigroup in integral form and completing the squares we find, after some tedious though elementary calculations, that 
\begin{equation}\label{eq:tedious}
\log S_{\sigma_t}\exp(-\bar{h})(e^{-t}z)= -\frac{\alpha e^{-2t}}{2(1+\sigma_t \alpha )}|z|^{2} + \log S_{\tau}e^{-h}\Big(\frac{e^{-t}z}{1+\sigma_t \alpha }\Big) + \frac{d}{2} \log \tau/\sigma_t
\end{equation}
with $\tau := \sigma_t/(1+\sigma_t\alpha)$. Note that $\tau\geq0$ since $\alpha>-1$ and therefore $S_{\tau}e^{-h}$ is well-defined.
To conclude, we observe for $r > 0$ and   $\bar{r}:=r e^{-t}/(1+\sigma_t \alpha)$ that
\begin{equation*}
\begin{split}
\kappa_{-\log \frac{dp_t}{d\gamma}}(r) & \   \stackrel{\eqref{eq:marg_density_cond_exp}+\eqref{eq:tedious}}{=}  \ \frac{\alpha e^{-2t}}{1+\sigma_t \alpha} + \frac{e^{-2t}}{(1+\sigma_t \alpha)^2}\Big(\kappa_{-\log S_{\tau}e^{-h}}\Big)\Big(\bar{r}\Big)\\
&\stackrel{\eqref{eq:h_inv}+\text{Thm.} \ref{thm:HJB}}{\geq}\frac{\alpha e^{-2t}}{1+\sigma_t \alpha} -\frac{e^{-2t}}{(1+\sigma_t \alpha)^2} {\bar r}^{-1}\hat{g}(\bar{r}) \\
&\qquad \ \geq \qquad \  \frac{\alpha e^{-2t}}{1+\sigma_t \alpha} - \frac{e^{-2t}}{(1+\sigma_t \alpha)^2}\hat{g}'(0),
\end{split}
\end{equation*}
where the last step used the concavity of $\hat{g}$. Since the last bound holds uniformly on $r>0$ the desired conclusion follows by definition of  $\kappa$ in \eqref{eq:kappa_def}.
\end{proof}

\begin{proof}[Proof of Theorem \ref{th:newLSI-kappa}]
 Let $(X_t)_{t\geq0}$ be the Ornstein-Uhlenbeck process started at $\rho$ as in \eqref{eq:OU_X}. Fix a sequence $\delta_N$ such that $\lim_N\delta_N=0$ and define the time-reversed flow $(\overleftarrow{p}^N_{t})_{t\in[0,N-\delta_N]}=(p_{N-t})_{t\in[0,N-\delta_N]}$. It is well known that 
$(p_t)_{t\in[0,N]}$ is a classical solution over  $(0,N]\times\R^d$ of the Fokker Planck-equation
\begin{equation*}
\partial_t p_{t} -\Delta  p_{t} + \nabla \cdot ((\nabla \log \gamma) p_{t}) =0.
\end{equation*}
But then, $(\overleftarrow{p}^N_{t})_{t\in[0,N-\delta_N]}$ is a classical solution over $[0,N-\delta_N)\times\R^d$ of the continuity equation
\begin{equation*}
\partial_t\overleftarrow{p}^N_{t} + \nabla \cdot (v^N_t \overleftarrow{p}^N_{t})=0, \quad \text{with} \quad v^N_t(\cdot)=\log\frac{d\overleftarrow{p}^N_{t}}{d\gamma}(\cdot).
\end{equation*}
We can invoke item (1) from Lemma \ref{lemma:log_hess_bound} to obtain that for all $N>0$
\begin{equation*}
\sup_{t\in[0,N-\delta_N]} \| D v^N_t\|_{op} <\infty
\end{equation*}
But then, by Picard-Lindel\"of Theorem, for any $x\in\R^d$  there exist a unique global solution for
\begin{equation*}
\frac{d}{dt}\Phi^N(t,x) = v^N_t(\Phi^N(t,x)), \qquad
\Phi^N(0,x)=x.
\end{equation*}
Recognizing these as the characteristic curves of the above continuity equation, it holds that $\Phi^N(N-\delta_N,\cdot)$ pushes forward $p_N$ onto $p_{\delta_n}$.
Moreover, invoking again item (1) from Lemma \ref{lemma:log_hess_bound} provides the one sided Lipschitz estimate
\begin{equation}\label{eq:one_sided}
\langle v^N_t(x)-v^N_t(\hat{x}),x-\hat{x} \rangle \leq  -\frac{e^{-2(N-t)} }{1+\sigma_{N-t} \alpha}\Big(\alpha-\frac{\hat{g}'(0)}{1+\sigma_{N-t}\alpha}\Big) |x-\hat{x}|^2,
\end{equation}
for all $x,\hat{x}\in\R^d$ and $t \in [0,N-\delta_N]$.
But then, we have for all such $(x,\hat{x},t)$ that
\begin{equation*}
    \frac{d}{dt}|\Phi^N(t,x)-\Phi^N(t,\hat{x})|^2\leq -\frac{2e^{-2(N-t)} }{1+\sigma_{N-t} \alpha}\Big(\alpha-\frac{\hat{g}'(0)}{1+\sigma_{N-t}\alpha}\Big)
  |\Phi^N(t,x)-\Phi^N(t,\hat{x})|^2.
\end{equation*}
An application of Gronwall's Lemma then shows that $\Phi^N(N-\delta_N,\cdot)$ is $L_N$-Lipschitz with 
\begin{equation*}
L_N:=\exp\Big( -\int_{\delta_N}^N \frac{e^{-2t} }{1+\sigma_{t} \alpha}\Big(\alpha-\frac{\hat{g}'(0)}{1+\sigma_{t}\alpha}\Big) dt \Big).
\end{equation*}
We will compute $L:=\lim_{N\to\infty}L_N$ is below.
Summarizing the previous discussion, we have
\begin{itemize}
\item $\Phi^N(N-\delta_N,\cdot)$ pushes forward $p_N$ onto $p_{\delta_N}$
\item $\Phi^N(N-\delta_N,\cdot)$ is $L_N$-Lipschitz.  with $\lim_{N\rightarrow + \infty } L_N  = L$.
\item $p_N$  converges to $\gamma$ and $p_{\delta_n}$ converges to $\rho$ in distribution as $N\rightarrow \infty$. 
\end{itemize}
Applying Lemma \ref{lemma:Lipschitz_stability}, we deduce that there exist a $L$-Lipschitz map pushing forward $\gamma$ onto $\rho$.
To compute $L$, we define the function $\theta(t)=1+\sigma_t \alpha=1+\alpha(1-e^{-2t})$ and observe that
\begin{align*}
L &= \exp\bigg( -\int_0^\infty \frac{e^{-2t} }{1+\sigma_{t} \alpha}\Big(\alpha-\frac{\hat{g}'(0)}{1+\sigma_{t}\alpha}\Big) dt \bigg) \\
	&=\exp\bigg(-\frac12\int_{0}^{\infty} \Big(\frac{ \theta'(t)}{\theta(t)} - \frac{\hat{g}'(0)}{\alpha}\frac{\theta'(t)}{\theta^2(t)}\Big)dt\bigg) \\
	&= \bigg(\frac{\theta(\infty)}{\theta(0)}\bigg)^{-1/2}\exp\Bigg(   \frac{\hat{g}'(0)}{2\alpha}\Big(\frac{1}{\theta(0)}-\frac{1}{\theta(\infty)}\Big)\Bigg)\\
&=(\alpha+1)^{-1/2}\exp\Big( \frac{\hat{g}'(0)}{2(\alpha+1)}\Big). 
\end{align*}
This completes the proof. To prove the ``in particular" claim, note that $\gamma$ satisfies $\lsi{2}$, and a well known and straightforward argument shows that the $L$-Lipschitz pushforward $\rho$ satisfies $\lsi{2L^{-2}}$.
\end{proof}

\appendix

\section{Differentiability of potentials} \label{ap:differentiability}

Throughout this section, we impose the same assumptions on $(c,U,V)$ as in Theorem \ref{th:main}.
Here we will explain some technical details on the differentiation of the Schr\"odinger potentials, justifying Proposition \ref{pr:derivatives-potentials}. Recall the formula for $\pi$,
\begin{align}
\pi(x,y) = \exp\big((\varphi(x)+\psi(y) - c(x,y))/\epsilon \big)\mu(x)\nu(y),  \label{ap:def:pi}
\end{align}
with the potentials  satisfying
\begin{align*}
\exp(-\varphi(x)/\epsilon) &= \int_{\R^d} \exp\big( ( \psi(y) - c(x,y))/\epsilon \big) \nu(dy), \ \ \text{for all } x, \\
\exp(-\psi(y)/\epsilon) &= \int_{\R^d} \exp\big( (\varphi(x) - c(x,y))/\epsilon\big) \mu(dx), \ \ \text{for all } y.
\end{align*}
Recall that $(\varphi,\psi)$ are uniquely defined only up to an additive constant. We may choose the additive constant so that $\int \varphi\,d\mu=0$. Then, taking logs in \eqref{ap:def:pi} and integrating with respect to $\pi \in \Pi(\mu,\nu)$, we have
\begin{align*}
\int \varphi\,d\mu + \int \psi\,d\nu &= \epsilon \int \log\frac{d\pi}{d\mu\otimes\nu}\,d\pi + \int c\,d\pi \ge 0,
\end{align*}
because $c \ge 0$.
We thus obtain $\int\psi\,d\nu \ge 0$.
We start with a lemma on integrability.

\begin{lemma} \label{ap:le:integrability}
For any bounded set $B \subset\R^d$ and any $p > 0$, we have
\begin{align*}
\int_{\R^d} \sup_{x \in B}\exp\big( p( \psi(y) - c(x,y))/\epsilon \big)\,\nu(dy) < \infty.
\end{align*}
\end{lemma}
\begin{proof}
First, recall from above that we have chosen the normalization that makes $\int\varphi\,d\mu=0$ and $\int\psi\,d\nu\ge 0$. By Jensen's inequality, this yields
\begin{align*}
\psi(y) &= -\epsilon \log \int_{\R^d} \exp\big( (\varphi(x) - c(x,y))/\epsilon\big) \mu(dx) \\
	&\le \int_{\R^d}  (-\varphi(x) + c(x,y) ) \mu(dx) \\
	&\le \int_{\R^d} c(x,y) \,\mu(dx).
\end{align*}
Using this and the boundedness of $\nabla^2c$ in operator norm by some constant $R$, we have
\begin{align*}
 \psi(y) - c(x,y) &\le \int_{\R^d} \big(c(x',y) - c(x,y)\big) \,\mu(dx') \\
 	&\le \int_{\R^d} \Big((x'-x)\cdot\nabla_xc(x,y) + \frac12R|x'-x|^2\Big) \,\mu(dx') \\
 	&\le \frac12(R+\delta^{-1})\int_{\R^d} |x'-x|^2\,\mu(dx') + \frac12\delta|\nabla_xc(x,y)|^2,
\end{align*}
for any $\delta > 0$. Use boundedness of $\nabla^2c$ once again to find a constant $R'$ such that $|\nabla_xc(x,y)|^2 \le R'(1+|x|^2+|y|^2)$. Then
\begin{align*}
\exp\big( ( \psi(y) - c(x,y))/\epsilon \big) &\le g(x)  \exp\Big(\frac{\delta R'}{2\epsilon} |y|^2\Big),  \\
\text{where } \quad g(x) &:= \exp\bigg(\frac{R+\delta^{-1}}{2\epsilon}\int_{\R^d}|x'-x|^2\,\mu(dx') + \frac{\delta R'}{2}(1+|x|^2)\bigg).  
\end{align*}
Because $\nu$ is subgaussian, for any $p > 0$ we may choose $\delta$ small enough to guarantee that $\int_{\R^d}\exp(p\delta R' |y|^2/ 2\epsilon)\,\nu(dy) < \infty$. Because $\mu$ has finite second moment, the function $g$ is continuous and thus locally bounded.
\end{proof}

We now turn to differentiability. 
Recall the definition of weak derivatives: For locally integrable functions $f : \R^d \to \R$ and $g:\R^d\to\R^d$, we say $g=\nabla f$ in the weak sense if
\begin{equation*}
\int_{\R^d} f \,\mathrm{div} (h) = - \int_{\R^d} h \cdot \nabla g, \quad \forall h \in C^\infty_c(\R^d;\R^d).
\end{equation*}
Here $C^\infty_c(\R^d;\R^d)$ is the space of smooth compactly supported functions from $\R^d\to\R^d$.
Differentiating under the integral sign is quite a simple matter when working with weak derivatives. Indeed, if a weakly differentiable function $f : \R^d \times \R^d \to \R$ is such that 
\begin{equation}
1_B(x)f(x,y), \ 1_B(x)|\nabla_xf(x,y)| \in L^1(dx\nu(dy)), \ \text{ for every open ball } B \subset \R^d, \label{ap:weakdiffcondition}
\end{equation}
then it is  an easy consequence of Fubini's theorem that (in the sense of weak derivatives)
\begin{align*}
\nabla_x \int_{\R^d} f(x,y)\,\nu(dy) = \int_{\R^d} \nabla_x f(x,y)\,\nu(dy) .
\end{align*}

We record here a well known \emph{chain rule} for weak derivatives, and an easy extension:

\begin{lemma} \label{le:chainrule}
Suppose $f > 0$ is weakly differentiable and that $|\nabla f|/f$ and $\log f$ are locally integrable. Then $\log f$ is weakly differentiable, and $\nabla \log f = \nabla f/f$.
\end{lemma}
\begin{proof}
Let $g_n : \R \to \R$ be smooth functions such that $\|g_n'\|_\infty < \infty$ for each $n$, $g_n(x) \downarrow \log x$ as $n\to\infty$, and $g_n(x) = \log(x)$ for $x \ge 1/n$. Using the well known chain rule for Lipschitz functions of weak derivatives \cite[Section 4.2.2]{EvansGariepy}, $g_n ( f)$ is weakly differentiable with $\nabla g_n (f) = g_n'(f)\nabla f$. Then, for any $h \in C^\infty_c(\R^d;\R^d)$, dominated convergence yields
\begin{align*}
\int \mathrm{div} (h ) \log f &= \lim_n \int \mathrm{div} (h ) g_n(f) = -\lim_n 
\int g_n'(f) h \cdot \nabla f = -\int h \cdot \frac{\nabla f}{f}.  \qedhere
\end{align*}
\end{proof}

\begin{proof}[Proof of Proposition \ref{pr:derivatives-potentials}]
We prove the claims only for $\varphi$, as $\psi$ is handled analogously.
We start by justifying the claimed formula for the first derivative, which was
\begin{align}
\nabla\varphi(x) = \E_\pi[\nabla_xc(X,Y)\,|\,X=x]. \label{ap:pf:claim0}
\end{align}
Note that the right-hand side is in $L^1(\mu)$ and thus in $L^1_{\mathrm{loc}}(\R^d)$, because $\mu$ is bounded away from zero on compact sets.
By the chain rule in the form of Lemma \ref{le:chainrule}, it suffices to show that
\begin{align}
\nabla_x e^{-\varphi(x)/\epsilon} = -\frac{1}{\epsilon}\int_{\R^d} \nabla_xc(x,y) \exp\big((\psi(y)-c(x,y))/\epsilon\big)\,\nu(dy). \label{ap:pf:claim1}
\end{align}
Indeed, then the calculation
\begin{align*}
\nabla \varphi(x) &= -\epsilon \frac{\nabla_x e^{-\varphi(x)/\epsilon}}{e^{-\varphi(x)/\epsilon}} = \frac{\int_{\R^d} \nabla_xc(x,y) \exp\big((\psi(y)-c(x,y))/\epsilon\big)\,\nu(dy)}{\int_{\R^d}  \exp\big((\psi(y)-c(x,y))/\epsilon\big)\,\nu(dy)} \\
	&= \frac{\int_{\R^d} \nabla_xc(x,y) \pi(x,y)\,dy}{\int_{\R^d} \pi(x,y)\,dy}
\end{align*}
completes the proof of \eqref{ap:pf:claim0}. To prove \eqref{ap:pf:claim1}, we must justify the interchange of weak derivative and integral, by checking that the function
\begin{align*}
f(x,y) &= \exp\big( ( \psi(y) - c(x,y))/\epsilon \big)
\end{align*}
satisfies the condition \eqref{ap:weakdiffcondition}. But this follows easily from Lemma \ref{ap:le:integrability}, after recalling that $|\nabla_xc(x,y)|$ has linear growth and $\nu$ has finite moments of every order.

We finally show the formula for the second derivatives,
\begin{align*}
\nabla^2\varphi(x) = \E_\pi[\nabla_{xx}^2 c(X,Y)\,|\,X=x] - \epsilon^{-1}\Cov_\pi(\nabla_x c(X,Y)\,|\,X=x).
\end{align*}
Indeed, the details of the chain rule and condition \eqref{ap:weakdiffcondition} are checked similary as in the proof of \eqref{ap:pf:claim0}, and we may differentiate a second time under the integral sign to get
\begin{align*}
\nabla_{xx}^2 e^{-\varphi(x)/\epsilon} &= -\frac{1}{\epsilon}\int_{\R^d} \nabla_{xx}^2c(x,y) \exp\big((\psi(y)-c(x,y))/\epsilon\big)\,\nu(dy) \\
	&\qquad  + \frac{1}{\epsilon^2}\int_{\R^d} \nabla_xc(x,y)^{\otimes 2} \exp\big((\psi(y)-c(x,y))/\epsilon\big)\,\nu(dy).
\end{align*}
Then, using the identity $\nabla^2\log h =  \nabla^2 h/h - (\nabla \log h)^{\otimes 2}$, we have
\begin{align*}
\nabla_{xx}^2\varphi(x) &= -\epsilon\bigg( \frac{\nabla_{xx}^2 e^{-\varphi(x)/\epsilon}}{e^{-\varphi(x)/\epsilon}} - \epsilon^{-2}(\nabla_x\varphi(x))^{\otimes 2}\bigg) \\
	&= \E_\pi[\nabla_{xx}^2 c(X,Y)\,|\,X=x] - \epsilon^{-1}\E_\pi[(\nabla_x c(X,Y))^{\otimes 2}\,|\,X=x] \\
	&\qquad + \epsilon^{-1}\E_\pi[\nabla_x c(X,Y)\,|\,X=x]^{\otimes 2}. \qedhere
\end{align*}
\end{proof}

\bibliographystyle{amsplain}
\bibliography{EOTbib}

\providecommand{\bysame}{\leavevmode\hbox to3em{\hrulefill}\thinspace}
\providecommand{\MR}{\relax\ifhmode\unskip\space\fi MR }
\providecommand{\MRhref}[2]{%
  \href{http://www.ams.org/mathscinet-getitem?mr=#1}{#2}
}
\providecommand{\href}[2]{#2}
\begin{thebibliography}{10}

\bibitem{AguehCarlier}
M.~Agueh and G.~Carlier, \emph{Barycenters in the {W}asserstein space}, SIAM
  Journal on Mathematical Analysis \textbf{43} (2011), no.~2, 904--924.

\bibitem{AGSbook}
L.~Ambrosio, N.~Gigli, and G.~Savar{\'e}, \emph{Gradient flows: in metric
  spaces and in the space of probability measures}, Springer Science \&
  Business Media, 2005.

\bibitem{aronson1968non}
D.G. Aronson, \emph{Non-negative solutions of linear parabolic equations},
  Annali della Scuola Normale Superiore di Pisa-Classe di Scienze \textbf{22}
  (1968), no.~4, 607--694.

\bibitem{BenamouBrenier}
J.-D. Benamou and Y.~Brenier, \emph{A computational fluid mechanics solution to
  the {M}onge-{K}antorovich mass transfer problem}, Numerische Mathematik
  \textbf{84} (2000), no.~3, 375--393.

\bibitem{bhatia2019bures}
R.~Bhatia, T.~Jain, and Y.~Lim, \emph{On the {B}ures--{W}asserstein distance
  between positive definite matrices}, Expositiones Mathematicae \textbf{37}
  (2019), no.~2, 165--191.

\bibitem{BKRSbook}
V.I. Bogachev, N.V. Krylov, M.~R{\"o}ckner, and S.V. Shaposhnikov,
  \emph{Fokker--{P}lanck--{K}olmogorov equations}, vol. 207, American
  Mathematical Society, 2022.

\bibitem{BogRocSha}
V.I. Bogachev, M.~R{\"o}ckner, and S.V. Shaposhnikov, \emph{Distances between
  transition probabilities of diffusions and applications to nonlinear
  {F}okker--{P}lanck--{K}olmogorov equations}, Journal of Functional Analysis
  \textbf{271} (2016), no.~5, 1262--1300.

\bibitem{bolley2005weighted}
F.~Bolley and C.~Villani, \emph{Weighted {C}sisz{\'a}r-{K}ullback-{P}insker
  inequalities and applications to transportation inequalities}, Annales de la
  Facult{\'e} des sciences de Toulouse: Math{\'e}matiques, vol.~14, 2005,
  pp.~331--352.

\bibitem{bossy2018wellposedness}
M.~Bossy and J.-F. Jabir, \emph{On the wellposedness of some {M}c{K}ean models
  with moderated or singular diffusion coefficient}, arXiv preprint
  arXiv:1809.01742 (2018).

\bibitem{bossy2011conditional}
M.~Bossy, J.-F. Jabir, and D.~Talay, \emph{On conditional {M}c{K}ean
  {L}agrangian stochastic models}, Probability theory and related fields
  \textbf{151} (2011), no.~1-2, 319--351.

\bibitem{brenier1991polar}
Y.~Brenier, \emph{Polar factorization and monotone rearrangement of
  vector-valued functions}, Communications on pure and applied mathematics
  \textbf{44} (1991), no.~4, 375--417.

\bibitem{BrunickShreve}
G.~Brunick and S.~Shreve, \emph{Mimicking an {I}t{\^o} process by a solution of
  a stochastic differential equation}, The Annals of Applied Probability
  \textbf{23} (2013), no.~4, 1584--1628.

\bibitem{caglioti2009constrained}
E.~Caglioti, M.~Pulvirenti, and F.~Rousset, \emph{On a constrained 2-{D}
  {N}avier-{S}tokes equation}, Communications in Mathematical Physics
  \textbf{290} (2009), 651--677.

\bibitem{CarlenGangbo}
E.A. Carlen and W.~Gangbo, \emph{Constrained steepest descent in the
  2-{W}asserstein metric}, Annals of mathematics (2003), 807--846.

\bibitem{carlier2020differential}
G.~Carlier and M.~Laborde, \emph{A differential approach to the multi-marginal
  {S}chr\"{o}dinger system}, SIAM Journal on Mathematical Analysis \textbf{52}
  (2020), no.~1, 709--717.

\bibitem{carrillo2003kinetic}
J.A. Carrillo, R.J. McCann, and C.~Villani, \emph{Kinetic equilibration rates
  for granular media and related equations: entropy dissipation and mass
  transportation estimates}, Revista Matematica Iberoamericana \textbf{19}
  (2003), no.~3, 971--1018.

\bibitem{chen2021stochastic}
Y.~Chen, T.T. Georgiou, and M.~Pavon, \emph{Stochastic {C}ontrol {L}iaisons:
  {R}ichard {S}inkhorn {M}eets {G}aspard {M}onge on a {S}chrodinger {B}ridge},
  SIAM Review \textbf{63} (2021), no.~2, 249--313.

\bibitem{chewi2022entropic}
S.~Chewi and A.-A. Pooladian, \emph{An entropic generalization of
  {C}affarelli's contraction theorem via covariance inequalities}, arXiv
  preprint arXiv:2203.04954 (2022).

\bibitem{conforti2022weak}
G.~Conforti, \emph{Weak semiconvexity estimates for {S}chr{\"o}dinger
  potentials and logarithmic {S}obolev inequality for {S}chr{\"o}dinger
  bridges}, arXiv preprint arXiv:2301.00083 (2022).

\bibitem{conforti2023convergence}
G.~Conforti, A.~Durmus, and G.~Greco, \emph{Quantitative contraction rates for
  {S}inkhorn algorithm: beyond bounded costs and compact marginals}, 2023.

\bibitem{CrimaldiPratelli}
I.~Crimaldi and L.~Pratelli, \emph{Two inequalities for conditional
  expectations and convergence results for filters}, Statistics \& probability
  letters \textbf{74} (2005), no.~2, 151--162.

\bibitem{csiszar1975divergence}
I.~Csisz{\'a}r, \emph{I-divergence geometry of probability distributions and
  minimization problems}, The annals of probability (1975), 146--158.

\bibitem{cuturi2013sinkhorn}
M.~Cuturi, \emph{Sinkhorn distances: {L}ightspeed computation of optimal
  transport}, Advances in neural information processing systems \textbf{26}
  (2013).

\bibitem{deb2023wasserstein}
N.~Deb, Y.-H. Kim, S.~Pal, and G.~Schiebinger, \emph{Wasserstein mirror
  gradient flow as the limit of the {S}inkhorn algorithm}, arXiv preprint
  arXiv:2307.16421 (2023).

\bibitem{DiaGangoli}
P.~Diaconis and A.~Gangolli, \emph{Rectangular arrays with fixed margins},
  Discrete Probability and Algorithms (New York, NY) (David Aldous, Persi
  Diaconis, Joel Spencer, and J.~Michael Steele, eds.), Springer New York,
  1995, pp.~15--41.

\bibitem{DiaSturm}
P.~Diaconis and B.~Sturmfels, \emph{{Algebraic algorithms for sampling from
  conditional distributions}}, The Annals of Statistics \textbf{26} (1998),
  no.~1, 363 -- 397.

\bibitem{djete2022non}
M.F. Djete, \emph{Non--regular {M}c{K}ean--{V}lasov equations and calibration
  problem in local stochastic volatility models}, arXiv preprint
  arXiv:2208.09986 (2022).

\bibitem{eberle2016reflection}
A.~Eberle, \emph{Reflection couplings and contraction rates for diffusions},
  Probability {T}heory and {R}elated {F}ields \textbf{166} (2016), no.~3-4,
  851--886.

\bibitem{eberle2017gradient}
S.~Eberle, B.~Niethammer, and A.~Schlichting, \emph{Gradient flow formulation
  and longtime behaviour of a constrained {F}okker--{P}lanck equation},
  Nonlinear Analysis \textbf{158} (2017), 142--167.

\bibitem{eldan2020stability}
R.~Eldan, J.~Lehec, and Y.~Shenfeld, \emph{Stability of the logarithmic
  {S}obolev inequality via the {F}ollmer process}, Annales De L Institut Henri
  Poincare-Probabilites Et Statistiques, vol.~56, 2020, pp.~2253--2269.

\bibitem{EvansPDE}
L.C. Evans, \emph{Partial differential equations}, vol.~19, American
  Mathematical Society, 2022.

\bibitem{EvansGariepy}
L.C. Evans and R.F. Gariepy, \emph{Measure theory and fine properties of
  functions}, Routledge, 2018.

\bibitem{fathi2020proof}
M.~Fathi, N.~Gozlan, and Maxime M.~Prod’homme, \emph{A proof of the
  {C}affarelli contraction theorem via entropic regularization}, Calculus of
  Variations and Partial Differential Equations \textbf{59} (2020), 1--18.

\bibitem{fathi2016quantitative}
M.~Fathi, E.~Indrei, and M.~Ledoux, \emph{Quantitative logarithmic {S}obolev
  inequalities and stability estimates}, Discrete and Continuous Dynamical
  Systems \textbf{36} (2016), no.~12, 6835--6853.

\bibitem{fathi2023transportation}
M.~Fathi, D.~Mikulincer, and Y.~Shenfeld, \emph{Transportation onto
  log-lipschitz perturbations}, 2023.

\bibitem{genevay2019sample}
A.~Genevay, L.~Chizat, F.~Bach, M.~Cuturi, and G.~Peyr{\'e}, \emph{Sample
  complexity of sinkhorn divergences}, The 22nd international conference on
  artificial intelligence and statistics, PMLR, 2019, pp.~1574--1583.

\bibitem{gentil2020entropic}
I.~Gentil, C.~L{\'e}onard, L.~Ripani, and L.~Tamanini, \emph{An entropic
  interpolation proof of the {HWI} inequality}, Stochastic Processes and their
  Applications \textbf{130} (2020), no.~2, 907--923.

\bibitem{jordan1998variational}
R.~Jordan, D.~Kinderlehrer, and F.~Otto, \emph{The variational formulation of
  the {F}okker--{P}lanck equation}, SIAM journal on mathematical analysis
  \textbf{29} (1998), no.~1, 1--17.

\bibitem{JouLelRou}
B.~Jourdain, T.~Leli\`evre, and R.~Roux, \emph{Existence, uniqueness and
  convergence of a particle approximation for the adaptive biasing force
  process}, ESAIM: Mathematical Modelling and Numerical Analysis \textbf{44}
  (2010), no.~5, 831--865.

\bibitem{JourdainZhou}
B.~Jourdain and A.~Zhou, \emph{Existence of a calibrated regime switching local
  volatility model}, Mathematical Finance \textbf{30} (2020), no.~2, 501--546.

\bibitem{KaratzasShreve}
I.~Karatzas and S.~Shreve, \emph{Brownian motion and stochastic calculus}, vol.
  113, Springer Science \& Business Media, 2012.

\bibitem{kim2012generalization}
Y.-H. Kim and E.~Milman, \emph{A generalization of {C}affarelli’s contraction
  theorem via (reverse) heat flow}, Mathematische Annalen \textbf{354} (2012),
  no.~3, 827--862.

\bibitem{lacker2021hierarchies}
D.~Lacker, \emph{Hierarchies, entropy, and quantitative propagation of chaos
  for mean field diffusions}, Probability and Mathematical Physics \textbf{4}
  (2023), no.~2, 377--432.

\bibitem{LacShkZha}
D.~Lacker, M.~Shkolnikov, and J.~Zhang, \emph{Inverting the {M}arkovian
  projection, with an application to local stochastic volatility models}, The
  Annals of Probability \textbf{48} (2020), no.~5, 2189--2211.

\bibitem{LackerZhang}
D.~Lacker and J.~Zhang, \emph{Stationary solutions and local equations for
  interacting diffusions on regular trees}, Electronic Journal of Probability
  \textbf{28} (2023), 1--37.

\bibitem{lambert2022variational}
M.~Lambert, S.~Chewi, F.~Bach, S.~Bonnabel, and P.~Rigollet, \emph{Variational
  inference via {W}asserstein gradient flows}, Advances in Neural Information
  Processing Systems \textbf{35} (2022), 14434--14447.

\bibitem{LelRouSto}
T.~Leli{\`e}vre, M.~Rousset, and G.~Stoltz, \emph{Long-time convergence of an
  adaptive biasing force method}, Nonlinearity \textbf{21} (2008), no.~6, 1155.

\bibitem{LeoSch}
C.~L{\'e}onard, \emph{A survey of the {S}chrödinger problem and some of its
  connections with optimal transport}, Discrete and Continuous Dynamical
  Systems \textbf{34} (2014), no.~4, 1533--1574.

\bibitem{lieberman1996second}
G.M. Lieberman, \emph{Second order parabolic differential equations}, World
  scientific, 1996.

\bibitem{LiptserShiryaev}
R.S. Liptser and A.N. Shiriaev, \emph{Statistics of random processes: {G}eneral
  theory}, vol. 394, Springer, 1977.

\bibitem{mccann1997convexity}
R.J. McCann, \emph{A convexity principle for interacting gases}, Advances in
  mathematics \textbf{128} (1997), no.~1, 153--179.

\bibitem{mena2019statistical}
G.~Mena and J.~Niles-Weed, \emph{Statistical bounds for entropic optimal
  transport: sample complexity and the central limit theorem}, Advances in
  neural information processing systems \textbf{32} (2019).

\bibitem{mikulincer2021brownian}
D.~Mikulincer and Y.~Shenfeld, \emph{The {B}rownian transport map}, 2021.

\bibitem{mikulincer2022lipschitz}
\bysame, \emph{On the {L}ipschitz properties of transportation along heat
  flows}, arXiv preprint arXiv:2201.01382 (2022).

\bibitem{nutz2021introduction}
M.~Nutz, \emph{Introduction to entropic optimal transport}, 2021.

\bibitem{otto2001geometry}
F.~Otto, \emph{The geometry of dissipative evolution equations: the porous
  medium equation},  (2001).

\bibitem{otto2000generalization}
F.~Otto and C.~Villani, \emph{Generalization of an inequality by {T}alagrand
  and links with the logarithmic {S}obolev inequality}, Journal of Functional
  Analysis \textbf{173} (2000), no.~2, 361--400.

\bibitem{pass2015multi}
B.~Pass, \emph{Multi-marginal optimal transport: theory and applications},
  ESAIM: Mathematical Modelling and Numerical Analysis-Mod{\'e}lisation
  Math{\'e}matique et Analyse Num{\'e}rique \textbf{49} (2015), no.~6,
  1771--1790.

\bibitem{PeyreCuturi}
G.~Peyr{\'e} and M.~Cuturi, \emph{Computational optimal transport: With
  applications to data science}, Foundations and Trends{\textregistered} in
  Machine Learning \textbf{11} (2019), no.~5-6, 355--607.

\bibitem{NilesWeed-Pooladian}
A.-A. Pooladian and J.~Niles-Weed, \emph{Entropic estimation of optimal
  transport maps}, arXiv preprint arXiv:2109.12004 (2021).

\bibitem{RogersWilliams}
L.C.G. Rogers and D.~Williams, \emph{Diffusions, {M}arkov processes and
  martingales: {V}olume 2, {I}t{\^o} calculus}, vol.~2, Cambridge university
  press, 2000.

\bibitem{Schr32}
E.~Schr\"odinger, \emph{La th\'eorie relativiste de l'\'electron et l'
  interpr\'etation de la m\'ecanique quantique}, Ann. Inst Henri Poincar\'e
  (1932), no.~2, 269 -- 310.

\bibitem{StroockVaradhan}
D.W. Stroock and S.R.S. Varadhan, \emph{Multidimensional diffusion processes},
  vol. 233, Springer Science \& Business Media, 1997.

\end{thebibliography}

\end{document}